\documentclass[11pt]{amsart}
\usepackage{appendix}
\usepackage{graphicx}
\usepackage{chngpage}
\usepackage{multirow}
\usepackage{hyperref}
\usepackage[all,cmtip]{xy}
\usepackage{amssymb}
\usepackage{amsmath}
\usepackage{longtable}

\usepackage{pstricks}
\usepackage{color}
\theoremstyle{plain}
\newtheorem{thm}{Theorem}[section]
\newtheorem{theorem}[thm]{Theorem}

\newtheorem{lemma}[thm]{Lemma}

\newtheorem{proposition}[thm]{Proposition}
\theoremstyle{definition}
\newtheorem{remark}[thm]{Remark}
\newtheorem{remarks}[thm]{Remarks}

\newtheorem{notation}[thm]{Notation}
\newtheorem{definition}[thm]{Definition}

\newtheorem{example}[thm]{Example}

\numberwithin{equation}{section}

\newcommand{\wt}{\widetilde}
\newcommand{\ti}{\times}

\newcommand{\rt}{\rtimes}

\newcommand{\xfi}{{\xi_5}}

\newcommand{\aod}{{A_1:=}}
\newcommand{\atd}{{A_2:=}}
\newcommand{\ahd}{{A_3:=}}

\newcommand{\la}{{\langle}}
\newcommand{\ra}{{\rangle}}

\newcommand{\vp}{\varphi}
\newcommand{\s}{\sigma}

\newcommand{\Hom}{{\rm Hom}}
\newcommand{\Res}{{\rm Res}}
\newcommand{\Tra}{{\rm Tra}}
\newcommand{\Aut}{{\rm Aut}}
\newcommand{\Ker}{{\rm Ker}}

\newcommand{\Diag}{{\rm diag}}
\newcommand{\Rank}{{\rm rank}}
\newcommand{\Gcd}{{\rm gcd}}
\newcommand{\Mod}{{\rm mod\; }}
\newcommand{\Ord}{{\rm ord}}
\newcommand{\SmallGroup}{{\rm SmallGroup}}

\newcommand{\GL}{{\rm GL}}
\newcommand{\PGL}{{\rm PGL}}
\newcommand{\PSL}{{\rm PSL}}
\newcommand{\Det}{{\rm det}}
\newcommand{\SL}{{\rm SL}}
\newcommand{\Max}{{\rm max}}
\newcommand{\Min}{{\rm min}}
\newcommand{\lmd}{{\lambda}}
\newcommand{\lra}{\longrightarrow}

\newcommand{\C}{{\mathbb C}}

\renewcommand{\P}{{\mathbb P}}

 \title{Automorphism groups of smooth quintic threefolds}

 \author{Keiji Oguiso and Xun Yu} 
\address{Department of Mathematics, Osaka University, Toyonaka 560-0043, Osaka, Japan and Korea Institute for Advanced Study, Hoegiro 87, Seoul, 
133-722, Korea}
\email{oguiso@math.sci.osaka-u.ac.jp}
\address{Center for Geometry and its Applications, POSTECH, Pohang, 790-784, Korea}
\email{yxn100135@postech.ac.kr}

\thanks{The first author is supported by JSPS Grant-in-Aid (S) No 25220701, JSPS Grant-in-Aid (S) No 22224001, JSPS Grant-in-Aid (B) No 22340009, and by KIAS Scholar Program.}

\dedicatory{Dedicated to Professor  Shigeru Mukai on the occasion of his sixtieth birthday.}

\date{May 4, 2015}

\begin{document}
\begin{abstract}  We study automorphism groups of smooth quintic threefolds. Especially, we describe all the maximal ones with explicit examples of target quintic threefolds. There are exactly $22$ such groups.

\end{abstract}

\maketitle

\tableofcontents

\section{Introduction}\label{Intro}

Throughout this paper, we work over the complex number field $\C$. 

The aim of this paper, is to study the automorphism group ${\rm Aut}\, (X)$ of a smooth quintic threefold $X$, a most basic example of Calabi-Yau threefolds. Our main results are Theorems \ref{thm:Main} and \ref{thm:goren}.  It turns out that there are exactly $22$ maximal groups which act faithfully on some smooth quintic threefolds.  This answers a question raised by \cite{LOP13}.

From now, we just call a smooth quintic threefold, simplly a QCY3 (quintic Calbai-Yau threefold). 

Let $X$ be a QCY3 defined by a homogeneous polynomial $F$ of degree $5$.

By a result of Matsumura-Monsky (\cite{MM63}), 
$${\rm Aut}\, (X) = \{\varphi \in \PGL(5, \C) | \varphi(X) = X\}\,\, ,$$
 ${\rm Aut}\, (X)$ is always a {\it finite} group, and two QCY3s are isomorphic if and only if they are projective linearly isomorphic, i.e., by a suitable change of homogeneous coordinates, the defining equations  are the same. Moreover, ${\rm Bir}\, (X) = {\rm Aut}\, (X)$ as $X$ is a projective minimal model of Picard number one (\cite{Ka08}). So, the classification of all possible groups of birational automorphisms of QCY3 is equivalent to the following {\it projective} linear algebra problem in the classical invariant theory:

Find all finite subgroups $G \subset \PGL(5, \C)$ 
such that there is a homogeneous polynomial $F$ of degree $5$, which is {\it smooth}, such that for each $g \in G$ there is $A \in \GL(5, \C)$ such that $[A] = g$, i.e., $A$ a representative of $g$, and $A(F) = F$. 

This is in principle possible but practically hard in general, which we first explain. As usual in the finite group theory, we proceed our classification of $G$ in the following three steps:

(i) determine all possible prime orders of elements of ${\rm Aut}\, (X)$ of some QCY3 $X$ (Section \ref{ss:primeorders}). It turns out that they are $2$, $3$, $5$, $13$, $17$, $41$;

(ii) determine all possible Sylow $p$-subgroups of ${\rm Aut}\, (X)$ of some QCY3 $X$ (Sections \ref{ss:primeorders}, \ref{ss:Sylow2}, \ref{ss:Sylow5});

(iii) determine all possible \lq\lq{}combinations\rq\rq{} of Sylow $p$-subgroups which result in ${\rm Aut}\, (X)$  for some QCY3 $X$ (Sections \ref{ss:primeorders}, \ref{ss:Sylow2}, \ref{ss:Sylow5}, \ref{ss:2^63^25^2}).

In each step:

(I) the existence of smooth $F$ is essential. In fact, many groups are excluded by showing that they can act only on {\it singular} quintic threefolds. For this, we will give some useful criterions for (non-)smoothness of hypersurfaces (Lemma \ref{lem:nonsmooth}, Proposition \ref{pp:nonsmoothquintic}) in Section \ref{ss:differentialmethod}.  We also mildly use Mathematica to check smoothness. 

If there is a subgroup $\tilde{G} \subset \GL(5, \C)$ such that 
$\tilde{G} \simeq G$ under $A \mapsto [A]$, such that $A(F) = F$ for all $A \in \tilde{G}$, we call $G$ {\it $F$-liftable}.  The problem is further reduced to a linear algebraic problem of the classical invariant theory {\it for $F$-liftable $G$}. 

The other involving issues are group theorectical ones (II), (III):

(II) the fact that in general $G$ are not $F$-liftable when $5 | G$ . 

In our case, with a help of Schur multiplier theory (see Section \ref{ss:Schurandliftable}), we find that there are only very few groups which are not $F$-liftable. 

(III) The numbers of groups whose orders are the product of powers of smaller primes are too huge to control just by hand.

As in Mukai's (\cite{Mu88}) classification of finite symplectic action on K3 surfaces, (III) is a combersome problem. In his case, it arises for groups of order $2^a3^b$, especially the cases $2^n$, which is overcome with a help of classifications of $2$-groups of order $\le 2^6$ in his paper. In our case, problems arise for order $2^a3^b5^c$. We treat them by using now a quite useful tool, GAP software. GAP shows all possible linear 
representation and subgroups etc. in the range of order $\le 2000$ (except 1024). With its help, we control the cases where $\le 2000$ and then larger orders cases are reduced to these cases in PC free way.

 On the other hand, the case where $G$ has a larger prime order element, Theorem of Brauer (see Theorem \ref{thm:Brauer}) is quite effective to determine possible $G$ in our case. We also note that this is used in \cite{Ad78} to show the simplicity of the automorphism group of the Klein cubic threefold. In our case, we use it to determine $G$ of which order is divisible by larger primes, say, $13$, $17$ or $41$. It is also worth noticing that the full automorphism group is never cyclic of prime orders in these cases, i.e., there is no QCY3 $X$ such that ${\rm Aut}\, (X) \simeq C_{13}$, $C_{17}$ or $C_{41}$ (Theorem \ref{thm:containorder13}, Theorem \ref{thm:containorder17}, Theorem \ref{thm:containorder41}), while there are QCY3 whose automorphism group is isomorphic to $\{e\}$, $C_2$, $C_3$ and $C_5$ (see Remark \ref{rmk:Autverysmall}).

Throughout this paper, we use so-called {\it the differential method} (Theorem \ref{lem:differentialmethod}, Theorem \ref{thm:Aut(X)1-16}) to compute full automorphism group $\Aut(X)$ when $F$ is \lq\lq{}special\rq\rq{}.

 We believe that our methods to determine $G$ can be applied to classify automorphism groups of smooth hypersurfaces of other types, especially automorphism groups of smooth cubic threefolds (cf. \cite{Ad78}) and those of smooth cubic fourfolds (cf. \cite{GL11}), the later of which may also be applicable to study interesting automorphisms of hyperkaehler fourfolds of ${\rm (K3)}^{[2]}$ type (\cite{BCS14}).

To close Introduction, we remark some possible applications and motivations.

From a group theoretical point, it is particularly interested in the solvability of the groups and what kind of non-commutative simple groups are realized as an automorphism subgroup of QCY3. Our project actually started by motivating a discovery of simple non-commutative groups acting on QCY3. It turns out that there is only one non-commutative simple group $A_5$, the simple non-commutative group of the smallest order, and that all automorphism subgroups are solvable unless they do contain $A_5$ as a subgroup. In fact, there are exactly 8 non-solvable ones: $A_5$, $S_5$, $A_5\times C_5$, $S_5\times C_5$, $C_5^3\rt A_5$, $C_5^3\rt S_5$,  $C_5^4\rt A_5$, $C_5^4\rt S_5$.

In the topological mirror symmetry of Calabi-Yau threefolds, finite Gorenstein automorphisms played important roles in constructing mirror families (\cite{CDP93}, \cite{BD96}). For example, a natural mirror family of quintic Calabi-Yau threefolds is given by a crepant resolution $Y_{\lambda}$ of the Gorenstein  quotient of 
$$X_{\lambda} = (x_1^5 + x_2^5 + x_3^5 + x_4^5 + x_5^5 - 5\lambda x_1x_2x_3x_4x_5 = 0)$$
by $\mu_5^3 \simeq \mu_5^4/\mu_5$. The  manifolds $Y_{\lambda}$ ($\lambda^5 \not= 1$) are smooth Calabi-Yau threefolds of 
$$h^{1,1}(Y_{\lambda}) = 101 = h^{1,2}(X_{\lambda})\,\, ,\,\, h^{1,2}(Y_{\lambda}) = 1 = h^{1,1}(X_{\lambda})\,\, .$$
Now $Y_{\lambda}$ is also understood by the derived McKay correspondece developed by \cite{BKR01} and also computation of these Hodge numbrs are given by the classical McKay correspondeces due to \cite{IR96}, \cite{BD96}. It will be also interesting to see what kind of manifolds appear as the non-Gorenstein quotients. 

In this paper, we will not touch these interesting questions. However, we hope that our classification with explicit equations of target QCYs will provide handy useful global test examples for further study in birational geometry of threefolds such as McKay correspondence problems mentioned above or its possible generalizations for non-Gorenstein quotient. \\[.3cm]

 \noindent
{\bf Notations and conventions.}  We use the following notations to describe groups.\\[.2cm] In this paper, if $A\in \GL(n,\mathbb{C})$, then we use $[A]$ denote the corresponding element in $\PGL(n, \mathbb{C})$. 

  $I_n:=$ the identity matrix of rank $n$;
  
  $\xi_k:=e^{\frac{2\pi i}{k}}$ a $k$-th primitive root of unity, where $k$ is a positive integer;
  
%

    If $A\in \GL(n, \mathbb{C})$ and $\alpha_1,...,\alpha_n$ are eigenvalues (considering multiplicities) of $A$, then  we use $\chi_A(t)=(t-\alpha_1)\cdot\cdot\cdot (t-\alpha_n)$ to denote the characteristic polynomial of $A$. If $B_1,...,B_k$ are square matrices, then we use $\Diag(B_1,...,B_k)$ to denote the obvious block diagonal matrix.
    
     We use $\pi: \GL(n,\mathbb{C})\longrightarrow \PGL(n,\mathbb{C})$ to denote the natural quotient map.
     
     Let $G$ be a finite group and $p$ be a prime. If no confusion causes, we use $G_p$ to denote a Sylow $p$-subgroup of $G$.
     
    The following is the list of symbols of finite groups used in this article:
     
     $C_n$: a cyclic group of order $n$,
     
     $D_{2n}$: a dihedral group of order $2n$,
     
     $S_n(A_n)$: a symmetric (alternative) group of degree $n$,
     
     $Q_8$: a quaternion group of order 8.
     \\[.5cm]

\section{Examples of group actions and main Theorem}\label{ss:22examples}

Let us begin with explicit examples (1)-(22). It turns out that the 22 groups essentially classify the all automorphism groups of smooth quintic 3-folds (see Theorem \ref{thm:Main} for a precise statement).

\begin{example}\label{mainex}
(1) Fermat quintic threefold $X$: $F=x_1^5+x_2^5+x_3^5+x_4^5+x_5^5=0$. Let $G$ be the subgroup of $\PGL(5,\mathbb{C})$ generated by the following seven matrices:

$A_1=
    \begin{pmatrix} 
    0&1&0&0&0 \\ 
    1&0&0&0&0 \\ 
    0&0&1&0&0 \\ 
    0&0&0&1&0 \\ 
    0&0&0&0&1 \\ 
    \end{pmatrix} A_2=
    \begin{pmatrix} 
    0&1&0&0&0 \\ 
    0&0&1&0&0 \\ 
    1&0&0&0&0 \\ 
    0&0&0&1&0 \\ 
    0&0&0&0&1 \\ 
    \end{pmatrix} A_3=
    \begin{pmatrix} 
    0&1&0&0&0 \\ 
    0&0&1&0&0 \\ 
    0&0&0&1&0 \\ 
    0&0&0&0&1 \\ 
    1&0&0&0&0 \\ 
    \end{pmatrix}$
    
    $A_4=
    \begin{pmatrix} 
    1&0&0&0&0 \\ 
    0&\xi_5&0&0&0 \\ 
    0&0&1&0&0 \\ 
    0&0&0&1&0 \\ 
    0&0&0&0&1 \\ 
    \end{pmatrix} A_5=
    \begin{pmatrix} 
    1&0&0&0&0 \\ 
    0&1&0&0&0 \\ 
    0&0&\xi_5&0&0 \\ 
    0&0&0&1&0 \\ 
    0&0&0&0&1 \\ 
    \end{pmatrix} A_6=
    \begin{pmatrix} 
    1&0&0&0&0 \\ 
    0&1&0&0&0 \\ 
    0&0&1&0&0 \\ 
    0&0&0&\xi_5&0 \\ 
    0&0&0&0&1 \\ 
    \end{pmatrix}$
    
    $A_7=
    \begin{pmatrix} 
    1&0&0&0&0 \\ 
    0&1&0&0&0 \\ 
    0&0&1&0&0 \\ 
    0&0&0&1&0 \\ 
    0&0&0&0&\xi_5 \\ 
    \end{pmatrix} $

Then $G$ acts on $X$, $G$ is isomorphic to $C_5^4\rtimes S_5$ and $|G|=2^3\cdot 3\cdot 5^5=75000$.
\\[.2cm]

(2) Let $X: F=x_1^4x_2+x_2^5+x_3^5+x_4^5+x_5^5=0$ and let $G$ be the subgroup of $\PGL(5,\mathbb{C})$ generated by the following six matrices: 

$A_1=
    \begin{pmatrix} 
    \xi_4&0&0&0&0 \\ 
    0&1&0&0&0 \\ 
    0&0&1&0&0 \\ 
    0&0&0&1&0 \\ 
    0&0&0&0&1 \\ 
    \end{pmatrix} A_2=
    \begin{pmatrix} 
    1&0&0&0&0 \\ 
    0&1&0&0&0 \\ 
    0&0&0&1&0 \\ 
    0&0&1&0&0 \\ 
    0&0&0&0&1 \\ 
    \end{pmatrix} A_3=
    \begin{pmatrix} 
    1&0&0&0&0 \\ 
    0&1&0&0&0 \\ 
    0&0&0&1&0 \\ 
    0&0&0&0&1 \\ 
    0&0&1&0&0 \\ 
    \end{pmatrix}$
    
    $A_4=
    \begin{pmatrix} 
    1&0&0&0&0 \\ 
    0&1&0&0&0 \\ 
    0&0&\xi_5&0&0 \\ 
    0&0&0&1&0 \\ 
    0&0&0&0&1 \\ 
    \end{pmatrix} A_5=
    \begin{pmatrix} 
    1&0&0&0&0 \\ 
    0&1&0&0&0 \\ 
    0&0&1&0&0 \\ 
    0&0&0&\xi_5&0 \\ 
    0&0&0&0&1 \\ 
    \end{pmatrix} A_6=
    \begin{pmatrix} 
    1&0&0&0&0 \\ 
    0&1&0&0&0 \\ 
    0&0&1&0&0 \\ 
    0&0&0&1&0 \\ 
    0&0&0&0&\xi_5 \\ 
    \end{pmatrix}$

  Then $G$ acts on $X$, $G\cong C_4\times (C_5^3\rtimes S_3)$, and $|G|=2^3\cdot 3\cdot 5^3=3000$.\\[.2cm]
  
 (3) Let $X: F=x_1^4x_2+x_2^5+x_3^4x_4+x_4^5+x_5^5=0$, and let $G$ be the subgroup of $\PGL(5,\mathbb{C})$ generated by the following five matrices:
 
 $A_1=
    \begin{pmatrix} 
    \xi_4&0&0&0&0 \\ 
    0&1&0&0&0 \\ 
    0&0&1&0&0 \\ 
    0&0&0&1&0 \\ 
    0&0&0&0&1 \\ 
    \end{pmatrix} A_2=
    \begin{pmatrix} 
    1&0&0&0&0 \\ 
    0&1&0&0&0 \\ 
    0&0&\xi_4&0&0 \\ 
    0&0&0&1&0 \\ 
    0&0&0&0&1 \\ 
    \end{pmatrix} A_3=
    \begin{pmatrix} 
    \xi_5&0&0&0&0 \\ 
    0&\xi_5&0&0&0 \\ 
    0&0&1&0&0 \\ 
    0&0&0&1&0 \\ 
    0&0&0&0&1 \\ 
    \end{pmatrix}$ 
    
    $A_4=
    \begin{pmatrix} 
    1&0&0&0&0 \\ 
    0&1&0&0&0 \\ 
    0&0&\xi_5&0&0 \\ 
    0&0&0&\xi_5&0 \\ 
    0&0&0&0&1 \\ 
    \end{pmatrix} A_5=
    \begin{pmatrix} 
    0&0&1&0&0 \\ 
    0&0&0&1&0 \\ 
    1&0&0&0&0 \\ 
    0&1&0&0&0 \\ 
    0&0&0&0&1 \\ 
    \end{pmatrix}$ 
 
 Then $G$ acts on $X$, $G$ is isomorphic to $(C_5^2\times C_4^2)\rtimes C_2$ and $|G|=2^5\cdot 5^2=800$. \\[.2cm]

 (4) Let $X: F=x_1^4x_2+x_2^4x_3+x_3^5+x_4^5+x_5^5=0$ and let $G$ be the subgroup of $\PGL(5,\mathbb{C})$ generated by the following four matrices:
 
 $A_1=
    \begin{pmatrix} 
    \xi_{16}&0&0&0&0 \\ 
    0&\xi_{16}^{-4}&0&0&0 \\ 
    0&0&1&0&0 \\ 
    0&0&0&1&0 \\ 
    0&0&0&0&1 \\ 
    \end{pmatrix} A_2=
    \begin{pmatrix} 
    1&0&0&0&0 \\ 
    0&1&0&0&0 \\ 
    0&0&1&0&0 \\ 
    0&0&0&\xi_5&0 \\ 
    0&0&0&0&1 \\ 
    \end{pmatrix} A_3=
    \begin{pmatrix} 
    1&0&0&0&0 \\ 
    0&1&0&0&0 \\ 
    0&0&1&0&0 \\ 
    0&0&0&1&0 \\ 
    0&0&0&0&\xi_5 \\ 
    \end{pmatrix}$ 
    
     $A_4=
    \begin{pmatrix} 
    1&0&0&0&0 \\ 
    0&1&0&0&0 \\ 
    0&0&1&0&0 \\ 
    0&0&0&0&1 \\ 
    0&0&0&1&0 \\ 
    \end{pmatrix} $
 
 Then $G$ acts on $X$, $G\cong C_{16}\times (C_5^2\rtimes C_2)$ and $|G|=2^5\cdot 5^2=800$.\\[.2cm]
 
 (5) Let $X:F=x_1^4x_2+x_2^4x_1+x_3^5+x_4^5+x_5^5=0$ and let $G$ be the subgroup of $\PGL(5,\mathbb{C})$ generated by the following seven matrices:
 
 $A_1=
    \begin{pmatrix} 
    1&0&0&0&0 \\ 
    0&1&0&0&0 \\ 
    0&0&0&1&0 \\ 
    0&0&1&0&0 \\ 
    0&0&0&0&1 \\ 
    \end{pmatrix} A_2=
    \begin{pmatrix} 
    1&0&0&0&0 \\ 
    0&1&0&0&0 \\ 
    0&0&0&1&0 \\ 
    0&0&0&0&1 \\ 
    0&0&1&0&0 \\ 
    \end{pmatrix} A_3=
    \begin{pmatrix} 
    1&0&0&0&0 \\ 
    0&1&0&0&0 \\ 
    0&0&\xi_5&0&0 \\ 
    0&0&0&1&0 \\ 
    0&0&0&0&1 \\ 
    \end{pmatrix}$
    
    $A_4=
    \begin{pmatrix} 
    1&0&0&0&0 \\ 
    0&1&0&0&0 \\ 
    0&0&1&0&0 \\ 
    0&0&0&\xi_5&0 \\ 
    0&0&0&0&1 \\ 
    \end{pmatrix} A_5=
    \begin{pmatrix} 
    1&0&0&0&0 \\ 
    0&1&0&0&0 \\ 
    0&0&1&0&0 \\ 
    0&0&0&1&0 \\ 
    0&0&0&0&\xi_5 \\ 
    \end{pmatrix} A_6=
    \begin{pmatrix} 
    \xi_3&0&0&0&0 \\ 
    0&\xi_3^2&0&0&0 \\ 
    0&0&1&0&0 \\ 
    0&0&0&1&0 \\ 
    0&0&0&0&1 \\ 
    \end{pmatrix}$
    
    $A_7=
    \begin{pmatrix} 
    0&1&0&0&0 \\ 
    1&0&0&0&0 \\ 
    0&0&1&0&0 \\ 
    0&0&0&1&0 \\ 
    0&0&0&0&1 \\ 
    \end{pmatrix} $

 Then $G$ acts on $X$, $G\cong S_3\times (C_5^3\rtimes S_3)$ and $|G|=2^2\cdot 3^2\cdot 5^3=4500$.\\[.2cm]
 
 (6) Let $X: F=x_1^4x_2+x_2^4x_3+x_3^5+x_4^4x_5+x_5^5=0$, and let $G$ be the subgroup of $\PGL(5,\mathbb{C})$ generated by the following three matrices:
 
  $A_1=
    \begin{pmatrix} 
    \xi_{16}&0&0&0&0 \\ 
    0&\xi_{16}^{-4}&0&0&0 \\ 
    0&0&1&0&0 \\ 
    0&0&0&1&0 \\ 
    0&0&0&0&1 \\ 
    \end{pmatrix} A_2=
    \begin{pmatrix} 
    1&0&0&0&0 \\ 
    0&1&0&0&0 \\ 
    0&0&1&0&0 \\ 
    0&0&0&\xi_4&0 \\ 
    0&0&0&0&1 \\ 
    \end{pmatrix} A_6=
    \begin{pmatrix} 
    \xi_5&0&0&0&0 \\ 
    0&\xi_5&0&0&0 \\ 
    0&0&\xi_5&0&0 \\ 
    0&0&0&1&0 \\ 
    0&0&0&0&1 \\ 
    \end{pmatrix}$
    
 Then $G$ acts on $X$, $G\cong C_5\times C_{16}\times C_4$ and $|G|=2^6\cdot 5=320$.\\[.2cm]

 (7) Let $X: F=x_1^4x_2+x_2^4x_3+x_3^4x_4+x_4^5+x_5^5=0$, and let $G$ be the subgroup of $\PGL(5,\mathbb{C})$ generated by the following two matrices:
 
    $A_1=
    \begin{pmatrix} 
    \xi_{64}&0&0&0&0 \\ 
    0&\xi_{64}^{-4}&0&0&0 \\ 
    0&0&\xi_{64}^{16}&0&0 \\ 
    0&0&0&1&0 \\ 
    0&0&0&0&1 \\ 
    \end{pmatrix} A_2=
    \begin{pmatrix} 
    1&0&0&0&0 \\ 
    0&1&0&0&0 \\ 
    0&0&1&0&0 \\ 
    0&0&0&1&0 \\ 
    0&0&0&0&\xi_5 \\ 
    \end{pmatrix}$

  Then $G$ acts on $X$, $G\cong C_{64}\times C_5$ and $|G|=2^6\cdot 5=320$.\\[.2cm]

  (8) Let $X: F=x_1^4x_2+x_2^5+x_3^4x_4+x_4^4x_3+x_5^5=0$, and let $G$ be the subgroup of $\PGL(5,\mathbb{C})$ generated by the following five matrices:
  
   $A_1=
    \begin{pmatrix} 
    \xi_{4}&0&0&0&0 \\ 
    0&1&0&0&0 \\ 
    0&0&1&0&0 \\ 
    0&0&0&1&0 \\ 
    0&0&0&0&1 \\ 
    \end{pmatrix} A_2=
    \begin{pmatrix} 
    \xi_5&0&0&0&0 \\ 
    0&\xi_5&0&0&0 \\ 
    0&0&1&0&0 \\ 
    0&0&0&1&0 \\ 
    0&0&0&0&1 \\ 
    \end{pmatrix} A_3=
    \begin{pmatrix} 
    1&0&0&0&0 \\ 
    0&1&0&0&0 \\ 
    0&0&0&1&0 \\ 
    0&0&1&0&0 \\ 
    0&0&0&0&1 \\ 
    \end{pmatrix}$
    
      $A_4=
    \begin{pmatrix} 
    1&0&0&0&0 \\ 
    0&1&0&0&0 \\ 
    0&0&\xi_3&0&0 \\ 
    0&0&0&\xi_3^2&0 \\ 
    0&0&0&0&1 \\ 
    \end{pmatrix} A_5=
    \begin{pmatrix} 
    1&0&0&0&0 \\ 
    0&1&0&0&0 \\ 
    0&0&1&0&0 \\ 
    0&0&0&1&0 \\ 
    0&0&0&0&\xi_5 \\ 
    \end{pmatrix}$
  
  Then $G$ acts on $X$, $G\cong C_5^2\times C_4\times S_3$ and $|G|=2^3\cdot 3\cdot 5^2=600$.\\[.2cm]

  (9) Let $X: F=x_1^4x_2+x_2^4x_3+x_3^4x_1+x_4^5+x_5^5=0$, and let $G$ be the subgroup of $\PGL(5,\mathbb{C})$ generated by the following five matrices:
  
  $A_1=
    \begin{pmatrix} 
    0&1&0&0&0 \\ 
    0&0&1&0&0 \\ 
    1&0&0&0&0 \\ 
    0&0&0&1&0 \\ 
    0&0&0&0&1 \\ 
    \end{pmatrix} A_2=
    \begin{pmatrix} 
    \xi_{13}&0&0&0&0 \\ 
    0&\xi_{13}^{-4}&0&0&0 \\ 
    0&0&\xi_{13}^3&0&0 \\ 
    0&0&0&1&0 \\ 
    0&0&0&0&1 \\ 
    \end{pmatrix} A_3=
    \begin{pmatrix} 
    1&0&0&0&0 \\ 
    0&1&0&0&0 \\ 
    0&0&1&0&0 \\ 
    0&0&0&0&1 \\ 
    0&0&0&1&0 \\ 
    \end{pmatrix}$
    
      $A_4=
    \begin{pmatrix} 
    1&0&0&0&0 \\ 
    0&1&0&0&0 \\ 
    0&0&1&0&0 \\ 
    0&0&0&\xi_5&0 \\ 
    0&0&0&0&1 \\ 
    \end{pmatrix} A_5=
    \begin{pmatrix} 
    1&0&0&0&0 \\ 
    0&1&0&0&0 \\ 
    0&0&1&0&0 \\ 
    0&0&0&1&0 \\ 
    0&0&0&0&\xi_5 \\ 
    \end{pmatrix}$
  
  Then $G$ acts on $X$, $G\cong (C_5^2\rtimes C_2)\times (C_{13}\rtimes C_3)$ and $|G|=2\cdot 3\cdot 5^2\cdot 13=1950$.\\[.2cm]
  
  (10)  Let $X: F=x_1^4x_2+x_2^4x_3+x_3^5+x_4^4x_5+x_5^4x_4=0$, and let $G$ be the subgroup of $\PGL(5,\mathbb{C})$ generated by the following four matrices:
  
   $A_1=
    \begin{pmatrix} 
    \xi_{16}&0&0&0&0 \\ 
    0&\xi_{16}^{-4}&0&0&0 \\ 
    0&0&1&0&0 \\ 
    0&0&0&1&0 \\ 
    0&0&0&0&1 \\ 
    \end{pmatrix} A_2=
    \begin{pmatrix} 
    1&0&0&0&0 \\ 
    0&1&0&0&0 \\ 
    0&0&1&0&0 \\ 
    0&0&0&0&1 \\ 
    0&0&0&1&0 \\ 
    \end{pmatrix} A_3=
    \begin{pmatrix} 
    1&0&0&0&0 \\ 
    0&1&0&0&0 \\ 
    0&0&1&0&0 \\ 
    0&0&0&\xi_3&0 \\ 
    0&0&0&0&\xi_3^2 \\ 
    \end{pmatrix}$ 
    
     $A_4=
    \begin{pmatrix} 
    1&0&0&0&0 \\ 
    0&1&0&0&0 \\ 
    0&0&1&0&0 \\ 
    0&0&0&\xi_5&0 \\ 
    0&0&0&0&\xi_5 \\ 
    \end{pmatrix} $
 
 Then $G$ acts on $X$, $G\cong C_{16}\times (C_5\times S_3)$ and $|G|=2^5\cdot 3\cdot 5=480$. \\[.2cm]

  (11) Let $X: F=x_1^4x_2+x_2^4x_3+x_3^4x_4+x_4^4x_5+x_5^5=0$  and let $G$ be the subgroup of $\PGL(5,\mathbb{C})$ generated by the following matrix:
  
    $A_1=
    \begin{pmatrix} 
    \xi_{256}&0&0&0&0 \\ 
    0&\xi_{256}^{-4}&0&0&0 \\ 
    0&0&\xi_{256}^{16}&0&0 \\ 
    0&0&0&\xi_{256}^{-64}&0 \\ 
    0&0&0&0&1 \\ 
    \end{pmatrix} $
 
 Then $G$ acts on $X$, $G\cong C_{256}$ and $|G|=2^8=256$.\\[.2cm]
 
 (12) Let $X: F=x_1^4x_2+x_2^5+x_3^4x_4+x_4^4x_5+x_5^4x_3=0$, and let $G$ be the subgroup of $\PGL(5,\mathbb{C})$ generated by the following four matrices:
 
  $A_1=
    \begin{pmatrix} 
    \xi_{4}&0&0&0&0 \\ 
    0&1&0&0&0 \\ 
    0&0&1&0&0 \\ 
    0&0&0&1&0 \\ 
    0&0&0&0&1 \\ 
    \end{pmatrix} A_2=
    \begin{pmatrix} 
    \xi_5&0&0&0&0 \\ 
    0&\xi_5&0&0&0 \\ 
    0&0&1&0&0 \\ 
    0&0&0&1&0 \\ 
    0&0&0&0&1 \\ 
    \end{pmatrix} A_3=
    \begin{pmatrix} 
    1&0&0&0&0 \\ 
    0&1&0&0&0 \\ 
    0&0&0&1&0 \\ 
    0&0&0&0&1 \\ 
    0&0&1&0&0 \\ 
    \end{pmatrix}$ 
    
     $A_4=
    \begin{pmatrix} 
    1&0&0&0&0 \\ 
    0&1&0&0&0 \\ 
    0&0&\xi_{13}&0&0 \\ 
    0&0&0&\xi_{13}^{-4}&0 \\ 
    0&0&0&0&\xi_{13}^3 \\ 
    \end{pmatrix} $

   Then $G$ acts on $X$, $G\cong C_{4}\times C_{5}\times (C_{13}\rtimes C_3)$ and $|G|=2^2\cdot 3\cdot 5\cdot 13=780$.
\\[.2cm]
   
   (13) Let $X: F=x_1^4x_2+x_2^4x_3+x_3^4x_4+x_4^4x_1+x_5^5=0$, and let $G$ be the subgroup of $\PGL(5,\mathbb{C})$ generated by the following four matrices:

     $A_1=
    \begin{pmatrix} 
    \xi_{17}&0&0&0&0 \\ 
    0&\xi_{17}^{-4}&0&0&0 \\ 
    0&0&\xi_{17}^{16}&0&0 \\ 
    0&0&0&\xi_{17}^{4}&0 \\ 
    0&0&0&0&1 \\ 
    \end{pmatrix} A_2=
    \begin{pmatrix} 
    \xi_3&0&0&0&0 \\ 
    0&\xi_3^2&0&0&0 \\ 
    0&0&\xi_3&0&0 \\ 
    0&0&0&\xi_3^2&0 \\ 
    0&0&0&0&1 \\ 
    \end{pmatrix} A_3=
    \begin{pmatrix} 
    0&1&0&0&0 \\ 
    0&0&1&0&0 \\ 
    0&0&0&1&0 \\ 
    1&0&0&0&0 \\ 
    0&0&0&0&1 \\ 
    \end{pmatrix}$ 
    
     $A_4=
    \begin{pmatrix} 
    1&0&0&0&0 \\ 
    0&1&0&0&0 \\ 
    0&0&1&0&0 \\ 
    0&0&0&1&0 \\ 
    0&0&0&0&\xi_5 \\ 
    \end{pmatrix} $

   Then $G$ acts on $X$, $G\cong C_{5}\times (C_{51}\rtimes C_4)$ and $|G|=2^2\cdot 3\cdot 5\cdot 17=1020$.\\[.2cm]

   (14) Let $X: F=x_1^4x_2+x_2^4x_1+x_3^4x_4+x_4^4x_3+x_5^5=0$ and let $G$ be the subgroup of $\PGL(5,\mathbb{C})$ generated by the following seven matrices:

     $A_1=
    \begin{pmatrix} 
    \xi_{3}&0&0&0&0 \\ 
    0&\xi_{3}^{2}&0&0&0 \\ 
    0&0&1&0&0 \\ 
    0&0&0&1&0 \\ 
    0&0&0&0&1 \\ 
    \end{pmatrix} A_2=
    \begin{pmatrix} 
    \xi_5&0&0&0&0 \\ 
    0&\xi_5&0&0&0 \\ 
    0&0&1&0&0 \\ 
    0&0&0&1&0 \\ 
    0&0&0&0&1 \\ 
    \end{pmatrix} A_3=
    \begin{pmatrix} 
    1&0&0&0&0 \\ 
    0&1&0&0&0 \\ 
    0&0&\xi_3&0&0 \\ 
    0&0&0&\xi_3^2&0 \\ 
    0&0&0&0&1 \\ 
    \end{pmatrix}$ 
    
      $A_4=
    \begin{pmatrix} 
    1&0&0&0&0 \\ 
    0&1&0&0&0 \\ 
    0&0&\xi_{5}&0&0 \\ 
    0&0&0&\xi_{5}&0 \\ 
    0&0&0&0&1 \\ 
    \end{pmatrix} A_5=
    \begin{pmatrix} 
    0&1&0&0&0 \\ 
    1&0&0&0&0 \\ 
    0&0&1&0&0 \\ 
    0&0&0&1&0 \\ 
    0&0&0&0&1 \\ 
    \end{pmatrix} A_6=
    \begin{pmatrix} 
    1&0&0&0&0 \\ 
    0&1&0&0&0 \\ 
    0&0&0&1&0 \\ 
    0&0&1&0&0 \\ 
    0&0&0&0&1 \\ 
    \end{pmatrix}$ 
    
     $A_7=
    \begin{pmatrix} 
    0&0&1&0&0 \\ 
    0&0&0&1&0 \\ 
    0&1&0&0&0 \\ 
    1&0&0&0&0 \\ 
    0&0&0&0&1 \\ 
    \end{pmatrix} $

   Then $G$ acts on $X$, $G\cong (C_{5}^2\times C_3^2)\rtimes D_8$ and $|G|=2^3\cdot 3^2\cdot 5^2=1800$.
  \\[.2cm]
   
   (15) Klein quintic threefold $X: F=x_1^4x_2+x_2^4x_3+x_3^4x_4+x_4^4x_5+x_5^4x_1=0$ and let $G$ be the subgroup of $\PGL(5,\mathbb{C})$ generated by the following three matrices:

     $A_1=
    \begin{pmatrix} 
    0&1&0&0&0 \\ 
    0&0&1&0&0 \\ 
    0&0&0&1&0 \\ 
    0&0&0&0&1 \\ 
    1&0&0&0&0 \\ 
    \end{pmatrix} A_2=
    \begin{pmatrix} 
    1&0&0&0&0 \\ 
    0&\xi_5&0&0&0 \\ 
    0&0&\xi_5^2&0&0 \\ 
    0&0&0&\xi_5^3&0 \\ 
    0&0&0&0&\xi_5^4 \\ 
    \end{pmatrix} A_3=
    \begin{pmatrix} 
    \xi_{41}&0&0&0&0 \\ 
    0&\xi_{41}^{-4}&0&0&0 \\ 
    0&0&\xi_{41}^{16}&0&0 \\ 
    0&0&0&\xi_{41}^{18}&0 \\ 
    0&0&0&0&\xi_{41}^{10} \\ 
    \end{pmatrix}$ 
 
  Then $G$ acts on $X$, $G\cong C_{205}\rtimes C_5$ and $|G|=5^2\cdot 41=1025$.
   \\[.2cm]
  
  (16) Let $X: F=x_1^4x_2+x_2^4x_3+x_3^4x_1+x_4^4x_5+x_5^4x_4=0$ and let $G$ be the subgroup of $\PGL(5,\mathbb{C})$ generated by the following five matrices:

   $A_1=
    \begin{pmatrix} 
    \xi_{13}&0&0&0&0 \\ 
    0&\xi_{13}^{-4}&0&0&0 \\ 
    0&0&\xi_{13}^3&0&0 \\ 
    0&0&0&1&0 \\ 
    0&0&0&0&1 \\ 
    \end{pmatrix} A_2=
    \begin{pmatrix} 
    0&1&0&0&0 \\ 
    0&0&1&0&0 \\ 
    1&0&0&0&0 \\ 
    0&0&0&1&0 \\ 
    0&0&0&0&1 \\ 
    \end{pmatrix} A_3=
    \begin{pmatrix} 
    1&0&0&0&0 \\ 
    0&1&0&0&0 \\ 
    0&0&1&0&0 \\ 
    0&0&0&0&1 \\ 
    0&0&0&1&0 \\ 
    \end{pmatrix}$
    
      $A_4=
    \begin{pmatrix} 
    1&0&0&0&0 \\ 
    0&1&0&0&0 \\ 
    0&0&1&0&0 \\ 
    0&0&0&\xi_3&0 \\ 
    0&0&0&0&\xi_3^2 \\ 
    \end{pmatrix} A_5=
    \begin{pmatrix} 
    1&0&0&0&0 \\ 
    0&1&0&0&0 \\ 
    0&0&1&0&0 \\ 
    0&0&0&\xi_5&0 \\ 
    0&0&0&0&\xi_5 \\ 
    \end{pmatrix}$
  
  Then $G$ acts on $X$, $G\cong C_5\times S_3\times (C_{13}\rtimes C_3)$ and $|G|=2\cdot 3^2\cdot 5\cdot 13=1170$.\\[.2cm]
  
  (17) Let $X: F=((x_1^4+x_2^4)+(2+4\xi_3^2)x_1^2x_2^2)x_3+(-(x_1^4+x_2^4)+(2+4\xi_3^2)x_1^2x_2^2)x_4+x_3^4x_4+x_4^4x_3+x_5^5=0$ and let $G$ be the subgroup of $\PGL(5,\mathbb{C})$ generated by the following five matrices :

   $ A_1=
    \begin{pmatrix} 
    0&1&0&0&0 \\ 
    -1&0&0&0&0 \\ 
    0&0&1&0&0 \\ 
    0&0&0&1&0 \\ 
    0&0&0&0&1 \\ 
    \end{pmatrix} A_2=
    \begin{pmatrix} 
    \xi_8^3&0&0&0&0 \\ 
    0&\xi_8&0&0&0 \\ 
    0&0&0&1&0 \\ 
    0&0&1&0&0 \\ 
    0&0&0&0&1 \\ 
    \end{pmatrix}  A_3=
    \begin{pmatrix} 
    -\frac{1}{\sqrt{2}}\xi_8&\frac{1}{\sqrt{2}}\xi_8&0&0&0 \\ 
    \frac{1}{\sqrt{2}}\xi_8^3& \frac{1}{\sqrt{2}}\xi_8^3&0&0&0 \\ 
    0&0&\xi_3&0&0 \\ 
    0&0&0&\xi_3^2&0 \\ 
    0&0&0&0&1 \\ 
    \end{pmatrix}$
    
      $A_4=
    \begin{pmatrix} 
    -1&0&0&0&0 \\ 
    0&1&0&0&0 \\ 
    0&0&1&0&0 \\ 
    0&0&0&1&0 \\ 
    0&0&0&0&1 \\ 
    \end{pmatrix} A_5=
    \begin{pmatrix} 
    1&0&0&0&0 \\ 
    0&1&0&0&0 \\ 
    0&0&1&0&0 \\ 
    0&0&0&1&0 \\ 
    0&0&0&0&\xi_5 \\ 
    \end{pmatrix}$
  
  Then $G$ acts on $X$, $G\cong C_5\times ((\SL(2,3).C_2)\rtimes C_2)$ and $|G|=2^5\cdot 3\cdot 5=480$.
  \\[.2cm]

    (18) Let $X: F=((x_1^4+x_2^4)+(2+4\xi_3^2)x_1^2x_2^2)x_3+((x_1^4+x_2^4)-(2+4\xi_3^2)x_1^2x_2^2)x_4+x_3^4x_4+x_4^4x_3+x_3^2x_4^2x_5+x_5^5=0$ and let $G$ be the subgroup of $\PGL(5,\mathbb{C})$ generated by the following four matrices:

   $A_1=
    \begin{pmatrix} 
    0&1&0&0&0 \\ 
    -1&0&0&0&0 \\ 
    0&0&1&0&0 \\ 
    0&0&0&1&0 \\ 
    0&0&0&0&1 \\ 
    \end{pmatrix} A_2=
    \begin{pmatrix} 
    -\xi_4&0&0&0&0 \\ 
    0&\xi_4&0&0&0 \\ 
    0&0&1&0&0 \\ 
    0&0&0&1&0 \\ 
    0&0&0&0&1 \\ 
    \end{pmatrix} A_3=
    \begin{pmatrix} 
    -\frac{1}{\sqrt{2}}\xi_8&\frac{1}{\sqrt{2}}\xi_8&0&0&0 \\ 
    \frac{1}{\sqrt{2}}\xi_8^3& \frac{1}{\sqrt{2}}\xi_8^3&0&0&0 \\ 
    0&0&\xi_3&0&0 \\ 
    0&0&0&\xi_3^2&0 \\ 
    0&0&0&0&1 \\ 
    \end{pmatrix}$
    
     $ A_4=
    \begin{pmatrix} 
    \xi_4&0&0&0&0 \\ 
    0&1&0&0&0 \\ 
    0&0&0&1&0 \\ 
    0&0&1&0&0 \\ 
    0&0&0&0&1 \\ 
    \end{pmatrix}$
  
  Then $G$ acts on $X$, $G\cong \SL(2,3)\rtimes C_4$ and $|G|=2^5\cdot 3=96$.\\[.2cm]

  (19) Let $X: F=x_1^4x_2+x_2^4x_1+x_3^4x_2+x_4^4x_1+x_5^5+x_2^3x_3x_4-x_1^3x_3x_4=0$ and let $G$ be the subgroup of $\PGL(5,\mathbb{C})$ generated by the following four matrices:

   $A_1=
    \begin{pmatrix} 
    1&0&0&0&0 \\ 
    0&1&0&0&0 \\ 
    0&0&1&0&0 \\ 
    0&0&0&1&0 \\ 
    0&0&0&0&\xi_5 \\ 
    \end{pmatrix} A_2=
    \begin{pmatrix} 
    \xi_3&0&0&0&0 \\ 
    0&\xi_3^2&0&0&0 \\ 
    0&0&\xi_3&0&0 \\ 
    0&0&0&\xi_3^2&0 \\ 
    0&0&0&0&1 \\ 
    \end{pmatrix} A_3=
    \begin{pmatrix} 
    0&1&0&0&0 \\ 
    1&0&0&0&0 \\ 
    0&0&0&1&0 \\ 
    0&0&-1&0&0 \\ 
    0&0&0&0&1 \\ 
    \end{pmatrix}$
    
      $A_4=
    \begin{pmatrix} 
    1&0&0&0&0 \\ 
    0&1&0&0&0 \\ 
    0&0&\xi_4&0&0 \\ 
    0&0&0&\xi_4^3&0 \\ 
    0&0&0&0&1 \\ 
    \end{pmatrix}$
  
   Then $G$ acts on $X$, $G\cong C_5 \times (C_3\rtimes Q_8)$ and $|G|=2^3\cdot 3\cdot 5=120$.
   \\[.2cm]
   
   (20) Let $X: F=x_1^4x_2+x_2^4x_1+x_3^4x_2+x_4^4x_1+x_5^5+x_1^3x_3x_4+x_2^3x_3x_4=0$ and let $G$ be the subgroup of $\PGL(5,\mathbb{C})$ generated by the following four matrices:

    $A_1=
    \begin{pmatrix} 
    1&0&0&0&0 \\ 
    0&1&0&0&0 \\ 
    0&0&1&0&0 \\ 
    0&0&0&1&0 \\ 
    0&0&0&0&\xi_5 \\ 
    \end{pmatrix} A_2=
    \begin{pmatrix} 
    \xi_3&0&0&0&0 \\ 
    0&\xi_3^2&0&0&0 \\ 
    0&0&\xi_3&0&0 \\ 
    0&0&0&\xi_3^2&0 \\ 
    0&0&0&0&1 \\ 
    \end{pmatrix} A_3=
    \begin{pmatrix} 
    0&1&0&0&0 \\ 
    1&0&0&0&0 \\ 
    0&0&0&1&0 \\ 
    0&0&1&0&0 \\ 
    0&0&0&0&1 \\ 
    \end{pmatrix}$
    
      $A_4=
    \begin{pmatrix} 
    1&0&0&0&0 \\ 
    0&1&0&0&0 \\ 
    0&0&\xi_4&0&0 \\ 
    0&0&0&\xi_4^3&0 \\ 
    0&0&0&0&1 \\ 
    \end{pmatrix}$
  
   Then $G$ acts on $X$, $G\cong C_5 \times D_{24}$ and $|G|=2^3\cdot 3\cdot 5=120$.
   \\[.2cm]
   
   (21) Let $X: \{x_1^5+x_2^5+x_3^5+x_4^5+x_5^5+x_6^5=x_1+x_2+x_3+x_4+x_5=0\}\subseteq \mathbb{P}^5$ and let $G$ be the subgroup of $\PGL(6,\C)$ generated by the following four matrices:

   $A_1=
    \begin{pmatrix} 
    0&1&0&0&0&0 \\ 
    1&0&0&0&0&0 \\ 
    0&0&1&0&0&0 \\ 
    0&0&0&1&0&0 \\ 
    0&0&0&0&1&0 \\
    0&0&0&0&0&1 \\ 
    \end{pmatrix} A_2=
    \begin{pmatrix} 
    0&1&0&0&0&0 \\ 
    0&0&1&0&0&0 \\ 
    1&0&0&0&0&0 \\ 
    0&0&0&1&0&0 \\ 
    0&0&0&0&1&0 \\
    0&0&0&0&0&1 \\ 
    \end{pmatrix} A_3=
    \begin{pmatrix} 
    0&1&0&0&0&0 \\ 
    0&0&1&0&0&0 \\ 
    0&0&0&1&0&0 \\ 
    0&0&0&0&1&0 \\ 
    1&0&0&0&0&0 \\
    0&0&0&0&0&1 \\ 
    \end{pmatrix}$
    
      $A_4=
    \begin{pmatrix} 
   1&0&0&0&0&0 \\ 
    0&1&0&0&0&0 \\ 
    0&0&1&0&0&0 \\ 
    0&0&0&1&0&0 \\ 
    0&0&0&0&1&0 \\
    0&0&0&0&0&\xi_5 \\ 
    \end{pmatrix}$

   Then $G$ acts on $X$, $G\cong C_5 \times S_5$ and $|G|=2^3\cdot 3\cdot 5^2=600$.   \\[.2cm]

 (22) Let $X: F=x_1^4x_2+x_2^4x_3+x_3^4x_4+x_4^5+x_5^4x_3+x_2^2x_4x_5^2=0$, and let $G$ be the subgroup of $\PGL(5,\mathbb{C})$ generated by the following two matrices:
 
    $A_1=
    \begin{pmatrix} 
    \xi_{32}&0&0&0&0 \\ 
    0&\xi_{32}^{-4}&0&0&0 \\ 
    0&0&\xi_{32}^{16}&0&0 \\ 
    0&0&0&1&0 \\ 
    0&0&0&0&\xi_{32}^4 \\ 
    \end{pmatrix} A_2=
    \begin{pmatrix} 
    1&0&0&0&0 \\ 
    0&1&0&0&0 \\ 
    0&0&1&0&0 \\ 
    0&0&0&1&0 \\ 
    0&0&0&0&-1 \\ 
    \end{pmatrix}$

  Then $G$ acts on $X$, $G\cong C_{32}\times C_2$ and $|G|=2^6=64$.
  
  \end{example}

  Our main theorem is the following (cf. \cite{Mu88}):
  \begin{theorem}\label{thm:Main}
  
  For a finite group $G$, the following two conditions are equivalent to each other: 
  
  (i) $G$ is isomorphic to a subgroup of one of the $22$ groups above, and 
  
  (ii) $G$ has a faithful action on a smooth quintic threefold.
\end{theorem}

We will prove Theorem \ref{thm:Main} in Section \ref{ss:proofmainthm}.

\begin{remark}
Let $Y$ be a smooth quintic threefold. Let $X_i$ be the smooth quintic threefold in Example (i) above, $1\leq i\leq 22$. It turns out that if $i\leq 17$,  then $\Aut(Y)\cong \Aut(X_i)$  if and only if, up to change of coordinates, $Y$ and $X_i$  are the same.

This remark is a byproduct of our proof of Theorem \ref{thm:Main}.
\end{remark}

\section{Smoothness of hypersurfaces and the differential method}\label{ss:differentialmethod}
\noindent

\begin{definition}
   Let $F=F(x_1,...,x_n)$ be a homogeneous polynomial of degree $d> 0$ and let $m=m(x_1,...,x_n)$ be a monomial of degree $d$. Then we say $m$ is in $F$ (or $m\in F$) if the coefficient of $m$ is not zero in the expression of $F$. 
\end{definition}

\begin{lemma}\label{lem:nonsmooth}
Let $F_d=F_d(x_1,...,x_{n+1})$ be a nonzero homogeneous polynomial of degree $d\geq 3$ and let $M:=\{F_d=0\}\subseteq \mathbb{P}^n$. Let $a$ and $b$ be two nonnegative integers, and $2a+b\leq n$. The hypersurface $M$ is not smooth if there exist $a+b$ distinct variables $x_{i_1},\ldots,x_{i_{a+b}}$ such that $F_d\in ( x_{i_1},\ldots,x_{i_a}) +( x_{i_{a+1}},\ldots,x_{i_{a+b}}) ^2$, where $( x_{k_1},\cdots,x_{k_m}) $ means the ideal of $\mathbb{C}[x_1,...,x_{n+1}]$ generated by $x_{k_1},\cdots,x_{k_m}$.

\end{lemma}

\begin{proof}
Without loss of generality, we may assume that $i_j=j$ for $1\leq j\leq a+b$. So $F_d$ can be written as $F_d=x_1G_1+\ldots +x_aG_a+H$ for some $G_i\in \C[x_1,...,x_{n+1}],  H\in ( x_{a+1},\ldots,x_{a+b}) ^2.$

We take the partial derivatives of $F_d$:

$\frac{\partial F_d}{\partial x_1}=G_1+x_1\frac{\partial G_1}{\partial x_1}+x_2\frac{\partial G_2}{\partial x_1}+\ldots +x_a\frac{\partial G_a}{\partial x_1}+\frac{\partial H}{\partial x_1}$;

$\vdots$

$\frac{\partial F_d}{\partial x_a}=x_1\frac{\partial G_1}{\partial x_a}+x_2\frac{\partial G_2}{\partial x_a}+\ldots +G_a+ x_a\frac{\partial G_a}{\partial x_a}+\frac{\partial H}{\partial x_a}$;

$\frac{\partial F_d}{\partial x_{a+1}}=x_1\frac{\partial G_1}{\partial x_{a+1}}+x_2\frac{\partial G_2}{\partial x_{a+1}}+\ldots +x_a\frac{\partial G_a}{\partial x_{a+1}}+\frac{\partial H}{\partial x_{a+1}}$;

$\vdots$

$\frac{\partial F_d}{\partial x_{n+1}}=x_1\frac{\partial G_1}{\partial x_{n+1}}+x_2\frac{\partial G_2}{\partial x_{n+1}}+\ldots +x_a\frac{\partial G_a}{\partial x_{n+1}}+\frac{\partial H}{\partial x_{n+1}}$.

Define $Z:=\{x_1=\cdots =x_{a+b}=G_1=\cdots=G_{a}=0\}\subseteq \mathbb{P}^n$. $Z\neq \emptyset $ since $2a+b\leq n$.

Then $\frac{\partial F_d}{\partial x_1}=\cdots=\frac{\partial F_d}{\partial x_{n+1}}=0$ at $Z$ and hence $M$ is singular at $Z$.
\end{proof}

\begin{proposition}\label{pp:smoothnessandmonomial}
Let $M=\{F_d=0\}\subseteq \mathbb{P}^n$ be a smooth hypersurface of degree $d\geq 3$. Then for $i$ such that $1\leq i\leq n+1$,  $x_i^{d-1}x_j\in F_d$ for some $j=j(i)$.
\end{proposition}

\begin{proof}
We may assume $i=1$. If otherwise, $F_d\in ( x_2,\cdots ,x_{n+1}) ^2$. Then by Lemma \ref{lem:nonsmooth} with $a=0, b=n$, $M$ is singular, a contradiction to smoothness of $M$.
\end{proof}

\begin{proposition}\label{pp:nonsmoothquintic}
Let $M$ be a hypersurface in $\mathbb{P}^4$ defined by a nonzero homogeneous polynomial $F_5$ of degree 5. Then $M$ is singular if one of the following three conditions is true:

(1) There exists $ 1\leq i\leq 5$, such that for all $ 1\leq j\leq 5,  x_i^4x_j\notin F_5$;

(2) There exists $ 1\leq p, q \leq 5, p\neq q,$  such that  $F_5\in ( x_p,x_q) $;

(3) There exist three distinct variables $x_i,x_j,x_k$, such that $F_5\in ( x_i) +( x_j,x_k) ^2$.

\end{proposition}

\begin{proof}
We check case by case:

(1)  We may assume $i=1$. Then for all  $1\leq j\leq 5, x_1^4x_j\notin F_5$. Then $M$ is singular by Proposition above.

(2) We may assume $F_d\in ( x_1,x_2) $. Then by Lemma \ref{lem:nonsmooth} with $a=2, b=0$, $M$ is singular.

(3) We may assume $F_d\in ( x_1) +( x_2,x_3) ^2$. Then by Lemma \ref{lem:nonsmooth} with $a=1, b=2$, $M$ is singular.

\end{proof}

Let $F=F(x_1,x_2,...,x_n)$ be a nonzero homogeneous polynomial. In general, it is an interesting but  difficult problem to determine all the projective linear automorphisms of the hypersurface $\{F=0\}\subset \mathbb{P}^{(n-1)}$. The differential method which we shall introduce below is a powerful general method especially when $F$ has a very few monomials like Fermat.

  Let $F=F(x_1,x_2,...,x_n)$ be a homogeneous polynomial of degree $d\geq 1$ in terms of variables $x_1,...,x_n$, and $G=G(y_1,y_2,...,y_n)$ be a homogeneous polynomial of degree $d$ in terms of variables $y_1,...,y_n$.

For $1\leq i\leq d$, we have the natural $i$-th order differential map naturally defined by $F$:

$$D_i^F:\mathcal{D}_i(x_1,...,x_n)\rightarrow \mathbb{C}[x_1,...,x_n]\,\, ,$$

where $\mathcal{D}_i(x_1,...,x_n)$ is the vector space of $i$-th order differential operators. For example, $D_1^F(\frac{\partial}{\partial x_j})=\frac{\partial F}{\partial x_j}$,  $D_2^F(\frac{\partial^2}{\partial x_i\partial x_j})=\frac{\partial^2 F}{\partial x_i\partial x_j}$, and so on. Obviously, $D_i^F$ is a $\mathbb{C}$-linear map, and we denote the dimension of the image space of $D_i^F$ as $\Rank(D_i^F)$.

The next Theorem, which we call the differential method, is quite effective.

\begin{theorem}\label{lem:differentialmethod}
If $F(x_1,...,x_n)=G(y_1,...,y_n)$ under an invertible linear change of coordinates, in other words, there exists an invertible matrix $L=(l_{ij})_{1\leq i, j\leq n}$, such that $F(x_1,...,x_n)=G(\displaystyle \sum_{i=1}^{n}l_{1i}x_i,...,\displaystyle \sum_{i=1}^{n}l_{ni}x_i)$, then $\Rank(D_i^F)=\Rank(D_i^G)$, for all $ 1\leq i\leq d$.
\end{theorem}

\begin{remark}
For example, if $F(x_1,x_2)=x_1^2+2x_1x_2+x_2^2$, $G(y_1,y_2)=y_1^2$, then the image space of $D_1^F$ spanned by $2x_1+2x_2$, and the image space of $D_1^G$ is spanned by $2y_1$. So  $\Rank(D_1^F)=\Rank(D_1^G)$. This equality also follows from Theorem \ref{lem:differentialmethod}, because $F(x_1,x_2)=G(x_1+x_2,x_2)$. Theorem \ref{lem:differentialmethod} is inspired by \cite{Po05}.
\end{remark}

\begin{proof}
The result simply follows from the linearality of the change of coordinates and the chain rule. To convince the readers, we shall give a detailed proof for $i=1$. Consider the following diagram,

\[\xymatrixcolsep{5pc}\xymatrix{
\mathcal{D}_1(x_1,...,x_n) \ar[d]^p \ar[r]^{D_1^F} &\mathbb{C}[x_1,...,x_n]\\
\mathcal{D}_1(y_1,...,y_n) \ar[r]^{D_1^G} &\mathbb{C}[y_1,...,y_n]\,\,\,\, ,\ar[u]^q}\]

where the map $p$ is given by chain rule, more explicitly, $$p(\frac{\partial}{\partial x_i})=\displaystyle \sum_{j=1}^{n}l_{ji}\frac{\partial}{\partial y_j}\, ,$$  and the map $q$ is induced by the linear change of coordinantes, more explicitly, if $H=H(y_i,...,y_n)$ is a polynomial in terms of $y_1,...,y_n$, then $$q(H)=H(\displaystyle \sum_{i=1}^{n}l_{1i}x_i,...,\displaystyle \sum_{i=1}^{n}l_{ni}x_i)\, .$$ This diagram commutes by definition and the chain rule. As both $p$ and $q$ are isomorphisms, it follows that $\Rank(D_1^F)=\Rank(D_1^G)$.
\end{proof}

\begin{definition}
Let $A\in \GL(n,\mathbb{C})$. We say $A$ is {\it semi-permutation} if $A$ is a diagonal matrix up to permutation of columns, or equivalently, $A$ has exactly $n$ nonzero entries. In literatures, such an $A$ is also called a generalized permutation matrix and a monomial matrix.
\end{definition}
  
  \begin{theorem}\label{thm:Aut(X)1-16}
  Suppose  $X$ is one of the Examples (1)-(16) in Example \ref{mainex}. Then the group $G$ in the  same example is the full automorphism group $\Aut(X)$ of $X$.
  \end{theorem}
  
  \begin{proof}
 By easy computation $G$ is the subgroup of $\Aut (X)$ generated by semi-permutation matrices preserving $X$. So in order to prove the theorem we are reduced to show that the full automorphism group $\Aut(X)$ is generated by semi-permutation matrices.
  
  We give a full proof when  $X$ is the Example (15) (Klein quintic threefold). If $X$ is one of the other 15 examples, then by similar arguments (the differential method) as below we can also show that the full automorphism group $\Aut(X)$ is generated by semi-permutation matrices.
  
  Suppose $L=(l_{ij})\in \GL(5,\mathbb{C})$ induces an automorphism of $X$. Denote change of coordinates: $y_i=\displaystyle \sum_{j=1}^{5}l_{ij}x_j$, for $1\leq i\leq 5$. Then we may assume

\begin{equation}\label{x=y}
  x_1^4x_2+x_2^4x_3+x_3^4x_4+x_4^4x_5+x_5^4x_1= y_1^4y_2+y_2^4y_3+y_3^4y_4+y_4^4y_5+y_5^4y_1.
\end{equation}

Then apply the operator $\frac{\partial}{\partial x_1}$ to both sides of the identity (\ref{x=y}), we get

\begin{equation}\label{dx1}
  (4x_1^3x_2+x_5^4)=(4y_1^3y_2+y_5^4)l_{11}+(4y_2^3y_3+y_1^4)l_{21}+(4y_3^3y_4+y_2^4)l_{31}+(4y_4^3y_5+y_3^4)l_{41}+(4y_5^3y_1+y_4^4)l_{51}
\end{equation}

Let us denote the polynomial on the right hand side of the identity (\ref{dx1}) as $h=h(y_1,y_2,y_3,y_4,y_5)$.

Notice that $\Rank(D_1^{ 4x_1^3x_2+x_5^4})=3$. We expect $\Rank(D_1^h)$ is also equal to 3.

\begin{lemma}\label{claim}
$\Rank(D_1^h)=3$ if and only if exactly one of the complex numbers $l_{11},l_{21},l_{31},l_{41},l_{51}$ is not equal to zero.
\end{lemma}

\begin{proof}
``if '' part is trivial. We only need to show ``only if'' part.

Suppose $\Rank(D_1^h)=3$. Then, at least one of $l_{i1}$ is not equal to zero. Without loss of generality, we may assume $l_{11}\neq 0$. Let us compute $\frac{\partial h}{\partial y_i}$ ( see Table \ref{table:partials}):

$\frac{\partial h}{\partial y_1}=12l_{11}y_1^2y_2+4l_{21}y_1^3+4l_{51}y_5^3$,

$\frac{\partial h}{\partial y_2}=4l_{11}y_1^3+12l_{21}y_2^2y_3+4l_{31}y_2^3$,

$\frac{\partial h}{\partial y_3}=4l_{21}y_2^3+12l_{31}y_3^2y_4+4l_{41}y_3^3$,

$\frac{\partial h}{\partial y_4}=4l_{31}y_3^3+12l_{41}y_4^2y_5+4l_{51}y_4^3$,

$\frac{\partial h}{\partial y_5}=4l_{11}y_5^3+4l_{41}y_4^3+12l_{51}y_5^2y_1$.

\begin{center}
\begin{longtable}{|p{1cm}|c|c|c|c|c|c|c|c|c|c|}
\caption{Matrix form of the coefficients of the partial derivatives}\label{table:partials}\\
 \hline 
 &$y_1^2y_2$&$y_1^3$&$y_5^3$&$y_2^2y_3$&$y_2^3$& $y_3^2y_4$&$y_3^3$&$y_4^2y_5$&$y_4^3$&$y_5^2y_1$\\
 \hline 
$\frac{\partial h}{\partial y_1}$& $12l_{11}$ & $4l_{21}$ &$4l_{51}$& & & & & & & \\
 \hline 
$\frac{\partial h}{\partial y_2}$&  & $4l_{11}$ & &$12l_{21}$ & $4l_{31}$ & & & & & \\
 \hline 
$\frac{\partial h}{\partial y_3}$&  & & & & $4l_{21}$ &$12l_{31}$ &$4l_{41}$ & & & \\
 \hline 
$\frac{\partial h}{\partial y_4}$&   &  & & & & &$4l_{31}$ &$12l_{41}$ &$4l_{51}$ & \\
 \hline 
$\frac{\partial h}{\partial y_5}$&  &   &$4l_{11}$ & & & & & &$4l_{41}$ & $12l_{51}$\\
 \hline

\end{longtable}%

\end{center}

As $\frac{\partial h}{\partial y_1},\frac{\partial h}{\partial y_2},\frac{\partial h}{\partial y_5}$ are linearly independent, $\Rank(D_1^h)\geq 3$. It is not hard to see that if one of $l_{21}, l_{31},l_{41},l_{51}$ is not also equal to zero, then $\Rank(D_1^h)\geq 4$. Therefore,  $l_{21}, l_{31},l_{41},l_{51}$ are all zero.
\end{proof}

By symmetry and Lemma \ref{claim}, it is clear that each column of the matrix $L$ has exactly one nonzero entry. Since $L$ is invertible, $L$ must be semi-permutation. Therefore, the full automorphism group $\Aut(X)$ is generated by semi-permutation matrices.

\end{proof}

  \begin{remark}
  Notice that the efficiency of the differential method depends on the dimension and the degree of the hypersurface in question. For example, let $X: x_1^2x_2+x_2^2x_3+x_3^2x_4+x_4^2x_5+x_5^2x_1=0$ be the Klein cubic threefold in $\P^4$. Then, $\Aut(X)$ is not generated by semi-permutation matrices by the main result of \cite{Ad78} based on Klein\rq{}s work. However, it turns out that for quintic threefolds, our differential method is quite useful. 
  \end{remark}

  \begin{remarks}\label{rmk:Autverysmall}
 Using the differential method, we can give explicit examples of smooth quintic threefolds such that the defining equations are of simple forms but the full automorphism groups are very small. We believe that these examples are of their interest.
 
 a) If $X:=\{x_1^5+x_2^4x_1+x_3^4x_2+x_4^4x_3+x_5^4x_4+x_5^5=0\}$ then $\Aut(X)$ is trivial (cf. \cite[Table~1]{Po05}).
 
 b) If $X:=\{x_1^4x_2+x_2^4x_3+x_3^4x_4+x_4^4x_5+x_5^5+x_1^2x_5^3=0\}$ then $\Aut(X)\cong C_2$ and $[A]\in \Aut(X)$, where $A=\Diag(-1,1,1,1,1)$
 
 c) If $X:=\{x_1^4x_2+x_2^4x_1+x_3^5+x_4^4x_3+x_5^4x_4+x_1^3x_4x_5=0\}$ then $\Aut(X)\cong C_3$ and $[A]\in \Aut(X)$, where $A:=\Diag(\xi_3,\xi_3^2,1,1,1)$.

 d) If $X:=\{x_1^5+x_2^4x_1+x_3^4x_2+x_4^4x_3+x_4^5+x_5^5=0\}$ then $\Aut(X)\cong C_5$ and $[A]\in \Aut(X)$, where $A:=\Diag(1,1,1,1,\xi_5)$.
 
 However,  there do not exist smooth quintic threefolds whose full automorphism group  $\Aut(X)$ are $C_{p}$ if $p$ is a prime larger than $5$ (see Theorem \ref{thm:containorder17}, Theorem \ref{thm:containorder41}, and Theorem \ref{thm:containorder13}).
 
\end{remarks}
  
  \section{Schur multiplier and liftablility of group actions}\label{ss:Schurandliftable}
  
  Let $G$ be a finite group. The Schur multiplier is by definition the second cohomology group $H^2(G, \mathbb{C}^*)$. We denote it by $M(G)$.
  
  \begin{theorem}\label{HSES}
(Hochschild-Serre exact sequence) (See, for example, \cite[ Chapter~2,~(7.29)]{Su82}) Let $H$ be a central subgroup of $G$. Consider the natural exact sequence: $$1\longrightarrow H\longrightarrow G\longrightarrow G/H\longrightarrow 1\, .$$ Then the induced sequence 
  
  $$1\rightarrow \Hom(G/H,\mathbb{C}^*)\rightarrow \Hom(G,\mathbb{C}^*)\rightarrow \Hom(H,\mathbb{C}^*)\rightarrow M(G/H)\rightarrow M(G)$$ is exact.
  \end{theorem}
  
  \begin{theorem}\label{SchurSylow}
 (See, for example, \cite[Chapter~2,~Corollary~3~to~Theorem~7.26]{Su82}) Let $p$ be a prime number. Let $P$ be a Sylow $p$-subgroup of $G$. Then the restriction map $M(G)\longrightarrow M(P)$ induces an injective homomorphism $M(G)_p\longrightarrow M(P)$.
  \end{theorem}

   \begin{notation} 
   Let $A=(a_{ij})\in \GL(n,\mathbb{C})$, and let $F\in \C[x_1,...,x_n]$ be a homogeneous polynomial of degree $d$. We denote by $A(F)$ the homogeneous polynomial $$F(\sum_{i=1}^{n}a_{1i}x_i,\cdots,\sum_{i=1}^{n}a_{ni}x_i)\, .$$

   \end{notation}
  

    \begin{definition}\label{def:leaveinvariant}
(1) We say $F$ is $A$-{\it invariant} if $A(F)=F$. In this case, we also say $A$ leaves $F$ invariant, or $F$ is invariant by $A$. We say $F$ is $A$-{\it semi-invariant} if $A(F)=\lambda F,$ for some $\lambda\in \mathbb{C}^*$.

  For example, let $A=\Diag (1,\xi_5,\xi_5^2,\xi_5^3,\xi_5^4)$ and $F=x_1^4x_2+x_2^4x_3+x_3^4x_4+x_4^4x_5+x_5^4x_1$, then $A(F)=\xi_5 F$, so $F$ is $A$-semi-invariant but not $A$-invariant.

  (2) Let $G$ be a finite subgroup of $\PGL(n,\mathbb{C})$. We say $F$ is $G$-{\it invariant} if for all $g\in G$, there exists $ A_g\in \GL(n,\mathbb{C})$ such that $g=[A_g]$ and $A_g(F)=F$.

  For example, let $X$ be a smooth quintic threefold defined by $F$ and $G$ is a finite subgroup of $\PGL(5,\mathbb{C})$. Then $F$ is $G$-invariant if and only if $G$ is a subgroup of $\Aut(X)$.

 (3) Let $G$ be a finite subgroup of $\PGL(n,\mathbb{C})$. We say a subgroup $\widetilde{G}< \GL(n,\mathbb{C})$ is a {\it lifting} of $G$ if $\widetilde{G}$ and $G$ are isomorphic via the natural projection $\pi: \GL(n,\C)\rightarrow \PGL(n,\C).$ We call $G$ {\it liftable} if $G$ admits a lifting.
 \end{definition}
  
  \begin{remark}
  Subgroups of $\PGL(n,\C)$ do not necessarily admit liftings to $\GL(n,\C)$.
  
  For example consider the dihedral group $D_8:=\langle a,b| a^4=b^2=1, b^{-1}ab=a^{-1} \rangle$. Then the map $$\rho(a)= \begin{pmatrix} 
    \xi_{8}&0 \\ 
    0&\xi_8^{-1} \\ 
   \end{pmatrix},\,\, \rho(b)= \begin{pmatrix} 
    0&\xi_4 \\ 
    \xi_4&0 \\ 
 \end{pmatrix}$$
 
 defines a projective representation $\rho: D_8\lra \PGL(2,\C)$. However, this is not induced by any linear representation $D_8\lra \GL(2,\C)$, that is, this $D_8$ does not admit a lifting (see e.g.\cite{Og05}).
  \end{remark}

  \begin{definition}
 (1) Let $G$ be a finite subgroup of $\PGL(n,\mathbb{C})$ and $F\in \C[x_1,...,x_n]$ be a homogeneous polynomial of degree $d$. We say $G$ is $F$-{\it liftable} if the following two conditions are satisfied:
  
  1) $G$ admits a lifting $\widetilde{G}< \GL(n,\mathbb{C})$; and
  
  2)  $A(F)=F$, for all $A$ in $\widetilde{G}$.

  In this case, we say $\widetilde{G}$ is an $F$-{\it lifting} of $G$.
  
  We say $G$ is $F$-{\it semi-liftable} if 2) is replaced by the following:
  
  2$)^\prime$ for all $A$ in $\widetilde{G}$, $A(F)=\lambda_A F$, for some $\lambda_A\in \mathbb{C}^*$ (depending on $A$).

(2)  Let $h$ be an element in $\PGL(n,\mathbb{C})$ of finite order. As a special case, we say $H\in \GL(n,\mathbb{C})$ is an $F$-{\it lifting} of $h$ if $\pi (H)=h$ and the group $\langle H\rangle $ is an $F$-lifting of the group $\langle h\rangle .$ 
  \end{definition}
  
 \begin{example}
   a) Let $F=x_1^5+x_2^5+x_3^5+x_4^5+x_5^5$ and let $G$ be the subgroup of $\PGL(5,\mathbb{C})$ generated by $[A_1]$ and $[A_2]$, where  $A_1=\Diag(\xi_5,1,1,1,1),A_2=\Diag(1,\xi_5,1,1,1)$. Then, clearly, $G\cong C_5^2$ and $G$ is $F$-liftable. 
  
  b) Let $F=x_1^4x_2+x_2^4x_3+x_3^4x_4+x_4^4x_5+x_5^4x_1$  and let $G$ be the subgroup of $\PGL(5,\mathbb{C})$ generated by $[A_1]$ and $[A_2]$, where $$A_1=
    \begin{pmatrix} 
    0&1&0&0&0 \\ 
    0&0&1&0&0 \\ 
    0&0&0&1&0 \\ 
    0&0&0&0&1 \\ 
    1&0&0&0&0 \\ 
    \end{pmatrix},\,\, A_2=
    \begin{pmatrix} 
    1&0&0&0&0 \\ 
    0&\xi_5&0&0&0 \\ 
    0&0&\xi_5^2&0&0 \\ 
    0&0&0&\xi_5^3&0 \\ 
    0&0&0&0&\xi_5^4 \\ 
    \end{pmatrix}.$$ Then $F$ is $G$-invariant, but $G$ is not $F$-semi-liftable (in fact, $G$ is not liftable).

\end{example}

    \begin{theorem}\label{liftable}
  Let $G$ be a finite subgroup of $\PGL(n,\mathbb{C})$. Let $F\in \C[x_1,...,x_n]$ be a nonzero homogeneous polynomial of degree $p$, where $p$ is a prime number. Suppose $F$ is $G$-invariant. Let $G_p$ be a Sylow $p$-subgroup. Then $G$ is $F$-liftable if the following two conditions are satisfied:
  
    (1) $G_p$ is $F$-liftable; and

    (2) either $G_p$ has no element of order $p^2$ or $G$ has no normal subgroup of index $p.$  
  \end{theorem}
  
  \begin{proof}
  Suppose both (1) and (2) are satisfied.
  
  Let $\widetilde{G_p}$ be an $F$-lifting of $G_p$. Define $H=\{ g\in G | \text{order of } g \text{ coprime to } p  \}$. Let $h\in H$ and $m$ be the order of $h.$ Then there exists a unique $F$-lifting of $h$. In fact, for any $A_h\in \pi^{-1}(h)$, there is $\lambda \in \C^{*}$ such that $A_h^m=\lambda I_n $.  Let $\alpha$ be a complex number such that $\alpha^m=\frac{1}{\lambda}$. Then $(\alpha A_h)^m=I_n$. Replacing $A_h$ by $\alpha A_h$, we obtain $A_h^m=I_n$. 
  
  Since $F$ is $G$-invariant we have $A_h(F)=\lambda^{\prime}F$ for some $\lambda^{\prime}\in \mathbb{C}^*$. Since $F=(A_h^m)(F)=(\lambda^{\prime})^mF$ we have $(\lambda^{\prime})^m=1$. So $\lambda^{\prime}=\xi_m^j$ for some $1\leq j\leq m$. Since $m$ and $p$ are coprime so there exists $i\in \mathbb{Z}$, such that $pi+j\equiv 0\,\,(\Mod m)$. Then $(\xi_m^i A_h)(F)=\xi_m^{pi}\xi_m^jF=\xi_m^{pi+j}F=F$. So $\xi_m^i A_h$ is an $F$-lifting of $h$. Hence an $F$-lifting of $h$ exists. For a similar reason, $h$ has a unique $F$-lifting.
  
 We define $\widetilde{H}=\{A_h\in \GL(n,\mathbb{C})|h\in H, A_h \text{ is the unique } F \text{-lifting of } h \}$. Let $\widetilde{G}$ be the subgroup of $\GL(n,\mathbb{C})$ generated by $\widetilde{G_p}$ and $\widetilde{H}$.
  
  We have the exact sequence $$1\rightarrow \Ker \pi|_{\widetilde{G}}\rightarrow \widetilde{G}\xrightarrow{\pi|_{\widetilde{G}}}G\rightarrow 1.$$
  
 Here $ \Ker \pi|_{\widetilde{G}}$ is either trivial or is generated by $\Diag(\xi_p,\cdots,\xi_p)$.
  
  If $\Ker \pi|_{\widetilde{G}}=1$, then $\widetilde{G}$ is an $F$-lifting of $G$ and hence we are done.
  
  Next we consider the case  when $\Ker  \pi|_{\widetilde{G}}$ is generated by  $\Diag(\xi_p,\cdots,\xi_p)$. We have the following commutative diagram:

  \xymatrix{
1\ar[r]\ar[d]& 1 \ar [d] \ar[r] & \widetilde{G_p}\ar[d]\ar[r] &G_p\ar[d]\ar[r]&1\ar[d] \\
1\ar[r]&\Ker \pi|_{\widetilde{G}} \ar[r]& \widetilde{G}\ar[r] &G\ar[r]&1\, , }

where the vertical maps are natural inclusion maps.  Applying Theorem \ref{HSES}, we get the  commutative diagram:

    \xymatrix{
1\ar[r]&\Hom(G_p,\mathbb{C}^*) \ar[r] & \Hom(\widetilde{G_p},\mathbb{C}^*)\ar[r] &1\ar[r]&M(G_p)\\
1\ar[r]\ar[u]&\Hom(G,\mathbb{C}^*)\ar[u]\ar[r]& \Hom(\widetilde{G},\mathbb{C}^*)\ar[u]\ar[r]^-{\Res} &\Hom(\Ker \pi|_{\widetilde{G}},\mathbb{C}^*)\ar[u]\ar[r]^-{\Tra}&M(G)\ar[u]^{\phi} }

By Theorem \ref{SchurSylow}, the map $\phi|M(G)_p: M(G)_p\rightarrow M(G_p)$ is injective. Since $\Hom(\Ker \pi|_{\widetilde{G}},\mathbb{C}^*)$ is isomorphic to $C_p$, the map $\Tra$ is the trivial map. So the map $\Res$ is surjective.

Therefore, there exists a homomorphism $f:\widetilde{G}\longrightarrow \mathbb{C}^*$ such that the restriction $f|_{\Ker \pi|_{\widetilde{G}}}$ is injective.

If $G_p$ has no element of order $p^2$, the Sylow $p$-subgroup of $f(\widetilde{G})$ is the group consisting of the $p$-th roots of unity. So, there exists a surjective homomorphism $\alpha: f(\widetilde{G})\longrightarrow C_p$. Then $\Ker(\alpha \circ f)$ is an $F$-lifting of $G$.

When $G$ has no normal subgroup of index $p$,  we take any surjective homomorphism $\beta : f(\widetilde{G})\longrightarrow C_p $. Then $\Ker (\beta\circ f)$ is an $F$-lifting of $G$.
  
      \end{proof}
 
 {\it In the rest of this section, let $X\subset \P^4$ be a smooth quintic threefold defined by $F$.}

 \begin{theorem}\label{thm:noorder25}
Assume that $g\in \Aut(X)$  and $\Ord(g)=5^n$ ($n\geq 1$). Then $n=1$, i.e. $g$ is necessarily of order $5$.
\end{theorem}

\begin{proof}
It suffices to show that there is no $g\in \Aut(X)$ such that $\Ord (g)=5^2$.

Assume to the contrary that there is $g\in \Aut(X)$ such that $\Ord (g)=5^2$.

Without loss of generality we may assume $g=[A]$ and $A=\Diag (\xi_{25}^{a_1},\xi_{25}^{a_2},\xi_{25}^{a_3},\xi_{25}^{a_4},\xi_{25}^{a_5})$, and $F$ is $A$-semi-invariant.

There are two possible cases: 1) $x_i^5\in F$ for some $i$; 2) $x_i^5\notin F$ for all $i$.

Case 1): We may assume $x_1^5\in F$. Replacing $A$ by $\xi_{25}^{-a_1}A$, we may assume $A=\Diag (1,\xi_{25}^{a_2},\xi_{25}^{a_3},\xi_{25}^{a_4},\xi_{25}^{a_5})$. Then $A(F)=F$. Clearly $A$ must have order $25$. We may assume $a_2=1$ and hence $A=\Diag (1,\xi_{25},\xi_{25}^{a_3},\xi_{25}^{a_4},\xi_{25}^{a_5})$.

By Proposition \ref{pp:smoothnessandmonomial}, $x_2^4x_j\in F$ for some $j$. But both $x_2^4x_1$ and $x_2^5$ can not be in $F$ since $A(x_2^4x_1)\neq x_2^4x_1$ and $A(x_2^5)\neq x_2^5$. Then we may assume $x_2^4x_3$ is in $F$. Then $A(x_2^4x_3)=x_2^4x_3$ implies $a_3\equiv -4(\Mod 25)$.

  Similarly,  we may successively assume $x_3^4x_4\in F$,  $x_4^4x_5\in F$. Then $A=\Diag (1,\xi_{25},\xi_{25}^{-4},\xi_{25}^{16},\xi_{25}^{11})$.

Hence, then $x_5^4x_i\notin F$ for all $1\leq i\leq 5$, a contradiction by Proposition \ref{pp:smoothnessandmonomial}. 

So the case 1) is impossible.

Case 2) We shall see the case 2) is impossible by argue by contradiction. We may assume $x_1^4x_2\in F$ and $A=(1,\xi_{25}^{a_2},\xi_{25}^{a_3},\xi_{25}^{a_4},\xi_{25}^{a_5}).$ Then there are three possibilities:

 Case 2)-i)  The number $a_2$ is coprime to $5$. Then we may assume $A=(1,\xi_{25},\xi_{25}^{a_3},\xi_{25}^{a_4},\xi_{25}^{a_5})$ and then $A(F)=\xi_{25}F$. 
 
As in case 1) above, using $A(F)=\xi_{25}F$ and Proposition \ref{pp:smoothnessandmonomial}, we may successively assume $x_2^4x_3\in F$, $x_3^4x_4\in F$,  $x_4^4x_5\in F$. Then $A=\Diag (1,\xi_{25},\xi_{25}^{-3},\xi_{25}^{13},\xi_{25}^{-1})$.  Then $x_5^4x_i\notin F$ for all $i$, a contradiction by Proposition \ref{pp:smoothnessandmonomial}. So the case $a_2$ is coprime to $5$ is impossible.

Case 2)-ii) $a_2$ is divided by $5$ but not by $25$. Then we may assume $a_2=5$ and hence $A=(1,\xi_{25}^{5},\xi_{25}^{a_3},\xi_{25}^{a_4},\xi_{25}^{a_5})$. Then $A(F)=\xi_{25}^5F$. Using similar arguments to case 2)-i), we see that this case is also impossible.

Case 2)-iii) $a_2$ is divided by $25$. Then $A=(1,1,\xi_{25}^{a_3},\xi_{25}^{a_4},\xi_{25}^{a_5})$. So, we may assume $a_3=1$ and hence $A=(1,1,\xi_{25},\xi_{25}^{a_4},\xi_{25}^{a_5})$. We have $A(F)=F$ since $A(x_1^4x_2)=x_1^4x_2$. Using similar arguments to the case 2)-i), we see that this case is also impossible.

Therefore, the case 2) is impossible.

\end{proof}
 
\begin{remark}
By Theorem \ref{thm:noorder25}, if $G\subset \PGL(5,\C)$ is a subgroup of $\Aut(X)$, the condition (2) in Theorem \ref{liftable} is always satisfied and hence $G$ is $F$-liftable if and only if $G_5$ is $F$-liftable.
\end{remark}

\begin{lemma}\label{lem:matrixshape}
Let $A=\Diag(a_1,...,a_n)$ and $A\rq{}=\Diag(a_1\rq{},...,a_n\rq{})$ be two diagonal $n\times n$ matrices. Suppose $B:=(b_{ij})$ is also an $n\times n$ matrix and $BA=A\rq{}B$. Then $b_{ij}=0$ if $a_j\neq a_i\rq{}$.

\end{lemma}

\begin{proof}
By an easy computation, $BA=(a_jb_{ij})$ and $A\rq{}B=(a_i\rq{}b_{ij})$. Then $a_jb_{ij}=a_i\rq{}b_{ij}$ because $BA=A\rq{}B$. Hence $b_{ij}=0$ whenever $a_j\neq a_i\rq{}$.
\end{proof}

\begin{remark}

The above simple fact is very helpful to determine the \lq\lq{}shape\rq\rq{} of the matrices we will consider, and it will be used frequently (either explicitly or implicitly) in the rest of this paper. 

\end{remark}

The next two lemmas tell us a very useful fact:  if a group $G \subset \PGL(5,\C)$ is a subgroup of $\Aut(X)$, $G$ is isomorphic to $C_5$ or $C_5^2$, and $G$ is not $F$-liftable, then $G$ must be generated by \lq\lq{}very special\rq\rq{} matrices. 

\begin{lemma}\label{lem:5notliftable}
 Let $g\in \Aut(X)$ with $\Ord (g)=5$. The group $\langle g\rangle $ is not $F$-liftable if and only if, up to change of coordinates, $g=[A]$ and $A(F)=\xi_5 F$ and $x_1^4x_2,x_2^4x_3,x_3^4x_4,x_4^4x_5,x_5^4x_1\in F$, where  $A=\Diag(1,\xi_5,\xi_5^2,\xi_5^3,\xi_5^4)$.
\end{lemma}

\begin{proof}
We may assume $g=[A],$ where $A:=\Diag(1,\xi_5^a,\xi_5^b,\xi_5^c,\xi_5^d)$ with $0\leq a,b,c,d\leq 4$.

Since $A(F)\neq F$, we have $x_1^5\notin F$. By the smoothness of $F$, we may assume $x_1^4x_2\in F$. 

Then $a\neq 0$. We may assume $a=1$. Then $A(F)=\xi_5 F$ as $A(x_1^4x_2)=\xi_5 x_1^4x_2$.

Then by the smoothness of $F$ and Proposition \ref{pp:nonsmoothquintic} (1), we may successively assume $x_2^4x_3,x_3^4x_4,x_4^4x_5\in F$. Then $A(F)=\xi_5F$ implies $b=2$, $c=3$, $d=4$. Furthermore, the smoothness of $F$ implies $x_5^4x_1\in F$. The lemma is proved.
\end{proof}

\begin{lemma}\label{lem:5^2notliftable}

Suppose $C_5^2\cong N< \Aut(X)$. The group $N$ is not $F$-liftable if and only if, up to change of coordinates,  $N$ is generated by $[A_1]$ and $[A_2]$, where $A_1=\Diag(1,\xi_5, \xi_5^2, \xi_5^3, \xi_5^4 )$ and $$A_2=
    \begin{pmatrix} 
    0&1&0&0&0 \\ 
    0&0&1&0&0 \\ 
    0&0&0&1&0 \\ 
    0&0&0&0&1 \\ 
    1&0&0&0&0 \\ 
    \end{pmatrix} \, .
   $$

\begin{proof}
\lq\lq{}If\rq\rq{} part is clear. We  prove \lq\lq{}only if\rq\rq{} part.

Suppose $N$ is not $F$-liftable. We may assume $N=\langle [A_1], [A_2]\rangle $ and $A_1^5=A_2^5=I_5$. Since $N$ is abelian, $A_2A_1=\xi_5^k A_1A_2$ for some $0\leq k\leq 4$.

Case (i) $k=0$. Then $A_2A_1=A_1A_2$. Therefore, $A_1$ and $A_2$ can be diagonalized simultaneously under suitable change of coordinates. So we may assume $A_1=\Diag(1,\xi_5,1,\xi_5^a,\xi_5^b)$, and $A_2=\Diag(1,1,\xi_5,\xi_5^c,\xi_5^d)$. Then by Lemma \ref{lem:5notliftable}, we must have $A_1(F)=A_2(F)=F$. However, then $\langle A_1,A_2\rangle$ is an $F$-lifting of $N$, a contradiction to our  assumption. So $k\neq0$.

Case(ii) $k\neq 0$. Replacing $A_2$ by $A_2^j$ for suitable $j$ if necessary, we may assume $k=1$. Then $A_2A_1A_2^{-1}=\xi_5A_1$. We may assume $A_1=\Diag(1,\xi_5, \xi_5^a, \xi_5^b, \xi_5^c)$, where $0\leq a\leq b\leq c\leq 4$. The matrices $A_1$ and $\xi_5 A_1$ must have the same characteristic polynomials. Therefore,$$(t-1)(t-\xi_5)(t-\xi_5^a)(t-\xi_5^b)(t-\xi_5^c)=(t-\xi_5)(t-\xi_5^{2})(t-\xi_5^{a+1})(t-\xi_5^{b+1})(t-\xi_5^{c+1}).$$  This implies $a=2,b=3$ and $c=4$, i.e., $A_1=\Diag(1,\xi_5, \xi_5^2, \xi_5^3, \xi_5^4).$

By the identity $A_2A_1=\xi_5A_1A_2$ and Lemma \ref{lem:matrixshape}, we may assume 

$$A_2=
    \begin{pmatrix} 
    0&a_1&0&0&0 \\ 
    0&0&a_2&0&0 \\ 
    0&0&0&a_3&0 \\ 
    0&0&0&0&a_4 \\ 
    a_5&0&0&0&0 \\ 
    \end{pmatrix} $$
    
 for some $a_1,a_2,a_3,a_4,a_5$. Here  $a_1a_2a_3a_4a_5=1$ as $A_2^5=I_5$.   Then after changing of coordinates $x_1^{\prime}=x_1,x_2^{\prime}=a_1x_2,x_3^{\prime}=a_1a_2x_3,x_4^{\prime}=a_1a_2a_3x_4,x_5^{\prime}=a_1a_2a_3a_4x_5$, we have  $A_1=\Diag(1,\xi_5, \xi_5^2, \xi_5^3, \xi_5^4 )$ and $$A_2= \begin{pmatrix} 
    0&1&0&0&0 \\ 
    0&0&1&0&0 \\ 
    0&0&0&1&0 \\ 
    0&0&0&0&1 \\ 
    1&0&0&0&0 \\ 
    \end{pmatrix}.$$

\end{proof}

\end{lemma}

In many cases, subgroups of small order imply $F$-liftability of $G$.

\begin{lemma}\label{10or50liftable}
Let $G< \Aut(X)$ and $|G|=5^nq$, where $q$ and $5$ are coprime, and  $n=1$ or $2$. If $G$ contains a subgroup of order $2\cdot 5^n$, then the group $G$ is $F$-liftable.
\end{lemma}

\begin{proof}

Suppose $H< G$ and $|H|=2\cdot 5^n$.

By Theorem \ref{liftable}, it suffices to show that  $H$ is $F$-liftable.

If $n=1$, then  $H\cong C_{10}$ or $D_{10}$. 

Suppose $H\cong C_{10}$ and $H$ is not $F$-liftable. By Lemma \ref{lem:5notliftable}, we may assume there exists $g\in H$ such that $\Ord(g)=5$, $g=[A]$, $A(F)=\xi_5 F$ and $x_1^4x_2,x_2^4x_3,x_3^4x_4,x_4^4x_5,x_5^4x_1\in F$. Here $A=\Diag(1,\xi_5,\xi_5^2,\xi_5^3,\xi_5^4)$. Let $h\in H$ with order $2$ then we can find $B\in \GL(5,\mathbb{C})$ such that  $h=[B]$, $B(F)=F$ and $\Ord (B)=\Ord (h)=2$.

Since $gh=hg$, it follows that $AB=\lambda BA$ for some nonzero constant $\lambda$. By considering eigenvalues, we find $\lambda=1$. In fact, $ABA^{-1}=\lambda B$ implies $\lambda^2=1$, and $B^{-1}AB=\lambda A$ implies $\lambda^5=1$, and hence $\lambda=1$. Then $AB=BA$. By Lemma \ref{lem:matrixshape},  $B$ must be a diagonal matrix. So $B$=$\Diag(\pm 1$, $\pm 1$, $\pm 1$, $\pm 1$, $\pm 1)$. Since $x_1^4x_2,x_2^4x_3,x_3^4x_4,x_4^4x_5,x_5^4x_1\in F$ and $B(F)=F$, we have $B=\Diag(1,1,1,1,1)$, which is absurd.

So $H$ must be $F$-liftable if $H\cong C_{10}$.

Similarly, if $H\cong D_{10}$ or $n=2$, then  $H$ is also $F$-liftable. The arguments are slightly more involved (e.g. use Lemma \ref{lem:5^2notliftable} in the case $n=2$), but the idea and computations are quite similar. So we may leave detailed proofs for the readers.
\end{proof}
  
  \section{Sylow $p$-subgroups of $\Aut(X)$ if $p\neq 2,5$}\label{ss:primeorders}
  
  Our main results of this section are Theorems \ref{thm:Sylow3}, \ref{thm:Sylow13}, \ref{Sylow17}, \ref{thm:Sylow41}, \ref{thm:containorder17}, \ref{thm:containorder41} and \ref{thm:containorder13}.
  
      \subsection{All possible prime factors of $|\Aut(X)|$}
      
      \begin{theorem}\label{thm:primaryorders}
      (\cite[Theorem~1.3]{GL13}) Let  $n\geq 1$ and $d\geq 3$ be integers, and $(n,d)\neq (1,3),(2,4)$. Let $q$ be a primary number, i.e., $q=p^k$ for some prime $p$, such that $\Gcd (q,d)=\Gcd(q,d-1)=1$. Then $q$  is the order of an automorphism of some smooth hypersurface of dimension $n$ and degree $d$ if and only if there exists $l\in \{1,...,n+2\}$ such that $$(1-d)^{l}\equiv 1 \;\Mod q. $$
      \end{theorem}

      \begin{proof}
We include the proof here for the reader\rq{}s convenience (cf. \cite[Theorem~1.3]{GL13} and the proof there).     
   
   To prove the \lq\lq{}only if\rq\rq{} part, suppose $F\in \C[x_1,...,x_{n+2}]$ is a homogeneous polynomial of degree $d$ such that the hypersurface $X:=\{F=0\}\subset \P^{n+1}$ is smooth and admits an automorphism $\varphi$ of order $q$, with $\Gcd(q,d)=\Gcd(q,d-1)=1$. Without loss of generality, we may assume $\vp=[A]$, where $A=\Diag(\xi_q^{\s_1},...,\xi_q^{\s_{n+2}})$, $0\leq \s_i\leq q-1$, for all $ 1\leq i\leq n+2$.
     
     We have $A(F)=\xi_q^a F$. Let $b$ be an integer such that $db\equiv-a\;\Mod q$. Such a $b$ always exists as $\Gcd(q,d)=1$. Then $(\xi_q^bA)(F)=\xi_q^{db+a}F=F$. So replacing $A$ by $\xi_q^b A$ if necessary, we may assume $A(F)=F$. Now choose an index $k_1$ such that $\Gcd(\s_{k_1},q)=1$. By smoothness of $F$ (see Proposition \ref{pp:smoothnessandmonomial}), $x_{k_1}^{d-1}x_{k_2}\in F$ for some $k_2\in \{1,...,n+2\}$. Because of $A(F)=F$, we have $A(x_{k_1}^{d-1}x_{k_2})=x_{k_1}^{d-1}x_{k_2}$ so that
     
     \begin{equation}\label{eq:modq}
     \s_{k_2}\equiv (1-d)\s_{k_1} \;\Mod q.
     \end{equation}
     
     Furthermore, since $\Gcd(q,d-1)=1$, we have $\s_{k_2}\neq 0\;\Mod q$, and since $\Gcd(q,d)=1$ we have $k_2\neq k_1$.
     
     Applying the above argument with $k_1$ replaced by $k_2$, we let $k_3$ be an index such that $A(x_{k_2}^{d-1}x_{k_3})=x_{k_2}^{d-1}x_{k_3}$ and $x_{k_2}^{d-1}x_{k_3}\in F$. Iterating this process, for all $i\in \{4,...,n+3\}$ we find $k_i\in \{1,...,n+2\}$ such that $A(x_{k_{i-1}}^{d-1}x_{k_i})=x_{k_{i-1}}^{d-1}x_{k_i}$ and  $x_{k_{i-1}}^{d-1}x_{k_i}\in F$.
     
     By the equation (\ref{eq:modq}), we have $$ \text{ for all\, } i\in \{3,...,n+3\}, \s_{k_i}\equiv (1-d)\s_{k_{i-1}}\equiv (1-d)^2\s_{k_{i-2}}\equiv (1-d)^{i-1}\s_{k_1}\; \Mod q,$$ and all of the $\s_{k_i}$ are non-zero.
     
     Since $k_i\in\{1,...,n+2\}$, there are at least two $i,j\in \{1,...,n+3\}$, $i>j$ such that $k_i=k_j$. Thus $\s_{k_i}=\s_{k_j}$, and since $\s_{k_i}\equiv (1-d)^{i-1}\s_{k_1}\;\Mod q$ and $\s_{k_j}\equiv (1-d)^{j-1}\s_{k_1}\;\Mod q$, we have $(1-d)^{i-1}\s_{k_1}\equiv (1-d)^{j-1}\s_{k_1}\; \Mod q$. Then $(1-d)^{i-j}\equiv 1 \;\Mod q$ as $\Gcd(1-d,q)=\Gcd(\s_{k_1},q)=1$. This finishes the proof of  \lq\lq{}only if\rq\rq{} part.
     
     To prove \lq\lq{}if\rq\rq{} part, let $q$ be a positive integer such that $\Gcd(q,d)=\Gcd(q,d-1)=1$, and assume that there exists $l\in \{1,...,n+2\}$ such that $(1-d)^l\equiv 1\;\Mod q$. We let $F\in \C[x_1,...,x_{n+2}]$ be the homogeneous polynomial $F=\displaystyle \sum_{i=1}^{l-1}x_i^{d-1}x_{i+1}+\displaystyle\sum_{i=l+1}^{n+2}x_i^d$. By construction, the hypersurface $X:=\{F=0\}\subset \P^{n+1}$ admits the automorphism $[A]$, where $A=\Diag(\xi_q,\xi_q^{1-d},...,\xi_q^{(1-d)^{l-1}},1,...,1)$. One can check the  smoothness of $X$ by the Jacobian criterion (cf \cite[ Example~3.5]{GL13}).

      \end{proof}
      
    {\it  In the rest of this section, let $X\subset \P^4$ be a smooth quintic threefold defined by $F$.}
      \medskip
  
  Let us consider the numbers $(1-5)^l-1$, for $1\leq l \leq 5$. These five numbers are $-5,15,-65,255,-1025$. Since $15=3\cdot 5$, $65=5\cdot 13$, $255=3\cdot 5\cdot 17$, and $1025=5^2\cdot 41$, all possible primary orders of elements in  $\Aut(X)$ are $2^a,3,5^b,13,17$ and $41$ by Theorem \ref{thm:primaryorders}.

  \begin{lemma}\label{lem:primeorder}
   Let $g\in \Aut(X)$. Suppose $\Ord (g)=p^a$, where $a>0$ and  $p=2,3,13,17$ or $41$. Then under suitable change of coordinates, we may assume $g=[A]$, where $A:=\Diag(\xi_{p^a},\xi_{p^a}^{b_1},\xi_{p^a}^{b_2},\xi_{p^a}^{b_3},\xi_{p^a}^{b_4})$  ($b_1,...,b_4$ are integers) and $A(F)=F$.
  \end{lemma}
  
  \begin{proof}
  By Theorem \ref{liftable}, $g$ has an $F$-lifting, say $A$. Then $g=[A], \Ord (A)=\Ord (g)=p^a$ and $A(F)=F$. Clearly by suitable change of coordinates if necessary, we may assume $A=\Diag(\xi_{p^a},\xi_{p^a}^{b_1},\xi_{p^a}^{b_2},\xi_{p^a}^{b_3},\xi_{p^a}^{b_4})$.
  \end{proof}

  \begin{lemma}\label{lem:em3}
Suppose $[A]\in \Aut(X)$ and $A(F)=F$. If $\xi$ is an eigenvalue of $A$ with multiplicity $\geq 3$, then $\xi^5=1$.
\end{lemma}

\begin{proof}
Suppose $\xi$ is an eigenvalue of $A$ with multiplicity $\geq 3$. We may assume $A=\Diag(\xi,\xi,\xi,\alpha,\beta)$. Since $X$ is smooth and defined by $F$, $F$ can not be written as $x_4 H+x_5 G,$ for some $H$ and $G$. Then there exists monomial $x_1^{i_1}x_2^{i_2}x_3^{i_3}\in F$, $i_1+i_2+i_3=5$, $i_j\geq 0$. Since $A(F)=F$, it follows that $A(x_1^{i_1}x_2^{i_2}x_3^{i_3})=\xi^5 x_1^{i_1}x_2^{i_2}x_3^{i_3}=x_1^{i_1}x_2^{i_2}x_3^{i_3}$, so $\xi^5=1$.
\end{proof}

\subsection{Sylow $p$-subgroups of $\Aut(X)$ for $p=3,13,17$ or $41$}

\begin{lemma}\label{order3power}
Suppose $g\in \Aut(X)$ and $\Ord (g)=3$. Then we may assume $g=[A]$ and $A(F)=F$, where $A=\Diag(\xi_3,\xi_3^2,1,1,1)$ or $\Diag(\xi_3,\xi_3^2,\xi_3,\xi_3^2,1)$.
\end{lemma}

\begin{proof}
By Lemma \ref{lem:primeorder},  we may assume $g=[A]$, $A=\Diag(\xi_3,\xi_3^{b_1},\xi_3^{b_2},\xi_3^{b_3},\xi_3^{b_4})$ and $A(F)=F$. Then note that $X$ is smooth. So, by Proposition \ref{pp:nonsmoothquintic} (1) and $A(x_1^5)\neq x_1^5$, we may assume $x_1^4x_2$ in $F$ and hence $b_1\equiv 2 (\Mod 3).$ Then $A=\Diag(\xi_3,\xi_3^2,\xi_3^{b_2},\xi_3^{b_3},\xi_3^{b_4})$.

If $A=\Diag(\xi_3,\xi_3^2,\xi_3,1,1)$, then $F\in ( x_2) +( x_4,x_5) ^2$, a contradiction to the smoothness of $X$ by  Proposition \ref{pp:nonsmoothquintic} (3). So $A\neq\Diag(\xi_3,\xi_3^2,\xi_3,1,1)$. Similarly, $A\neq\Diag(\xi_3,\xi_3^2,\xi_3^2,1,1)$.

By Lemma \ref{lem:em3}, the multiplicities of $\xi_3$ and $\xi_3^2$ as eigenvalues of $A$ are less than or equal to two.

Therefore, we may assume  $A=\Diag(\xi_3,\xi_3^2,1,1,1)$ or $\Diag(\xi_3,\xi_3^2,\xi_3,\xi_3^2,1)$.
\end{proof}

\begin{lemma}\label{C_3^3}
The group $C_3^3$ can not be a subgroup of $\Aut(X)$.
\end{lemma}

\begin{proof}
Assume to the contrary that there exists $G< \Aut(X)$ such that $G\cong C_3^3$. By Theorem \ref{liftable}, $G$ has an $F$-lifting, say $\widetilde{G}$. Then by Lemma \ref{order3power} we may assume that $\widetilde{G}=\langle A_1,A_2,A_3\rangle $, where $A_1=\Diag(\xi_3,\xi_3^2,\xi_3^{b_2},\xi_3^{b_3},\xi_3^{b_4})$, $x_1^4x_2\in F$ and $A_2$, $A_3$ are diagonal matrices whose eigenvalues are the third roots of unity.

Replacing $A_2$ by $A_1^iA_2$ for some $i$ if necessary, we may assume $A_2=\Diag(1,1,\xi_3^{c_2},\xi_3^{c_3},\xi_3^{c_4})$. 

Then by Lemma \ref{order3power} we may assume $A_2=\Diag(1,1,\xi_3,\xi_3^2,1)$ and hence $x_3^4x_4\in F$.

Again replacing $A_3$ by $A_1^iA_3$ for some $i$ if necessary, we may assume $A_3=\Diag(1,1,\xi_3^{d_2},\xi_3^{d_3},\xi_3^{d_4})$. 

Since $x_3^4x_4\in F$,  replacing $A_3$ by $A_2^iA_3$ for some $i$ if necessary, we may assume $A_3=\Diag(1,1,1,1,\xi_3^j)$, $j=1$ or $2$.

Then $A_3(x_5^4x_k)\neq x_5^4x_k$ for any $1\leq k\leq 5$. However, then $x_5^4x_k\notin F$ for all $k$, a contradiction to the smoothness of $X$ by Proposition \ref{pp:nonsmoothquintic} (1).
\end{proof}

\begin{remark}\label{rmk:abelian}
In general, if an abelian group acts on a smooth quintic threefold (assuming the action is $F$-liftable), then the matrices inducing the action  are often of very special kind. In fact, the proofs of Lemmas \ref{order3power} and \ref{C_3^3} tell us how to determine such matrices. First, we may assume those matrices are diagonal. Then by smoothness of $F$ and computation (using Mathematica), we can exclude many possibilities and only few possibilities for those matrices are left. 
\end{remark}

\begin{theorem}\label{thm:Sylow3}
Let $G$ be a subgroup of $\Aut(X)$ such that $|G|$ is divided by $3$. Then $G_3\cong C_3$ or $C_3^2$.
\end{theorem}

\begin{proof}
Clearly it suffices to show that there exists no groups of order 27 acting on smooth quintic threefolds.

Assume to the contrary that $H< G$ and $|H|=27$.

By the classification of groups of order 27, there are five different (up to isomorphism) groups of order 27: $C_{27}$, $C_9\times C_3$, $UT(3,3)$ (it is, by definition, the unitriangular matrix group of degree three over the field of three elements), $C_9\rtimes C_3$, $C_3^3$.

By Lemma \ref{order3power} $H$ is not isomorphic to $C_{27}$, $C_9\times C_3$ or $C_9\rtimes C_3$. By Lemma \ref{C_3^3}, $H$ is not isomorphic to $C_3^3$.

So $H$ must be isomorphic to $UT(3,3)$. Then, by Theorem \ref{liftable}, $H$ has an $F$-lifting, say $\widetilde{H}<\GL(5,\C)$. This subgroup $\widetilde{H}$ is a faithful five dimensional linear representation of $UT(3,3)$. By looking at the character table of $UT(3,3)$ (see e.g. GAP and Section \ref{ss:Sylow2} for more details about GAP), $\widetilde{H}$ contains a matrix which has $\xi_3$ as an eigenvalue of multiplicity $3$, a contradiction to Lemma \ref{lem:em3}.

%
%

\end{proof}

\begin{theorem}\label{thm:Sylow13}
Let $G< \Aut(X)$. Suppose $|G|$ is divided by $13$. Then $G_{13}\cong C_{13}$. Let $g$ be a generator of $G_{13}$. Then we may assume $g=[A], A(F)=F, A=\Diag (\xi_{13},\xi_{13}^{-4},\xi_{13}^3,1,1)$ and $$F=x_1^4x_2+x_2^4x_3+x_3^4x_1+x_1x_2x_3G(x_4,x_5)+H(x_4,x_5)$$ where $G$ and $H$ are of degree $2$ and $5$ respectively.
\end{theorem}

\begin{proof}
By Theorem \ref{thm:primaryorders}, $C_{13^2}$ can not act on $X$. In order to show $G_{13}\cong C_{13}$, it suffices to show that $C_{13}^2$ can not act on $X$.

Suppose $G$ has a subgroup, say $H$, isomorphic to $C_{13}^2$. By Theorem \ref{liftable}, $H$ has an $F$-lifting, say $\widetilde{H}$. We may assume $\widetilde{H}=\langle A_1,A_2\rangle $. So $A_1(F)=A_2(F)=F$ and $A_1A_2=A_2A_1$. We may assume that both $A_1$ and $A_2$ are diagonal and $A_1=\Diag(\xi_{13},\xi_{13}^{b_1},\xi_{13}^{b_2},\xi_{13}^{b_3},\xi_{13}^{b_4})$.

{\it Now we argue by the same way as in Lemmas \ref{order3power} and \ref{C_3^3}.}

By $A_1(F)=F$ we may assume $x_1^4x_2\in F$. Then $A_1(F)=F$ implies $b_1\equiv -4 (\Mod 13)$. So $A_1=\Diag(\xi_{13},\xi_{13}^{-4},\xi_{13}^{b_2},\xi_{13}^{b_3},\xi_{13}^{b_4})$.

Similarly we may assume $x_2^4x_3\in F$. Then $A_1=\Diag(\xi_{13},\xi_{13}^{-4},\xi_{13}^{3},\xi_{13}^{b_3},\xi_{13}^{b_4})$ and $b_3\equiv b_4\equiv 0(\Mod 13)$.

Therefore $A_1=\Diag(\xi_{13},\xi_{13}^{-4},\xi_{13}^{3},1,1)$. 

Note that a degree five monomial $x_1^{i_1}x_2^{i_2}x_3^{i_3}x_4^{i_4}x_5^{i_5}$ is left invariant by $A_1$ if and only if 

$$i_1-4i_2+3i_3\equiv 0 (\Mod 13).$$

Then by computing (e.g., by using Mathematica) invariant monomials of $A_1$, we have $$F=a_1x_1^4x_2+a_2x_2^4x_3+a_3x_3^4x_1+x_1x_2x_3G(x_4,x_5)+H(x_4,x_5).$$ 

By Proposition \ref{pp:smoothnessandmonomial}, $a_1a_2a_3\neq 0$. We may assume $a_1=1$. By adjusting variables $x_2$ and $x_3$ by nonzero multiples we may further assume $a_2=a_3=1$. Then replacing $A_2$ by $A_1^jA_2$ for some $j$ if necessary we may assume $A_2=\Diag (1,\xi_{13}^{c_1},\xi_{13}^{c_2},\xi_{13}^{c_3},\xi_{13}^{c_4})$.

By the expression of $F$ and $A_2(F)=F$, we must have $c_1\equiv c_2\equiv 0 (\Mod 13)$, i.e., $A_2=\Diag (1,1,1,\xi_{13}^{c_3},\xi_{13}^{c_4})$. Therefore, $C_{13}^2$ can not act on $X$. 

So $G_{13}\cong C_{13}$. Furthermore, by arguments above we may assume $G_{13}$ is generated by $[\Diag (\xi_{13},\xi_{13}^{-4},\xi_{13}^3,1,1)]$ and $$F=x_1^4x_2+x_2^4x_3+x_3^4x_1+x_1x_2x_3G(x_4,x_5)+H(x_4,x_5)$$ as wanted.
\end{proof}

\begin{theorem}\label{Sylow17}
Let $G<\Aut(X)$. Suppose $|G|$ is divided by $17$. Then $G_{17}\cong C_{17}$. Let $g$ be a generator of $G_{17}$. Then we may assume $g=[A], A(F)=F, A=\Diag (\xi_{17},\xi_{17}^{-4},\xi_{17}^{16},\xi_{17}^4,1)$ and $$F=x_1^4x_2+x_2^4x_3+x_3^4x_4+x_4^4x_1+x_5^5+ax_1x_2x_3x_4x_5+bx_1x_3x_5^3+cx_2x_4x_5^3+dx_1^2x_3^2x_5+ex_2^2x_4^2x_5,$$ where $a,b,c,d,e$ are complex numbers (possibly zero).
\end{theorem}

\begin{proof}
Similar to proof of Theorem \ref{thm:Sylow13}.
\end{proof}

\begin{theorem}\label{thm:Sylow41}
Let $G< \Aut(X)$. Suppose $|G|$ is divided by $41$. Then $G_{41}\cong C_{41}$. Let $g$ be a generator of $G_{41}$. Then we may assume $g=[A], A(F)=F, A=\Diag (\xi_{41},\xi_{41}^{-4},\xi_{41}^{16},\xi_{41}^{18},\xi_{41}^{10})$ and $$F=x_1^4x_2+x_2^4x_3+x_3^4x_4+x_4^4x_5+x_5^4x_1+ax_1x_2x_3x_4x_5,$$ where $a$ is a complex numbers (possibly zero).
\end{theorem}

\begin{proof}
Similar to proof of Theorem \ref{thm:Sylow13}.
\end{proof}

\subsection{Brauer\rq{}s Theorem and  $\Aut(X)$}

\begin{theorem}\label{thm:Brauer}
(\cite[Theorem~3]{Br42}) Let $p$ be a prime number and $G$ be a finite group such that $G_p\cong C_p$ and $G_p$ is not a normal subgroup. Then the degree of any faithful representation of $G$ is not smaller than $\frac{p-1}{2}$.
\end{theorem}

\begin{theorem}\label{thm:normalsylow}
Let $X$ be a smooth quintic threefold. Let $G$ be a subgroup of $\Aut(X)$. Suppose $|G|$ is divided by $p$, where $p$ is one of $41$, $17$, $13$. Then $G_p$ is a normal subgroup of $G$.
\end{theorem}

\begin{proof}
Our proof is inspired by \cite{Ad78}.

Since $p=13,17$, or $41$, we have $G_p\cong C_p$ by Theorems \ref{thm:Sylow13}, \ref{Sylow17} and \ref{thm:Sylow41}. 

The group $G$ is a finite subgroup of $\PGL(5,\mathbb{C})=\PSL(5,\mathbb{C})$. We have an exact sequence $$1\rightarrow Z\xrightarrow{i} \SL(5,\mathbb{C})\xrightarrow{\theta} \PSL(5,\mathbb{C})\rightarrow 1\, ,$$

where $Z$ denotes the center of $\SL(5,\mathbb{C})$, a group of order $5$. Denote by $\mathcal{G}$ the preimage $\theta^{-1}(G)$ of $G$ in $\SL(5,\mathbb{C})$ under $\theta.$ The order of $\mathcal{G}$ is $5\cdot |G|$. Let $\mathcal{G}_p$ be a Sylow $p$-subgroup of $\mathcal{G}$. Then $\mathcal{G}_p\cong C_p$ by our assumption of $p$.

Applying Theorem \ref{thm:Brauer}, $\mathcal{G}_p$ is normal in $\mathcal{G}$. In fact, $5<\frac{p-1}{2}$. So $\pi (\mathcal{G}_p)$ is normal in $G$, which means $G_p$ is normal in $G$.
\end{proof}

Using Theorem \ref{thm:normalsylow}, we shall study $\Aut(X)$ when $|\Aut(X)|$ is divided by $41,17$ or $13$ (Theorems \ref{thm:containorder17}, \ref{thm:containorder41} and \ref{thm:containorder13}).

\begin{theorem}\label{thm:containorder17}
 Suppose  $\Aut(X)$ contains an element of order $17$. Then $\Aut(X)$ is isomorphic to a subgroup of  the group appears in Example (13) in Example \ref{mainex}, and $|\Aut(X)|$ is divided by $2$.
\end{theorem}

\begin{proof}
By Theorem \ref{Sylow17}, $\Aut(X)_{17}\cong C_{17}$. Let $g$ be a generator of $\Aut(X)_{17}$. Then we may assume $g=[A], A(F)=F, A=\Diag (\xi_{17},\xi_{17}^{-4},\xi_{17}^{16},\xi_{17}^4,1)$ and $F$ is of the form (just compute invariant monomials of $A$ using Mathematica) $$F=x_1^4x_2+x_2^4x_3+x_3^4x_4+x_4^4x_1+x_5^5+ax_1x_2x_3x_4x_5+bx_1x_3x_5^3+cx_2x_4x_5^3+dx_1^2x_3^2x_5+ex_2^2x_4^2x_5,$$ where $a,b,c,d,e$ are complex numbers (possibly zero).

By Theorem \ref{thm:normalsylow}, $\Aut(X)_{17}$ is normal in $\Aut(X)$.

Let $h$ be an element of $\Aut(X)$. Suppose $h=[B]$ for some $B\in \GL(5,\mathbb{C})$. We can choose $B$ such that $B(F)=F$. By normality of $\Aut(X)_{17}$ we have $hgh^{-1}=g^k$ for some $k$. Then $BAB^{-1}=\lambda A^k$ for some nonzero constant $\lambda$. Clearly $F=(BAB^{-1})(F)=(\lambda A^k)(F)=\lambda^5 F$. Then $\lambda^5=1$. Since $A$ and $\lambda A^k$ must have the same set of eigenvalues we have $\lambda^{17}=1$. Then $\lambda=1$, i.e., $BAB^{-1}=A^k$. By the shape of $A$ (eigenvalues), $k\equiv 1,-4,16$ or $4 (\Mod 17)$. By $BA=A^kB$ and Lemma \ref{lem:matrixshape}, we can  show that $B$ must be a semi-permutation matrix.

Then by the shape of $F$ above,  $B$ must leave $F\rq{}:=x_1^4x_2+x_2^4x_3+x_3^4x_4+x_4^4x_1+x_5^5$  invariant. Let $G\subset \PGL(5,\C)$ be the group as in Example (13) in Example \ref{mainex}. Therefore, $\Aut(X)$ is a subgroup of $G$ by Theorem \ref{thm:Aut(X)1-16}. 

Notice that no matter what $a,b,c,d,e$ are, the following matrix can always act on $X$:

 $$ \begin{pmatrix} 
    0&0&1&0&0 \\ 
    0&0&0&1&0 \\ 
    1&0&0&0&0 \\ 
    0&1&0&0&0 \\ 
    0&0&0&0&1 \\ 
    \end{pmatrix}.$$

Therefore, $\Aut(X)$ contains an element of order two.
\end{proof}

\begin{theorem}\label{thm:containorder41}
  Suppose  $\Aut(X)$ contains an element of order $41$. Then $\Aut(X)$ is isomorphic to a subgroup of  the group appears in Example (15) in Example \ref{mainex}, and $|\Aut(X)|$ is divided by $5$.
\end{theorem}

\begin{proof}
The proof is the same as that of Theorem \ref{thm:containorder17}.
\end{proof}

\begin{theorem}\label{thm:containorder13}
  Suppose  $\Aut(X)$ contains an element of order $13$. Then $\Aut(X)$ is isomorphic to a subgroup of  one of the three groups $G$ which appear in Example (9), Example (12) and Example (16) in Example \ref{mainex} and $|\Aut(X)|$ is divided by $3$.
\end{theorem}

\begin{proof}
Proof is the same as that of Theorem \ref{thm:containorder17} except that there are $3$ possible maximal $G$.
\end{proof}

\section{Sylow 2-subgroups}\label{ss:Sylow2}

{\it In this section, $X$ is always a smooth quintic threefold defined by $F$.}

In this section, we study Sylow $2$-subgroups of $\Aut(X)$.

In the rest of the paper,  we extensively use the mathematical software GAP. In GAP library, groups of order $\leq 2000$ are stored (except groups of order 1024). In GAP, all the information (structure descriptions, subgroups, character tables, automorphism groups, etc.) of these groups we need are included. 

A terminology used in GAP: $\SmallGroup(a,b)$:= the $b$-th group of order $a$. 

For example, by classification, up to isomorphism, there are exactly five different groups of order $8$: $C_8,C_4\times C_2, D_8,Q_8,C_2^3$.  In GAP, these five groups are stored in a specific order. In fact,  $\SmallGroup(8,1)\cong C_8$,  $\SmallGroup(8,2)\cong C_4\times C_2$,  $\SmallGroup(8,3)\cong D_8$,  $\SmallGroup(8,4)\cong Q_8$,  $\SmallGroup(8,5)\cong C_2^3$.

In the rest of this paper, if no confuse causes, we use the following:

{\bf Convention.} Let $a$ be a positive integer less than or equal to $2000$, and $a\neq 1024$. Suppose, up to isomorphism, there are $k_a$ many different groups of order $a$. Let $b$ be a positive integer less than or equal to $k_a$. Then we denote by $[a,b]$ a group isomorphic to $\SmallGroup(a,b)$. In fact, in GAP, $[a,b]$ is regarded as the \lq\lq{}ID\rq\rq{} of $\SmallGroup(a,b)$. We also call $[a,b]$ the \lq\lq{}GAP ID\rq\rq{} of groups isomorphic to $\SmallGroup(a,b)$. For example, $[8,3]$ is a group isomorphic to the dihedral group $D_8$.

Using GAP, one can quickly find all possible $2$-groups which are isomorphic to a subgroup of the 22 groups in Section \ref{ss:22examples}. There are  25 such $2$-groups. For reader\rq{}s convenience, we list all of them below (including their GAP IDs):

$[2,1]\cong C_2$, $[4,1]\cong C_4$, $[4,2]\cong C_2^2$, $[8,1]\cong C_8$,$[8,2]\cong C_4\times C_2$, $[8,3]\cong D_8$, $[8,4]\cong Q_8$, $[16,1]\cong C_{16}$, $[16,2]\cong C_4^2$, $[16,5]\cong C_8\times C_2$, $[16,6]$, $[16,7]\cong D_{16}$, $[16,8]$, $[16,9]$, $[16,13]$, $[32,1]\cong C_{32}$, $[32,3]\cong C_8\times C_4$,  $[32,11]$, $[32,16]\cong C_{16}\times C_2$, $[32,42]$, $[64,1]\cong C_{64}$, $[64,26]\cong C_{16}\times C_4$, $[64,50]\cong C_{32}\times C_2$, $[128,1]\cong C_{128}$, $[256,1]\cong C_{256}$.

Our goal is to show that the 25 groups above are all possible $2$-groups which acts on a smooth quintic threefold. In other words, we exclude all other $2$-groups. 

\begin{theorem}\label{thm:sylow2}
Let $G<\Aut(X)$. Suppose $G$ is a $2$-group. Then $G$ is isomorphic to a subgroup of one of the $22$ groups in the  Examples (1)-(22) in Example \ref{mainex}.
\end{theorem}

\begin{remark}\label{rmk:stepstoexcludegroups}

Here we explain the main ideas of the proof.

  We will exclude groups inductively (from smaller orders to larger orders). Our strategies to exclude groups consist of two steps:

{\bf Step one:} Let $G$ be a $2$-groups of order $2^n$. If  Theorem \ref{thm:sylow2} has been proved for orders strictly less than $2^n$ and $G$ contains a proper subgroup which is not isomorphic to one of the above 25 groups then the group $G$ is excluded. 

In the sequel, we call this method of excluding groups as {\it sub-test}. We always use GAP to do sub-test. The detailed GAP codes can be found on the second author\rq{}s personal website \cite{Yu}. (Sub-test will also be used in Sections \ref{ss:2^63^25^2} and \ref{ss:goren}.)  {\it Surprisingly, it turns out that sub-test is quite effective in our study.}

{\bf Step two:} If $G$ survives after sub-test and $G$ is not one of the above 25 groups, we just do case by case consideration.

\end{remark}

We now start to prove Theorem \ref{thm:sylow2}.

\begin{proof}
1) $|G|=2$: Trivial.

2) $|G|=4$: Trivial.

3) $|G|=8$: It suffices to exclude $[8,5]\cong C_2^3$.

\begin{lemma}\label{lem:[8,5]}
$C_2^3$ can not be a subgroup of $Aut(X).$
\end{lemma}

\begin{proof}
Suppose $N< \Aut(X)$ and $N\cong C_2^3$. Then we may assume $N=\langle [A_1],[A_2],[A_3]\rangle $, where $A_1=\Diag(1,-1,1,1,(-1)^a)$, $A_2=\Diag(1,1,-1,1,(-1)^b)$, $A_3=\Diag(1,1,1,-1,(-1)^c) $, and $A_1(F)=A_2(F)=A_3(F)=F.$  Then $x_1^{i_1}x_2^{i_2}x_3^{i_3}x_4^{i_4}x_5^{i_5}\in F$ only if $i_1\neq 0$ or $i_5\neq 0$. Therefore, $F=x_1 H+x_5G$, for some degree 4 polynomials $H$ and $G$, a contradiction to the smoothness of $X$.
\end{proof}

\begin{lemma}\label{lem:to2}
Let $g\in \Aut(X)$. Suppose $\Ord(g)=2$ and $A$ is an $F$-lifting of $g$. Then the trace of $A$ must be  positive (more precisely, $1,$ or $3$).
\end{lemma}

\begin{proof}
Just apply Lemma \ref{lem:em3}.
\end{proof}

\begin{remark}
It turns out this simple lemma is extremely useful to exclude groups in the rest of this paper. It is a little mysterious why it is so useful in our study. 
\end{remark}

4) $|G|=16$: after sub-test, one sees that it suffices to exclude: $[16,4]$, $[16,12]$.

\begin{lemma}\label{lem:[16,4]}
$[16,4]\cong C_4\rtimes C_4$ is not a subgroup of $\Aut(X)$.
\end{lemma}

\begin{proof}
Assume to the contrary that $[16,4]\cong G<\Aut (X)$. Then by Theorem \ref{liftable}, $G$ is $F$-liftable, i.e., there exists a faithful representation of degree $5$, say $\rho$,  of $[16,4]$ such that  $A(F)=F$ for all matrices $A$ belonging to $\rho$.

Next we look at the character table of $[16,4]$, see Figure \ref{table:[16,4]}.

\begin{figure}[htbp]
\begin{center}

\includegraphics[width=15cm, height=18cm]{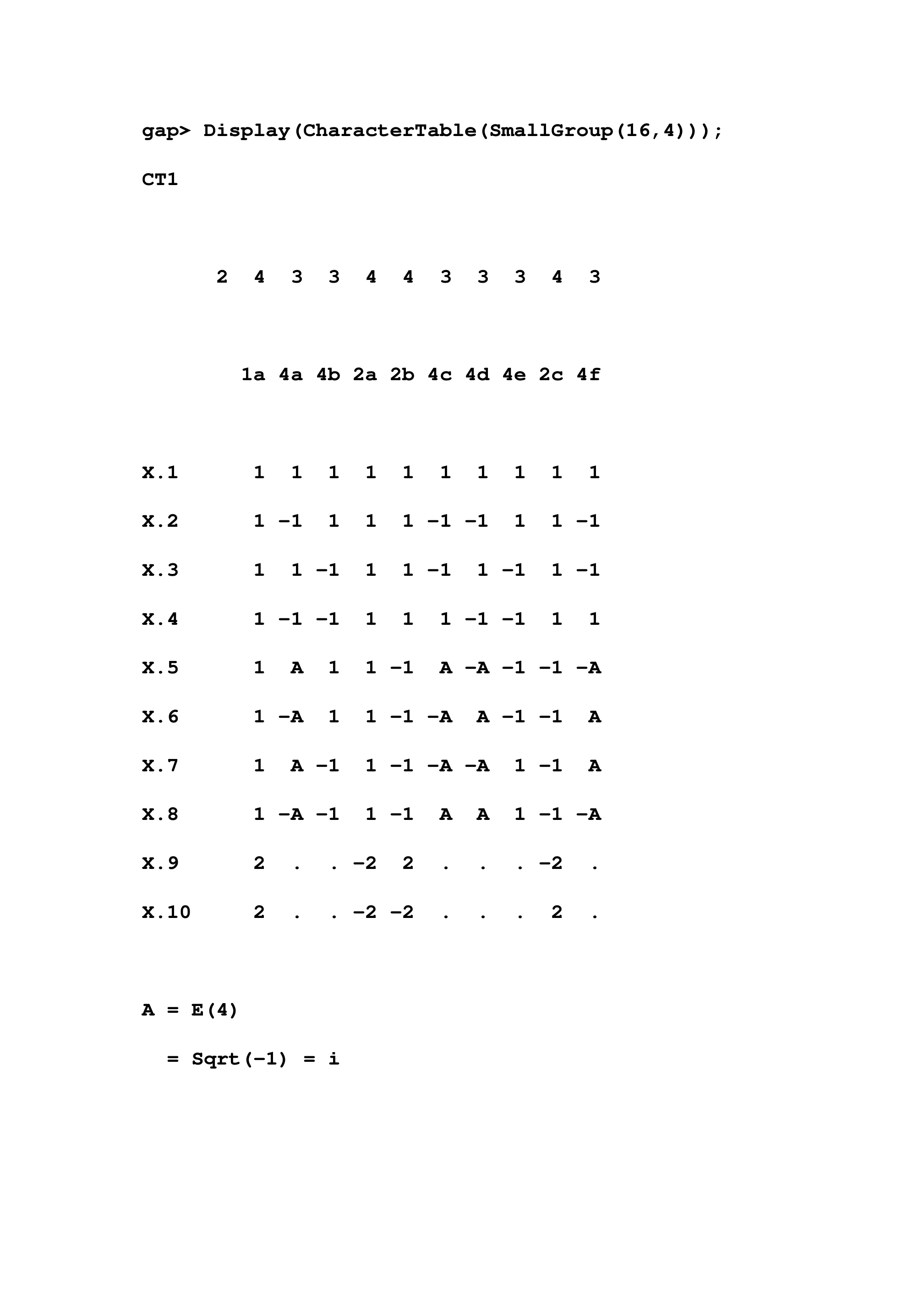}

\caption{Character table of $[16,4]$}
\label{table:[16,4]}
\end{center}
\end{figure}

As the group $[16,4]$ is not abelian, the faithful representation $\rho$ must contain a $2$-dimensional irreducible representation. By character values of the conjugacy class  $2a$ and Lemma \ref{lem:to2}, $\rho$ must be of type $2\oplus 1\oplus 1\oplus 1$ (i.e., $\rho$ decomposes into  one $2$-dimension irreducible representation, and three $1$-dimensional irreducible representations). 
  
If the 2-dimensional irreducible representation contained in $\rho$ is $X.9$, then by character values of conjugacy class $2b$ and faithfulness of $\rho$, one of $X.5,X.6,X.7,X.8$ must be contained in $\rho$. Then the trace of the conjugacy class $2c$ must be negative, a contradiction to Lemma \ref{lem:to2}. Therefore, $X.9$ is not contained in $\rho$.

Similarly, $X.10$ is also not contained in $\rho$. 

But one of $X.9$ or $X.10$ must be in $\rho$, a contradiction.
\end{proof}

\begin{lemma}\label{lem:[16,12]}
$[16,12]\cong C_2\times Q_8$ is not a subgroup of $\Aut(X)$.
\end{lemma}

\begin{proof}
Similar to $[16,4]$. For the character table of $[16,12]$, see the website \cite{Yu}.

%
%
%

\end{proof}

5) $|G|=32$: Again, after sub-test, it suffices to exclude: $[32,4]$, $[32,12]$, $[32,15]$, $[32,17]$, $[32,18]$, $[32,19]$, $[32,20]$, $[32,38]$.

\begin{lemma}\label{lem:[32,4]}
Neither $[32,4]\cong C_8\rt C_4$ nor $[32,12]\cong C_4\rt C_8$ is a subgroup of $\Aut(X)$.
\end{lemma}

\begin{proof}
Similar to [16,4] case. 

%
%
%
%
\end{proof}

For the remaining $6$ cases we need some new tools to exclude.

\begin{lemma}\label{lem:eoo8}
Let $g\in \Aut(X)$ be of order $8$. Suppose $A\in \GL(5,\C)$ is an $F$-lifting of $g$. Then:

(i) $-1$ is one of the eigenvalues of $A$;

(ii) If $\xi_8$ is an eigenvalue of $A$, then the multiplicity of $\xi_8$ as an eigenvalue of $A$ is exactly one; and

(iii) If $-1$ is an eigenvalue of $A$ of multiplicity two, then, up to replacing $A$ by its odd power, we may assume $A=\Diag(\xi_8,-1,1,-1,1)$.

\end{lemma}

\begin{proof}
We may assume $A=\Diag(\xi_8, \xi_8^a,\xi_8^b,\xi_8^c,\xi_8^d)$, where $0\leq a,b,c,d\leq 7$. Since $A(F)=F$, then $x_1^5\notin F$. We may assume $x_1^4x_2\in F$. Then $a=4$. So, $-1$ is an eigenvalue of $A$, and i) is proved.

For ii), assume to the contrary that $\xi_8$, as eigenvalue of $A$, is of multiplicity greater than one. 

We may assume $A=\Diag(\xi_8,-1,1,\xi_8,\xi_8^d)$, $0\leq d \leq 7.$ (Since $1$ must be an eigenvalue of $A$.) By Lemma \ref{lem:em3}, $d$ can not be odd; otherwise, $-1$ is an eigenvalue of $A^4$ of multiplicity three.

If $d=0,2,6$, computing invariant monomials of $A$ (e.g., using Mathematica), we have $F\in ( x_2) +( x_3,x_5) ^2$, a contradiction to Proposition \ref{pp:nonsmoothquintic} (3). If $d=4$, computing invariant monomials of $A$ (e.g., using Mathematica), we have $F\in ( x_3) +( x_1,x_4) ^2$, a contradiction. So ii) is proved.

For iii), suppose, $-1$ is an eigenvalue of $A$ of multiplicity two. Then we may assume $A=\Diag(\xi_8,-1,1,-1,\xi_8^d)$, where $0\leq d \leq 7$. Clearly $d\neq 1$ by ii) above. Also $d\neq 4$.

If $d=2,3,5,6,7$, then as before, we have $F\in ( x_3) +( x_1,x_5) ^2$, a contradiction. Therefore, iii) is proved.

\end{proof}

\begin{lemma}\label{lem:to4}
Let $g\in \Aut(X)$ of order $4$. Suppose $A\in \GL(5,\C)$ is an $F$-lifting of $g$. Then the trace of $A$ is not equal to $-1$.
\end{lemma}

\begin{proof}
Assume to the contrary that ${\rm tr}A=-1$.

Then, up to change of coordinates, $A=\Diag(\xi_4,-\xi_4,-1,-1,1)$ or $\Diag(\xi_4,-\xi_4,\xi_4,-\xi_4,-1)$.

If $A=\Diag(\xi_4,-\xi_4,\xi_4,-\xi_4,-1)$, computing invariant monomials of $A$ as in the proof of Lemma \ref{lem:eoo8}, we have $F\in ( x_5) $, a contradiction.

If $A=\Diag(\xi_4,-\xi_4,-1,-1,1)$, then we have $F\in ( x_5) +( x_1,x_2)^2 $, a contradiction, as before.

\end{proof}

\begin{lemma}\label{lem:[32,15]}
Neither $[32,15]\cong C_4.(C_4\times C_2)$ nor $[32,28]\cong (C_8\times C_2)\rt C_2$is not a subgroup of $\Aut(X)$.

\end{lemma}

\begin{proof}

We only give a proof for $[32,15]$. Proof of $[32,28]$ is similar. Assume to the contrary. Then we may assume $\rho$ is a 5 dimensional faithful representation of $[32,15]$ which leaves $F$ invariant.

By character table of $[32,15]$ (see Figure \ref{table:[32,15]}), since $\rho$ is faithful, the trace of the conjugacy class $2a$ must be positive and the trace of the conjugacy class $4a$ is not equal to $-1$, we have that $\rho$ is of type $2\oplus 1\oplus 1\oplus 1$, and the 2 dimensional irreducible representation is one of  $X.11\sim X.14$.

\begin{figure}[htbp]
\begin{center}

\includegraphics[width=20cm, height=22cm]{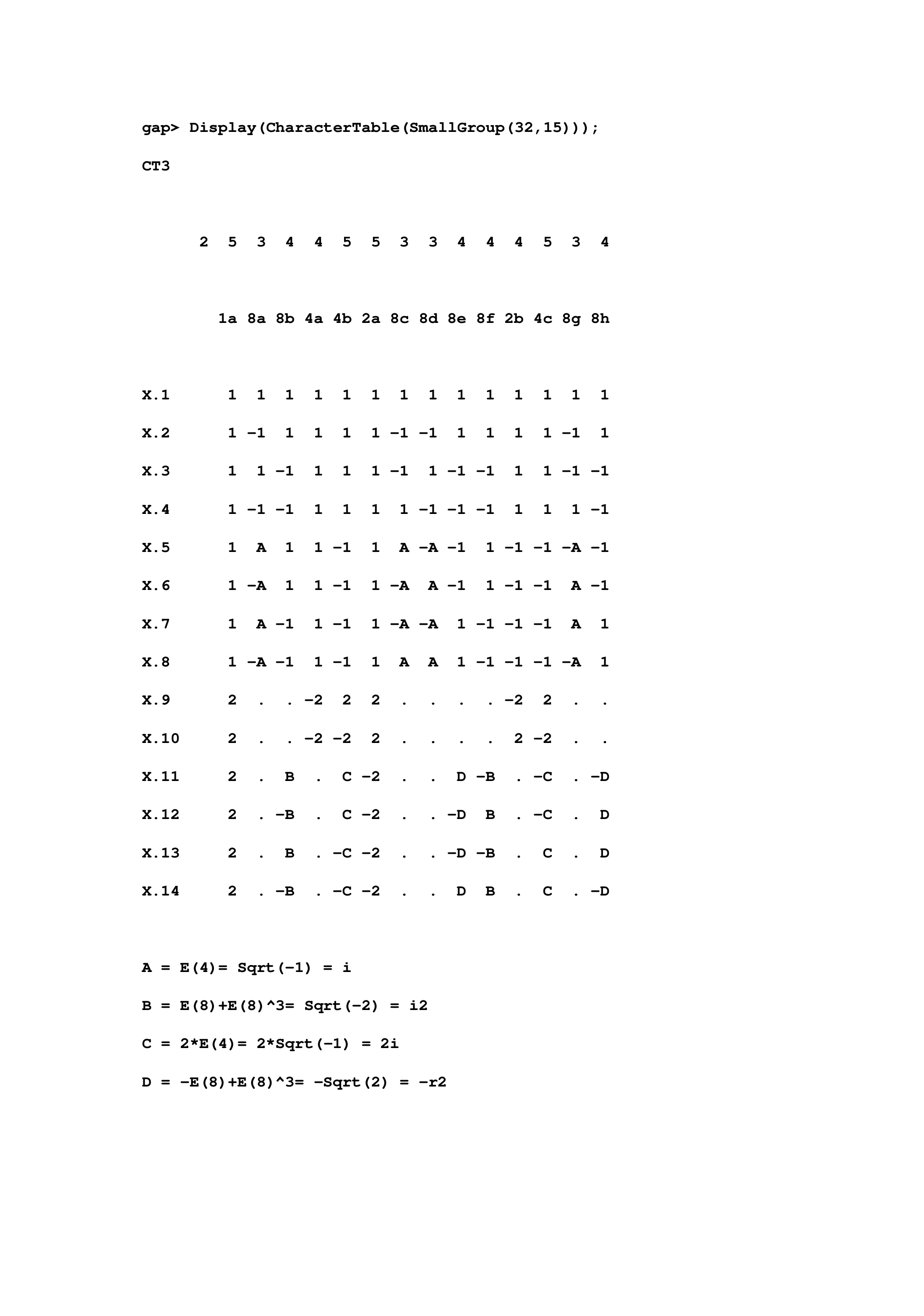}

\caption{Character table of $[32,15]$}
\label{table:[32,15]}
\end{center}
\end{figure}

Considering values of conjugacy classes $8a,8b,8c$, by Lemma \ref{lem:eoo8}, we see that all $8a,8c$ must have $-1$ as one of their eigenvalues.

If $X.2$ is not in $\rho$, then by eigenvalue consideration of $8a$, $X.4$ must be in $\rho$. Then by eigenvalue consideration of $8c$, $X.3$ must be in $\rho$. Then $8b$ has $-1$ as an eigenvalue of multiplicity greater than one, a contradiction to Lemma \ref{lem:eoo8}.

So $X.2$ must be in $\rho$.

Then $X.3$ is not in $\rho$ (otherwise, $8c$ has $-1$ as eigenvalue of multiplicity greater than one), and $X.4$ is also not in $\rho$ (otherwise, $8a$ has $-1$ as eigenvalue of multiplicity greater than one). Then one of $X.7$ and $X.8$ must be in $\rho$ (by eigenvalue consideration of $8b$), and one of $X.5$ and $X.6$ must be in in $\rho$ (by eigenvalue consideration of $8h$). However, then trace of $2b$ must be negative, a contradiction.
 
So such $\rho$ does not exist and the lemma is proved.

\end{proof}

\begin{lemma}\label{lem:C16}
If $C_{16}\cong N\lhd G< \Aut(X)$. Then $G=C_G(N)$.
\end{lemma}

\begin{proof}
We may assume $N=\langle [A]\rangle $, $A=\Diag(\xi_{16},\xi_{16}^{-4},1,\xi_{16}^a,\xi_{16}^b), A(F)=F$, and $0\leq a\leq b\leq 15$. Suppose $G\neq C_G(N)$, then there exists $[B]\in G$, such that $BAB^{-1}=A^{\alpha}$, $1< \alpha\leq15$. Since $A$ and $A^{\alpha}$ have the same characteristic polynomial, so $\alpha$ is odd. 

Case (1): $\alpha =3$. Then $$\chi_{A}(t)=(t-\xi_{16})(t-\xi_{16}^{-4})(t-1)(t-\xi_{16}^a)(t-\xi_{16}^b)=(t-\xi_{16}^3)(t-\xi_{16}^{4})(t-1)(t-\xi_{16}^{3a})(t-\xi_{16}^{3b})=\chi_{A^3}(t).$$ As $(t-\xi_{16}^3)(t-\xi_{16}^4)$ divides $\chi_{A^3}(t)$, so $(t-\xi_{16}^3)(t-\xi_{16}^4)$ divides $\chi_{A}(t)$, too. Therefore, $a=3, b=4$. Then $(t-\xi_{16}^9)$ divides $\chi_{A^3}(t)$ but $(t-\xi_{16}^9)$ does not divide $\chi_A(t)$, a contradiction. Therefore, $\alpha =3$ is impossible.

Case (2): $\alpha=5$. By the same argument as in Case (1), $\alpha=5$ is also impossible.

Case (3): $\alpha=7$. Then $\chi_A(t)=\chi_{A^7}(t)$ implies $a=4$ and $b=7$. Then $F\in ( x_3) +(x_1,x_5) ^2$, contradiction. So this case is also impossible.

Case (4): $\alpha=9$. $\chi_A(t)=\chi_{A^9}(t)$ implies $(a,b)$ = $(2,9),(4,9),(6,9)$, $(8,9),(9,10)$, $(9,12),(9,14)$. If $(a,b)=(2,9)$ or $(9,10)$, then $x_5^4x_i\notin F$, for all $1\leq i\leq 5$, so $X$ is singular, a contradiction. If $(a,b)=(4,9),$ or $(9,12)$, computing invariant monomials of $A$ (e.g., using Mathematica), we have $F\in ( x_3) +( x_1,x_4) ^2$, a contradiction. If $(a,b)=(6,9)$, $(8,9),$ or $(9,14)$, then $F\in ( x_2,x_3) $, a contradiction. Therefore, $\alpha=9$ is impossible.

Case (5): $\alpha=11$. By the same argument as in Case (1), $\alpha=11$ is also impossible.

Case (6): $\alpha=13$. By the same argument as in Case (1), $\alpha=13$ is also impossible.

Case (7): $\alpha=15$. $\chi_A(t)=\chi_{A^{15}}(t)$ implies $(a,b)=(4,15)$, but then $F\in ( x_3) +( x_1,x_4) ^2$, a contradiction.

Therefore, we must have $G=C_G(N)$.
\end{proof}

\begin{lemma}\label{lem:[32,17]}
None of $[32,17]$, $[32,18]$, $[32,19]$, $[32,20]$ is a subgroup of $\Aut(X)$.
\end{lemma}

\begin{proof}
By GAP, these four groups are non-abelian and contain a subgroup isomorphic to $C_{16}$. So,  just apply Lemma \ref{lem:C16}.
\end{proof}

This completes the proof for $|G|=32$.

6) $|G|=64$, after sub-test, it suffices to exclude five groups: $[64,2]$, $[64,3]$, $[64,27]$, $[64,44]$, $[64,51]$.

\begin{lemma}\label{lem:8times8}

$C_8\times C_8\cong [64,2]$ can not be a subgroup of $\Aut(X).$
\end{lemma}

\begin{proof}
 Suppose $C_8\times C_8\cong N< \Aut(X)$. Then we may assume $N=\langle [A_1],[A_2]\rangle ,A_1=\Diag(\xi_8,\xi_8^{-4},1,\xi_8^a,\xi_8^b),A_2=\Diag(1,1,1,\xi_8,\xi_8^{-4}),A_1(F)=A_2(F)=F$. Then we may assume $a=0$. Then $b=0$ since $x_4^4x_5\in F$. $A_1(F)=A_2(F)=F$ implies $F\in ( x_3) +( x_1,x_4) ^2$, a contradiction.

\end{proof}

\begin{lemma}\label{lem:[64,27]}
None of $[64,3]\cong C_8\rtimes C_8$, $[64,27]\cong C_{16}\rtimes C_{4}$, $[64,44]\cong C_4\rtimes C_{16}$ is isomorphic to a subgroup of $\Aut(X)$.
\end{lemma}

\begin{proof}
Similar to [16,4] case. (Character tables we need are on the website \cite{Yu}.)
\end{proof}

\begin{lemma}\label{lem:C32}

If $C_{32}\cong N\lhd G < \Aut(X)$, then $G=C_G(N).$
\end{lemma}

\begin{proof}
We may assume $N=\langle [A]\rangle $, $A=\Diag(\xi_{32},\xi_{32}^{-4},\xi_{32}^{16},1,\xi_{32}^a)$, $A(F)=F$, and $0\leq a\leq 31$. Let $[B]\in G$. Then we may assume $BAB^{-1}=A^{\alpha}$. If $\alpha=1$, then $[B]\in C_G(N)$. If $\alpha\neq 1$, then $32> \alpha> 1$ and $\alpha$ is odd. As $\chi_{A}(t)=\chi_{A^{\alpha}}(t)$, we have $\alpha=a=17$. But then $A(F)=F$ implies $F\in ( x_2,x_4) $, a contradiction. Therefore, $\alpha$ must be 1, and hence $G=C_G(N)$.
\end{proof}

\begin{lemma}\label{lem:[64,51]}
$[64,51]\cong C_{32}\rt C_2$ is not a subgroup of $\Aut(X)$.
\end{lemma}

\begin{proof}
 Just apply Lemma \ref{lem:C32}.
\end{proof}

7) $|G|=128$, by sub-test we need to exclude five groups: $[128,128]$, $[128,129]$, $[128,153]$, $[128,159]$, $[128,160]$.

\begin{lemma}\label{lem:32times4}
$C_{32}\times C_4\cong [128,128]$ is not isomorphic to a subgroup of $\Aut(X)$.
\end{lemma}

\begin{proof}
Suppose $C_{32}\times C_4\cong N< \Aut(X)$. We may assume  $N=\langle [A_1],[A_2]\rangle $, $A_1=\Diag(\xi_{32},\xi_{32}^{-4},\xi_{32}^{16},1,\xi_{32}^a)$,  $A_2=\Diag(1,1,1,1,\xi_4)$ and $A_1(F)=A_2(F)=F$. We may assume $0\leq a\leq 7$. So $a=0,1,2,3,4,5$ or $6$.

If $a=0,1,2,3,5$, or $6$, computing invariant monomials of $A_1$ and $A_2$ as before, we have $F\in ( x_2,x_4)$, a contradiction. If $a=4$, then $F\in ( x_3) +( x_1,x_4) ^2$, a contradiction. 

Therefore, $C_{32}\times C_4$ can not be a subgroup of $\Aut(X)$.
 
\end{proof}

\begin{lemma}\label{lem:[128,129]}
$[128,129]\cong C_{32}\rt C_4$ is not isomorphic to a subgroup of $\Aut(X)$.
\end{lemma}

\begin{proof}
Just apply Lemma \ref{lem:C32}.
\end{proof}

\begin{lemma}\label{lem:[128,153]}
$[128,153]\cong C_4\rt C_{32}$ is not isomorphic to a subgroup of $\Aut(X)$.
\end{lemma}

\begin{proof}
Similar to [16,4] case. (The character table of $[128,153]$ is on the website \cite{Yu})

\end{proof}

\begin{lemma}\label{lem:64times2}
$C_{64}\times C_2\cong [128,159]$ is not isomorphic to a subgroup of $\Aut(X)$.
\end{lemma}

\begin{proof}
Suppose $C_{64}\times C_2\cong N<  \Aut(X)$. We may assume $N=\langle [A_1],[A_2]\rangle $, $A_1=\Diag(\xi_{64},\xi_{64}^{-4},\xi_{64}^{16},1,\xi_{64}^a)$,  $A_2=\Diag(1,1,1,1,-1)$ and $A_1(F)=A_2(F)=F$. We may assume $0\leq a\leq 31$. If $a\neq 8,12,23$ or $28$, then $F\in ( x_2,x_4) $, a contradiction. If $a=8,12,23$ or $28$, then $F\in ( x_3) +( x_1,x_4) ^2$, a contradiction.

Therefore, $C_{64}\times C_2$ can not be a subgroup of $\Aut(X)$.
 
\end{proof}

\begin{lemma}\label{lem:C64}
If $C_{64}\cong N\lhd G< \Aut(X)$, then $G=C_G(N).$
\end{lemma}

\begin{proof}
Similar to Lemma \ref{lem:C32}.
\end{proof}

\begin{lemma}\label{lem:[128,160]}
$[128,160]\cong C_{64}\rt C_2$ is not a subgroup of $\Aut(X)$.
\end{lemma}

\begin{proof}
Just apply Lemma \ref{lem:C64}.
\end{proof}

Now, the following Theorem completes the proof of Theorem \ref{thm:sylow2}:

\begin{theorem}\label{thm:C128}
If $\Aut(X)$ contains a subgroup of order $2^n$ with $n\geq 7$, then $\Aut(X)\cong C_{128}$ or $C_{256}$.
\end{theorem}

\begin{proof}
Let $G<\Aut(X)$ and $|G|=2^n$, $n\geq 7$. 

Since $G$ is a $2$-group, $G$ contains a subgroup, say $N$, of order $128$. Then, $N\cong C_{128}$, by the results for $|G|=128$ explained above. Then we may assume $N=\langle [A]\rangle $, where $A=\Diag(\xi_{128},\xi_{128}^{-4},\xi_{128}^{16},\xi_{128}^{-64},1)$ and $A(F)=F$. Then by computing (e.g. Mathematica) the invariant monomials of $A$, we find that, $F=x_1^4x_2+x_2^4x_3+x_3^4x_4+x_4^4x_5+x_5^5+\lambda x_4^2x_5^3,$ for some $\lambda\in \mathbb{C}$, up to change of coordinates. Then, applying the differential method in Section \ref{ss:differentialmethod}, we can compute $\Aut(X)\cong C_{256}$ if $\lambda =0$ (cf. Theorem \ref{thm:Aut(X)1-16}) or $\Aut(X)\cong C_{128}$ if $\lambda\neq 0$.
\end{proof}

%
%

\end{proof}

\section{Sylow 5-subgroups}\label{ss:Sylow5}

{\it In this section, $X$ is a smooth quintic threefold defined by $F$.} 

In this section, we study Sylow $5$-subgroups of $\Aut(X)$.

\begin{lemma}\label{lemma1}
Suppose $C_5^2\cong N< \Aut(X)$. Then:

(i) $N$ is generated by $[A_1],$ and $[A_2]$,  $A_1A_2=A_2A_1$, and both $A_1$ and $A_2$ have order 5 as elements in $\GL(5, \mathbb{C})$; or

(ii) up to change of coordinates, $N$ is generated by $[A_1]$ and $[A_2]$, where $A_1=\Diag(1,\xi_5, \xi_5^2, \xi_5^3, \xi_5^4 )$ and $$A_2=
    \begin{pmatrix} 
    0&1&0&0&0 \\ 
    0&0&1&0&0 \\ 
    0&0&0&1&0 \\ 
    0&0&0&0&1 \\ 
    1&0&0&0&0 \\ 
    \end{pmatrix} 
   $$
\end{lemma}

\begin{proof}
See Lemma \ref{lem:5^2notliftable} and its proof.
\end{proof}


\begin{lemma}\label{lemma2}
Let $G=\Aut(X)$. If $|G_5|=125$ and $G_5$ is not abelian, then there exist $ A_1,A_2, \;\text{and}\; A_3$ in $\GL(5,\mathbb{C})$, such that $G=\langle [A_1],[A_2],[A_3]\rangle , A_1A_2=A_2A_1, A_1A_3=A_3A_1$, and, for all $i=1,2,3$, both $A_i$ and $[A_i]$ are of order 5 as elements in $\GL(5,\C)$ and $\PGL(5,\C)$ respectively.
\end{lemma}

\begin{proof}
By Theorem \ref{thm:noorder25}, each nontrivial element in $G_5$ has order 5. Then $G_5$ contains a normal abelian subgroup $N\cong C_5^2$. Since $N\cap Z(G_5)\neq \emptyset$, where $ Z(G_5)$ is the center of the $5$-group $G_5$, we may assume $N=\langle [A_1],[A_2]\rangle $, and $[A_1]\in Z(G_5)$, and $A_1=\Diag(1,\xi_5,\xi_5^a,\xi_5^b,\xi_5^c)$, $0\leq a\leq b\leq c\leq 4$, and $A_2^5=I_5$ and $F$ is $A_i$-semi-invariant, $i=1,2$. 

Then by Lemma \ref{lemma1}, there are two possibilities:

  (i) $A_2A_1=A_1A_2$; or
  
  (ii) $A_2A_1\neq A_1A_2$, and we can assume $A_1=\Diag(1,,\xi_5, \xi_5^2, \xi_5^3, \xi_5^4)$, and $$A_2= \begin{pmatrix} 
    0&1&0&0&0 \\ 
    0&0&1&0&0 \\ 
    0&0&0&1&0 \\ 
    0&0&0&0&1 \\ 
    1&0&0&0&0 \\ 
    \end{pmatrix}.$$
    
   We show that the case (ii) does not happen. Suppose that the case (ii) happens. Then we can choose $A_3\in \GL(5,\mathbb{C})$, such that, $[A_3]\in G_5\setminus N$, and $A_3^5=I_5$ and $F$ is $A_3$-semi-invariant. We first show that $A_1A_3\neq A_3A_1$ by argue by contradiction. If $A_1A_3=A_3A_1$, then $A_3$ must be diagonal, and $A_3=\Diag(1,1,\xi_5^{a^\prime},\xi_5^{b^\prime},\xi_5^{c^\prime})$. Since $N\lhd G_5,$ it follows that $[A_3][A_2][A_3]^{-1}\in N$, and  $$A_3A_2A_3^{-1}= \begin{pmatrix} 
    0&1&0&0&0 \\ 
    0&0&\xi_5^{-a^{\prime}}&0&0 \\ 
    0&0&0&\xi_5^{a^{\prime}-b^{\prime}}&0 \\ 
    0&0&0&0&\xi_5^{b^{\prime}-c^{\prime}} \\ 
    \xi_5^{c^{\prime}}&0&0&0&0 \\ 
    \end{pmatrix}=\xi_5^{\alpha}A_1^iA_2^j\, ,$$ where $0\leq \alpha, i, j\leq 4$. Then $j=1$ and $\alpha =0$. Therefore, $$ \begin{pmatrix} 
    0&1&0&0&0 \\ 
    0&0&\xi_5^{-a^{\prime}}&0&0 \\ 
    0&0&0&\xi_5^{a^{\prime}-b^{\prime}}&0 \\ 
    0&0&0&0&\xi_5^{b^{\prime}-c^{\prime}} \\ 
    \xi_5^{c^{\prime}}&0&0&0&0 \\ 
    \end{pmatrix}= \begin{pmatrix} 
    0&1&0&0&0 \\ 
    0&0&\xi_5^{i}&0&0 \\ 
    0&0&0&\xi_5^{2i}&0 \\ 
    0&0&0&0&\xi_5^{3i} \\ 
    \xi_5^{4i}&0&0&0&0 \\ 
    \end{pmatrix}\, .$$ 
    
    So, $A_3=\Diag(1,1,\xi_5^{-i},\xi_5^{-3i},\xi_5^{-6i})$. Then without loss of generality, we can assume $A_3=\Diag(1,1,\xi_5,\xi_5^3,\xi_5)$. Recall that  $A_2(F)=\lambda_2 F,A_3(F)=\lambda_3 F$, where $\lambda_2,\lambda_3$ are nonzero complex numbers. Then $(A_2A_3A_2^{-1})(F)=A_3(F)$. So $\Diag(1,\xi_5,\xi_5^3,\xi_5,1)(F)=\Diag(1,1,\xi_5,\xi_5^3,\xi_5)(F)$. Then by smoothness of $X$, $x_1^5\in F$. So $A_1(F)=A_2(F)=F$. Then $x_1^{i_1}x_2^{i_2}x_3^{i_3}x_4^{i_4}x_5^{i_5}\in F$ only if $$i_1+i_2+i_3+i_4+i_5=5,i_3+3i_4+i_5\equiv 0(\Mod 5),i_2+2i_3+3i_4+4i_5\equiv 0(\Mod 5)\, ,$$ and $i_1,...,i_5\geq 0$. It follows that only $x_1^5,x_2^5,x_3^5,x_4^5,x_5^5,x_1x_2x_3x_4x_5$ can appear in $F$. However, then $|G_5|\geq 625$ as $G_5$ then contains the Sylow $5$-subgroup of {\it Gorenstein} automorphism groups of the Fermat quintic threefold (cf. Theorem \ref{thm:goren}), a contradiction. Hence $A_3A_1\neq A_1A_3$ in the case (ii) and we may assume $A_3A_1=\xi_5 A_1A_3$. Then by  Lemma \ref{lem:matrixshape} $$A_3= 
    \begin{pmatrix} 
    0&a_1&0&0&0 \\ 
    0&0&a_2&0&0 \\ 
    0&0&0&a_3&0 \\ 
    0&0&0&0&a_4 \\ 
    a_5&0&0&0&0 \\ 
    \end{pmatrix}$$ for some $a_1,a_2,a_3,a_4,a_5$ with $a_1a_2a_3a_4a_5=1$. Then $\Diag(a_1,a_2,a_3,a_4,a_5)=A_3A_2^{-1}$. By replacing $A_3$ by $A_3A_2^{-1}$, we are reduced to the previous situation (i.e. $A_3A_1=A_1A_3$. Therefore, the case (ii) is impossible.
    
    Hence $A_2A_1=A_1A_2$. Since $[A_1]\in Z(G_5)$, if we choose  $A_3\in \GL(5,\mathbb{C})$ such that  $[A_3]\in G_5\setminus N$, and $A_3^5=I_5$, then, by symmetry, we must also have  $A_1A_3=A_3A_1$. So the lemma is proved.

\end{proof}

\begin{lemma}\label{le3}
Let $G=\Aut(X)$. If $|G_5|=125$ and $G_5$ is not abelian, then there exist $ A_1,A_2,A_3\in \GL(5,\mathbb{C})$ such that  $G=\langle [A_1],[A_2],[A_3]\rangle , A_1A_2=A_2A_1, A_1A_3=A_3A_1$,  $\Det(A_1)=1$, and for all $i=1,2,3$, both $A_i$ and $[A_i]$ are of order $5$ as elements of $\GL(5,\C)$ and $\PGL(5,\C)$ respectively.
\end{lemma}
\begin{proof}
By Lemma \ref{lemma2}, there exist $ A_1,A_2,A_3\in \GL(5,\mathbb{C})$ such that $G=\langle [A_1],[A_2],[A_3]\rangle $ and

\begin{equation}\label{eq0}
  A_1A_2=A_2A_1, A_1A_3=A_3A_1, A_i^5=I_5, i=1,2,3.
\end{equation}  
   
     We show $\Det(A_1)=1$. Since $G_5$ is not abelian, the commutator subgroup $[G_5,G_5]$ is not the trivial subgroup. Then $[G_5,G_5]=Z(G_5)=\langle [A_1]\rangle $. So $A_2A_3A_2^{-1}A_3^{-1}=\xi_5^iA_1^{\alpha}$. Since $G_5$ is not abelian, so $\alpha\neq 0(\Mod 5)$. Then $$1=\Det(A_2A_3A_2^{-1}A_3^{-1})=\Det(\xi_5^iA_1^{\alpha})=(\Det(A_1))^{\alpha}.$$ On the other hand, $A_1^5=I_5$. Hence $\Det(A_1)=1$.
\end{proof}

\begin{proposition}\label{leAb}
Let $G=\Aut(X)$. If $|G_5|=125$, then $G_5$ is abelian.

\end{proposition}

\begin{proof}
Assuming that $G_5$ is not abelian, we shall get a contradiction. By Lemma \ref{le3}, there exist $ A_1,A_2, \;\text{and}\; A_3\in \GL(5,\mathbb{C})$ such that $G=\langle [A_1],[A_2],[A_3]\rangle , A_1A_2=A_2A_1, A_1A_3=A_3A_1$ and $A_i^5=I_5$,  and $\Det(A_1)=1$. Notice that $A_1$ has two or more distinct eigenvalues.

Case (1) $A_1$ has exactly two distinct eigenvalues. Then we may assume $A_1=\Diag(1,1,1,\xi_5,\xi_5)$ or $\Diag(1,1,1,1,\xi_5)$. But $\Det(\Diag(1,1,1,\xi_5,\xi_5))\neq 1$, and  $\Det(\Diag(1,1,1,1,\xi_5))\neq 1$. So this case is impossible.

Case (2) $A_1$ has exactly three distinct eigenvalues. Then we may assume (i): $A_1=\Diag(1,1,1,\xi_5,\xi_5^a)$, for some $1< a< 5$ or (ii):  $A_1=\Diag(1,1,\xi_5,\xi_5,\xi_5^b)$, for some $1< b< 5$.

For Case (2)-(i), $\Det(A_1)=1$ implies $a=4$. Then $A_1A_2=A_2A_1$ and $A_1A_3=A_3A_1$ imply $A_2=\Diag(B_2, \xi_5^p,\xi_5^q)$ and $A_3=\Diag(B_3, \xi_5^m,\xi_5^n)$, where $0\leq p,q,m,n\leq 4$ and $B_2,B_3\in \GL(3,\mathbb{C})$. As in the proof of Lemma \ref{le3}, we have
\begin{equation}\label{eq1}
  A_2A_3=\xi_5^iA_1^\alpha A_3A_2,\text{ for some}\; i\; \text{and}\; \alpha\, .
\end{equation}

By the equation (\ref{eq1}) we have

  \begin{equation}\label{eq2}
  \Diag(B_2B_3,\xi_5^{p+m},\xi_5^{q+n})=\Diag(B_3B_2,\xi_5^{p+m+i+\alpha},\xi_5^{q+n+i+4\alpha}).
  \end{equation}  
  
  By the equation (\ref{eq2}), $i=\alpha=0$. However, then, by the equation (\ref{eq1}), $A_2A_3=A_3A_2$, a contradiction. So the case (2)-(i) is impossible.
  
 Case (2)-(ii)  By $\Det(A_1)=1$, $b=3$. Then by the equalities (\ref{eq0}), $A_2=\Diag(B_2,C_2,\xi_5^{p}),A_3=\Diag(B_3,C_3,\xi_5^{m})$, where $B_2,B_3,C_2,C_3\in \GL(2,\mathbb{C})$, and $B_2^5=B_3^5=C_2^5=C_2^5=I_2$. Since $G_5$ is not abelian, $A_2A_3\neq A_3A_2$. Therefore, either $B_2B_3\neq B_3B_2$ or $C_2C_3\neq C_3C_2$. Without loss of generality, we may assume $B_2B_3\neq B_3B_2$ and $B_2=\Diag(1,\xi_5)$. Then by the equation (\ref{eq1}), we have 
 \begin{equation}\label{eq3}
 \Diag(B_2B_3,C_2C_3,\xi_5^{p+m})=\Diag(\xi_5^i B_3B_2,\xi_5^{i+\alpha}C_3C_2,\xi_5^{p+m+i+3\alpha}).
 \end{equation} 
 By the equality (\ref{eq3}), $B_2B_3=\xi_5^iB_3B_2$, in particular, $\Det(B_2B_3)=\Det(\xi_5^iB_3B_2)$. However, then $\xi_5^{2i}=1$ and $i=0$, a contradiction to $B_2B_3\neq B_3B_2$. Therefore, Case (2)-(ii) is impossible.

 Case (3) $A_1$ has exactly four distinct eigenvalues. Then we may assume $A_1=\Diag(1,1,\xi_5,\xi_5^a,\xi_5^b)$, $2\leq a< b\leq 4$. But then $\Det(A_1)$ can not be 1. Hence Case (3) is impossible. 
 
Case (4) $A_1$ has exactly five distinct eigenvalues. Then we may assume $A_1=\Diag(1,\xi_5, \xi_5^2, \xi_5^3, \xi_5^4)$. Then by the equalities (\ref{eq0}), $A_2$ and $A_3$ are both diagonal matrices, and hence $G_5$ is abelian, a contradiction. Therefore, Case (4) is impossible.

So, $G_5$ has to be abelian if $|G_5|=125$.
  
\end{proof}

\begin{lemma}\label{le4}
Suppose $C_5^3\cong N< \Aut(X)$. Then there exist $ A_1,A_2,A_3\in \GL(5,\mathbb{C})$ such that $N=\langle [A_1],[A_2],[A_3]\rangle , A_iA_j=A_jA_i$, and $A_i^5=I_5,$ where $i,j=1,2,3$.
\end{lemma}

\begin{proof}
Similar to the proof of Lemma \ref{lemma2} (and actually easier).
\end{proof}

\begin{theorem}\label{thms5}
If $C_5^3$ is isomorphic to a subgroup of $\Aut(X)$, then, up to linear change of coordinates,  $X$ is defined by one of the following equations:\\
(i) $x_1^5+x_2^5+x_3^5+x_4^5+x_5^5=0$;\\
(ii) $x_1^5+x_2^5+x_3^5+x_4^5+x_5^5+ax_1x_2x_3x_4x_5=0$, $a\neq 0$;\\
(iii) $x_1^5+x_2^5+x_3^5+x_4^5+x_5^5+ax_1x_2x_3x_5^2=0$,  $a\neq 0$;\\
(iv) $x_1^5+x_2^5+x_3^5+x_4^5+x_5^5+ax_1x_2^3x_5+bx_1^2x_2x_5^2=0$,  either $a$ or $b$ is not equal to zero;\\
(v) $x_1^5+x_2^5+x_3^5+G(x_4,x_5)=0$, where $G$ can not be written as $x_4^5+x_5^5$ under any change of coordinates.
\end{theorem}

\begin{proof}
Suppose $C_5^3\cong N< \Aut(X)$. By Lemma \ref{le4}, there exist $ A_1,A_2,A_3\in \GL(5,\mathbb{C})$ such that $N=\langle [A_1],[A_2],[A_3]\rangle , A_iA_j=A_jA_i$, and $A_i^5=I_5,$ where $i,j=1,2,3$. Clearly, we may assume $A_1=\Diag(1,\xi_5,1,1,\xi_5^a)$, $A_2=\Diag(1,1,\xi_5,1,\xi_5^b)$, $A_3=\Diag(1,1,1,\xi_5,\xi_5^c)$, where $0\leq a,b,c\leq 4$. 

Case (1) $a,b,c$ are all equal to zero. In this case, $A_1=\Diag(1,\xi_5,1,1,1)$, $A_2=\Diag(1,1,\xi_5,1,1)$,  $A_3=\Diag(1,1,1,\xi_5,1)$. Then $A_i(F)=F,i=1,2,3$. By computing invariant monomials, $F=a_1 x_2^5+a_2x_3^5+a_3x_4^5+G(x_1,x_5).$ Then by suitable changing of coordinates, the defining equation of $X$ belongs to the case (i) or (v) in Theorem \ref{thms5}. 

Case(2) exactly one of $a,b,c$ is not equal to zero. We may assume $a\neq 0,b=0,c=0$. Then $A_2=\Diag(1,1,\xi_5,1,1)$, $A_3=\Diag(1,1,1,\xi_5,1)$, and  $A_1=\Diag(1,\xi_5,1,1,\xi_5)$, $\Diag(1,\xi_5,1,1,\xi_5^2)$, $\Diag(1,\xi_5,1,1,\xi_5^3)$, or $\Diag(1,\xi_5,1,1,\xi_5^4)$. 

If $A_1=\Diag(1,\xi_5,1,1,\xi_5)$, then replacing $A_1$ by $\xi_5 A_1^4A_2^4A_3^4$, we are reduced to case (1).

If $A_1=\Diag(1,\xi_5,1,1,\xi_5^3)$, then $A_1^2=\Diag(1,\xi_5^2,1,1,\xi_5)$, then by symmetry (i.e. interchanging coordinates $x_2$ and $x_5$), we are reduced to the case $A_1=\Diag(1,\xi_5,1,1,\xi_5^2)$. If $A_1=\Diag(1,\xi_5,1,1,\xi_5^4)$, then $\xi_5A_1A_2^{-1}A_3^{-1}=\Diag(\xi_5,\xi_5^2,1,1,1)$, again by symmetry we are reduced to case $A_1=\Diag(1,\xi_5,1,1,\xi_5^2)$.

Therefore, in the case (2), we may assume $A_1=\Diag(1,\xi_5,1,1,\xi_5^2),$  $A_2=\Diag(1,1,\xi_5,1,1)$,  $A_3=\Diag(1,1,1,\xi_5,1)$. Then, $A_1(F)=A_2(F)=A_3(F)=F$. By computing  invariant monomials of $A_i$, $i=1,2,3$, we may assume $F=x_1^5+x_2^5+x_3^5+x_4^5+x_5^5+\lambda_1 x_1x_2^3x_5+\lambda_2x_1^2x_2x_5^2=0$.  So, $F$ belongs to case (i) or (iv) of Theorem \ref{thms5}.

Case (3) exactly two of $a,b,c$ are not equal to zero. We may assume $ab\neq 0,c=0$. By symmetry, we may assume $0< a\leq b\leq 4$. So $(a,b)=(1,1),(1,2),(1,3),(1,4)$, $(2,2),(2,3)$, $(2,4),(3,3),(3,4),$ or $(4,4)$.

Case (3)-(i) $(a,b)=(1,1),(1,2),(1,3),$ or $(1,4)$. By computing invariant monomials,  up to changing of coordinates, $F=x_1^5+x_2^5+x_3^5+x_4^5+x_5^5$, belonging to the case (i) of  Theorem \ref{thms5}.

Case (3)-(ii) $(a,b)=(2,2)$. Then, up to changing of coordinates, $F$ belongs to the case (iii) of  Theorem \ref{thms5}.

Case (3)-(iii) $(a,b)=(2,3)$. Up to changing of coordinates, $F$ belongs to the case (i) of  Theorem \ref{thms5}.

Case (3)-(iv) $(a,b)=(2,4),$ or $(3,3)$. Up to changing of coordinates, $F$ belongs to the case (i) or  (iv) of  Theorem \ref{thms5}.

Case (3)-(v) $(a,b)=(3,4),$ or $(4,4)$. Up to changing of coordinates, $F$ belongs to case (i) or  (iii) of  Theorem \ref{thms5}.

Case (4) $abc\neq 0$. Again, $A_1(F)=A_2(F)=A_3(F)=F$. Suppose a monomial $m=x_1^{i_1}x_2^{i_2}x_3^{i_3}x_4^{i_4}x_5^{i_5}\in F$. Then $i_1+i_2+i_3+i_4+i_5=5$, $i_2+ai_5\equiv i_3+bi_5\equiv i_4+ci_5\equiv 0(\Mod 5)$. If $i_5\equiv 0(\Mod 5)$, then $i_1\equiv i_2\equiv i_3\equiv i_4\equiv 0(\Mod 5)$. So, $M=x_1^5,x_2^5,x_3^5,x_4^5$ or $x_5^5$. If $0< i_5< 5$, then $0< i_2,i_3,i_4< 5$, and $M=x_2x_3x_4x_5x_j$ for some $1\leq j\leq 5$. Furthermore, if $x_2x_3x_4x_5x_{j_1}\in F$ and $x_2x_3x_4x_5x_{j_2}\in F$ then $j_1=j_2$. To sum up, in the case of $abc\neq 0$, up to changing of coordinates, $F=x_1^5+x_2^5+x_3^5+x_4^5+x_5^5+\lambda x_2x_3x_4x_5x_j$ for some $j$. Hence $F$ belongs to case (i), (ii) or (iii) of  Theorem \ref{thms5}.

This completes the proof.
\end{proof}

Next we prove the following theorem:

\begin{theorem}\label{thm:greaterthan125}
Suppose $|\Aut(X)_5|\ge 125$. Then, $\Aut(X)$ is isomorphic to a subgroup of one of the three groups in Examples (1), (2), (5)  in Example \ref{mainex}.
\end{theorem}

\begin{proof}
If $|\Aut(X)_5|=5^3$, then by Proposition \ref{leAb}, $\Aut(X)_5$ is abelian. If $|\Aut(X)_5|\geq 625$, then by a theorem of Burnside  (Theorem \ref{thm:Burnsidemax} below), $\Aut(X)_5$ has a maximal normal abelian subgroup of order $\geq 125$. So, we may consider the 5 cases in Theorem \ref{thms5}.

Case (i) $X: x_1^5+x_2^5+x_3^5+x_4^5+x_5^5=0$. Then $\Aut(X)\cong  C_5^4\rtimes S_5$.

Case (ii) $X:$  $x_1^5+x_2^5+x_3^5+x_4^5+x_5^5+ax_1x_2x_3x_4x_5=0$, $a\neq 0$. Then using the differential method introduced in Section \ref{ss:differentialmethod}, we can show that $\Aut(X)$ is generated by semi-permutation matrices. So $\{x_1^5+x_2^5+x_3^5+x_4^5+x_5^5=0\}$ is preserved by $\Aut(X)$. Hence $\Aut(X)$ is a subgroup of the group $C_5^4\rtimes S_5$.

Case (iii) $X:$  $x_1^5+x_2^5+x_3^5+x_4^5+x_5^5+ax_1x_2x_3x_5^2=0$,  $a\neq 0$. Then $\Aut(X)$ is a subgroup of $C_5^4\rtimes S_5$. The proof is the same as in the case (ii).

Let us consider the cases (iv), (v). In these cases, the differential method  reduces to problem of automorphisms of plane curves and points. First consider the case (iv). 

Case (iv) $X:$  $x_1^5+x_2^5+x_3^5+x_4^5+x_5^5+ax_1x_2^3x_3+bx_1^2x_2x_3^2=0$,  where either $a$ or $b$ not equal to zero. By the differential method, we can show that if $A\in \GL(5,\C)$ and $A(F)=\lambda F$, then $A=\Diag(B,\alpha,\beta)$, where $B\in \GL(3,\C)$ and $B(G)=\lambda G$,$G=x_1^5+x_2^5+x_3^5+ax_1x_2^3x_3+bx_1^2x_2x_3^2$. Notice that we have the following exact sequence:

$$1\rightarrow C_5^2\xrightarrow{\psi} \Aut(X)\xrightarrow{\varphi}\Aut(C)\rightarrow 1\, ,$$

where $\psi (a,b)=[\Diag(1,1,1,\xi_5^a,\xi_5^b)]$, $\varphi ([\Diag(B,\alpha,\beta)])=[B]\in \Aut(C).$

We can compute $\Aut(C)$ using \cite[Theorem~2.1]{Ha13}. First, notice that $D_{10}$ is a subgroup of $\Aut(C)$, which is generated by $[\Diag(1,\xi_5,\xi_5^2)]$ and interchanging $x_1$ and $x_3$. Notice the following Lemma:

\begin{lemma}
Suppose $C$ is a smooth plane curve of degree $5$, then $C_2^2$ can not be a subgroup of $\Aut(C).$
\end{lemma}

    \begin{proof}
    Similar to Section \ref{ss:Sylow2} (and in fact much easier). We leave details to the readers.
    \end{proof}
        
  So, $\Aut(C) $ doesn't belong to case (a-i), (b-ii), or (c) of \cite[Theorem~2.1]{Ha13}. If $\Aut(C)$ belongs to the case (a-ii) or (b-i) of \cite[Theorem~2.1]{Ha13}, then $\Aut(C)\cong D_{10}$.  Therefore, $\Aut(X)$ is an extension of $D_{10}$ by $C_5^2$. Actually, $\Aut(X)\cong C_5^3\rtimes C_2$, generated by $[\Diag(1,1,1,\xi_5,1)],[\Diag(1,1,1,1,\xi_5)], [\Diag(1,\xi_5,\xi_5^2,1,1)]$ and the involution $x_1 \leftrightarrow x_3$. In particular, $\Aut(X)$ is isomorphic to a subgroup of the group $C_5^4\rtimes S_5$.

Finally, we treat the case (v). In this case $X:$  $x_1^5+x_2^5+x_3^5+G(x_4,x_5)=0$, where $G$ can not be written as $x_4^5+x_5^5$ under any change of coordinates. Again, applying the differential method, we can easily show that if $A\in \GL(5,\mathbb{C})$ and $A(F)=F$, then $A=\Diag(\alpha,\beta,\gamma,B)$, where $B\in \GL(2,\mathbb{C})$ and $B(G)=G$. Then applying Lemma \ref{lem:Autfivepoints} below, one finally finds that $\Aut(X)$ is isomorphic to either $(C_5^3\rtimes S_3)\times D_6$, $(C_5^3\rtimes S_3)\times C_4$, $(C_5^3\rtimes S_3)\times C_2$, or $ C_5^3\rtimes S_3$. Here the factor $C_5^3\rtimes S_3$ generated by $[\Diag(\xi_5,1,1,1,1)]$, $[\Diag(1,\xi_5,1,1,1)]$, $[\Diag(1,1,\xi_5,1,1)]$ and permutations of the first three variables $x_1,x_2,x_3$ and the last factor $D_6$, $C_4$, $C_2$ is the corresponding automorphism groups in Lemma \ref{lem:Autfivepoints}. Hence $\Aut(X)$ is isomorphic to a subgroup of Examples (1),(2),(5) in Example \ref{mainex}.

This complete the proof.
\end{proof}

\begin{theorem}\label{thm:Burnsidemax}
(See, for example, \cite[Chapter 2,~Corollary 2~to~Theorem~1.17]{Su82}) Let $A$ be an abelian normal subgroup of maximal order of a $p$-group $G$. If $|G|=p^n$ and $|A|=p^a$, we have $2n\leq a(a+1)$.
\end{theorem}

\begin{lemma}\label{lem:Autfivepoints}
Let $H=H(x_1,x_2)$ be a degree five homogeneous polynomial. Suppose $\{H=0\}$ is a set of five distinct points in $\P^1.$  Let $G\subset \PGL(2,\C)$ be the group which is generated by matrices leaving $H$ invariant. Then, as an abstract group, $G$ has five possibilities:

1) $G\cong C_5\rtimes C_2$, an example of $H$: $x_1^5+x_2^5$;

2) $G\cong S_3$, an example of $H$: $x_1^4x_2+x_2^4x_1$;

3) $G\cong C_4$, an example of $H$: $x_1^4x_2+x_2^5$;

4)  $G\cong C_2$, an example of $H$: $x_1^4x_2+x_2^5+x_1^2x_3$;

5)  $G$ trivial group, an example of $H$: $x_1^5+x_2^5+x_1^4x_2$.
\end{lemma}

\begin{proof}
First, it is easy to show that all possible non-trivial Sylow subgroups of $G$ are: $C_2$, $C_4$, $C_3$, $C_5$. And then it is easy to show $G$ has exactly five possibilities. Note that for a given \lq\lq{}nice\rq\rq{} (for example, those in the theorem) $H$, it is possible to directly compute (e.g. using Mathematica) all possible matrices which leave $H$ invariant. 

The detailed proof is by direct computation, so we skip it.
\end{proof}

\section{The cases where $|G|$ divides $2^63^25^2$}\label{ss:2^63^25^2}

{\it In this section, $X$ is a smooth quintic threefold defined by $F$.}

\subsection{Solvable groups with order smaller than $2000$}

In the proof of Theorem \ref{solvableliftable}, we use the following two results known in the group theory:

\begin{theorem}\label{thm:solvablesylow}
(See, for example, \cite[Chapter 4,~Theorem 5.6]{Su86}) Let $G$ be a solvable group. We can write $$|G|=mn\;\;\;\ (m,n)=1.$$

Then, the following propositions hold.

(i) There are subgroups of order $m$.

(ii) Any two subgroups of order $m$  are conjugate.

(iii) Any subgroup whose order divides $m$ is contained in a subgroup of order $m$.
\end{theorem}

\begin{theorem}\label{thm:Burnsidenormalp}
(Burnside normal p-complement Theorem, see \cite[ Theorem II,~Section 243]{Bu11}) If a Sylow $p$-subgroup of a finite group $G$ is in the center of its normalizer then $G$ has a normal $p$-complement. (Here,  a {\it normal $p$-complement} of a finite group for a prime $p$ is a normal subgroup of order coprime to $p$ and index a power of $p$. In other words the group is a semidirect product of the normal $p$-complement and any Sylow $p$-subgroup.)
\end{theorem}

\begin{theorem}\label{solvableliftable}
 Let $G$ be a finite solvable group such that  $|G|$ divides $2^63^25^2$ and $|G|\leq 2000$. Suppose $G< \Aut(X)$ and $G_2\neq 1$. Then $G$ is $F$-liftable.
\end{theorem}

\begin{proof}
By Theorem \ref{liftable}, $G$ is $F$-liftable if $G_5$ is trivial.

From now on, we may suppose that $G_5$ not trivial and hence $G_5\cong C_5$ or $C_5^2$. Then $|G|=2^{a_2}3^{a_3}5^{a_5}$, $7> a_2> 0$ and $a_5=1$ or $2$. Since $G$ is solvable, by Theorem \ref{thm:solvablesylow}, $G$ has a subgroup of order $2^{a_2}5^{a_5}$, say $H$.

If $H$ has a subgroup of order $2\cdot 5^{a_5}$, $G$ is $F$-liftable by Lemma \ref{10or50liftable} and Theorem \ref{liftable}.

If  $H$ has no subgroup of order $2\cdot 5^{a_5}$, then $|H|=2^45^{a_5}$ and the normalizer of $H_5$ inside $H$ is $H_5$ itself  by Sylow Theorems and $|H|| 2^65^2$. Then by Theorem \ref{thm:Burnsidenormalp}, $H_2$ is normal in $H$. However, by classification of subgroups of $\Aut(X)$ of order $16$ (see Section \ref{ss:Sylow2}), we have the order of automorphism group (which can be quickly computed by GAP) of $H_2$ is not divided by $5$. Then $H\cong H_2\times H_5$, contradicting to  the assumption that $H$ has no subgroups of order $2\cdot 5^{a_5}$.

Therefore, $H$ always has a subgroup of order $2\cdot 5^{a_5}$, and we are done.

\end{proof}

Let $G$ be a finite solvable group such that  $|G|$ divides $2^63^25^2$ and $|G|\leq 2000$. Using GAP, one finds all possible groups $G$ (up to isomorphism) which also satisfy the condition: $G$ is isomorphic to a subgroup of the 22 groups in the Examples (1)-(22) in Example \ref{mainex}. In fact, there are $184$ such groups (including the trivial group). (Their GAP IDs and structure descriptions can be find on the website \cite{Yu}.)

\begin{theorem}\label{thm:solvableless2000}
Let $G$ be a finite solvable group such that  $|G|$ divides $2^63^25^2$ and $|G|\leq 2000$. Suppose $G< \Aut(X)$. Then $G$ is isomorphic to a subgroup of the 22 groups in the Examples (1)-(22) in Example \ref{mainex}.
\end{theorem}

\begin{remark}\label{rmk:tricks}
In order to prove Theorem \ref{thm:solvableless2000}, we need to exclude all the other groups except those 184 groups mentioned above. Like in Section \ref{ss:Sylow2}, we exclude groups in two steps: sub-test and case by case consideration. (See Remark \ref{rmk:stepstoexcludegroups}.)

It turns out that we need to exclude 67 groups (not including $p$-groups since we already treated them in previous sections.) in the second step (case by case consideration). Again, their GAP IDs  and structure descriptions can be found on the website \cite{Yu}.

   Although our main strategy is essentially the same as in Section \ref{ss:Sylow2}, we need more tricks to exclude those 67 groups. Roughly speaking, besides the smoothness of $X$ and $F$-liftability of $G$, we combine the following six tricks in various ways:

a) small order elements or subgroups consideration;

b) character table observation;

c) large abelian subgroup consideration;

d) invariant quintic consideration;

e) group structure consideration;

f) no five dimensional faithful representation.\\[.1cm]

In the proof of Theorem \ref{thm:solvableless2000}, we will only select some typical examples (more precisely, 9 of them) among those 67 groups and show how to use the above six tricks to exclude them in details. However, all the other groups can be excluded in similar ways and more details can be found on the website \cite{Yu}.
\end{remark}

Let us prove Theorem \ref{thm:solvableless2000}. In the proof, we often denote groups by their GAP IDs.

\begin{proof}

As mentioned in Remark \ref{rmk:tricks} above, we exclude groups inductively and it turns out that we are reduced to exclude 67 groups. We will show in details how to exclude the following 9 groups among these 67 groups: $[12,5]$, $[24,11]$, $[40,3]$, $[40,7]$, $[48,5]$, $[72,39]$, $[150,9]$, $[400,50]$, and $[480,257]$.

\medskip

 \begin{lemma}\label{lem:3times2}
Let $G<\Aut(X)$. Suppose $G\cong C_2\times C_3$. Let $\widetilde{G}$ be an $F$-lifting of $G$. Then, up to change of coordinates,  $\widetilde{G}$ is generated by either

 (i) $\Diag(\xi_3,\xi_3^2,1,1,1)$ and $\Diag(1,1,-1,1,1)$;

 (ii)   $\Diag(\xi_3,\xi_3^2,\xi_3,\xi_3^2,1)$ and $\Diag(1,1,-1,-1,1)$; or
 
 (iii)  $\Diag(\xi_3,\xi_3^2,\xi_3,\xi_3^2,1)$ and $\Diag(1,1,-1,1,1)$.
 
 In particular, if $A(F)=F$ and $\Ord([A])=\Ord(A)=6$, then ${\rm tr}(A)=0,1,\xi_3^2-\xi_3$ or $-\xi_3^2+\xi_3$
\end{lemma}

\begin{proof}
As before, the essential idea is to use $F$-liftablility of $G$, smoothness of $X$ (Proposition \ref{pp:nonsmoothquintic}) and Mathematica.

Since $\widetilde{G}$ is an $F$-lifting of $G$, we may assume $\widetilde{G}=\langle [A_1],[A_2]\rangle $,  where $A_1,A_2$ are diagonal matrices and $\Ord (A_1)=3, \Ord(A_2)=2$. Then by Lemma \ref{order3power}, we may assume $A_1=\Diag(\xi_3,\xi_3^2,1,1,1)$ or $\Diag(\xi_3,\xi_3^2,\xi_3,\xi_3^2,1)$.

Case 1) $A_1=\Diag(\xi_3,\xi_3^2,1,1,1)$.

  Then by $A_1(F)=F$ and the smoothness of $F$,  both $x_1^4x_2$ and $x_2^4x_1$ are in $F$. Then by $A_2(F)=F$ we have $A_2=\Diag(1,1,\pm 1,\pm 1,\pm 1)$. Then we may assume $A_2=\Diag(1,1, -1,1,\pm 1)$ and $x_3^4x_4\in F$. If $A_2=\Diag(1,1, -1,1,-1)$ then $A_1(F)=A_2(F)=F$ implies $F\in ( x_4) +( x_1,x_2) ^2$. This is a contradiction by Proposition \ref{pp:nonsmoothquintic}. So $A_2=\Diag(1,1, -1,1,1)$.

Case 2) $A_1=\Diag(\xi_3,\xi_3^2,\xi_3,\xi_3^2,1)$ and $A_2$ has $-1$ as eigenvalue of multiplicity $2$. 

By $A_1(F)=F$ and the smoothness of $X$ we have $x_5^5\in F$. So $A_2=\Diag(\pm 1,\pm 1,\pm 1,\pm 1, 1)$.

By $A_1(F)=F$ we may assume $x_1^4x_2\in F$. Then $A_2=\Diag(\pm 1, 1,\pm 1,\pm 1, 1)$. Then either $x_2^4x_1$ or $x_2^4x_3\in F$. 

If $x_2^4x_1\in F$ then $A_2=\Diag( 1, 1,-1,- 1, 1)$. So $A_1=\Diag(\xi_3,\xi_3^2,\xi_3,\xi_3^2,1)$ and $A_2=\Diag( 1, 1,-1,- 1, 1)$, which is the same as in case ii) of the Lemma \ref{lem:3times2}. 

If $x_2^4x_3\in F$ then $A_2=\Diag( -1, 1,1,- 1, 1)$. Interchanging coordinates $x_1$ and $x_3$, we also get case ii) of the Lemma \ref{lem:3times2}

Case 3): Suppose $A_1=\Diag(\xi_3,\xi_3^2,\xi_3,\xi_3^2,1)$ and $A_2$ has $-1$ as eigenvalue of multiplicity $1$. Clearly, interchanging coordinates if necessary we may assume  $A_1=\Diag(\xi_3,\xi_3^2,\xi_3,\xi_3^2,1)$ and $A_2=\Diag(1,1,-1,1,1)$, which is case iii) of the Lemma \ref{lem:3times2}

\end{proof}

\begin{lemma}\label{C_2^2C_3}
The group $[12,5]\cong C_2^2\times C_3$ is not a subgroup of $\Aut(X)$.
\end{lemma}

\begin{proof}
We mainly  use tricks a) and d) in Remark \ref{rmk:tricks}.

\medskip

Assume to the contrary that $G< \Aut(X)$ and $G\cong C_2^2\times C_3$.

By Theorem \ref{liftable}, $G$ has an $F$-lifting, say $\widetilde{G}$. We may assume  $\widetilde{G}=\langle A_1,A_2,A_3\rangle $, $\Ord(A_1)=3$, $\Ord(A_2)=\Ord(A_3)=2$, and $A_i$ are diagonal matrices for all $i$. 

By Lemma \ref{lem:3times2}, we may consider three cases.

Case i) $A_1=\Diag(\xi_3,\xi_3^2,1,1,1)$ and $A_2=\Diag(1,1,-1,1,1)$. 

Then $x_1^4x_2, x_2^4x_1\in F$. We may assume $x_3^4x_4\in F$. Then $A_3=\Diag(1,1,\pm 1,1,\pm 1)$. Replacing $A_3$ by $A_2A_3$ if necessary, we may assume $A_3=\Diag(1,1,1,1,- 1)$. Then $A_1(F)=A_2(F)=A_3(F)=F$ implies that $F\in ( x_4) +( x_1,x_2) ^2$, which is a contradiction by Proposition \ref{pp:nonsmoothquintic}.

Case ii)  $A_1=\Diag(\xi_3,\xi_3^2,\xi_3,\xi_3^2,1)$ and $A_2=\Diag(1,1,-1,-1,1)$.

  Then $x_1^4x_2,x_2^4x_1,x_3^4x_2,x_4^4x_1$ and $x_5^5\in F$. Clearly $A_3=\Diag(1,1,\pm 1,\pm 1,1)$ and hence we may assume $A_3=\Diag(1,1, 1, -1,1)$. Then $A_1=\Diag(\xi_3,\xi_3^2,\xi_3,\xi_3^2,1)$, $A_2=\Diag(1,1,-1,-1,1)$ and $A_3=\Diag(1,1, 1, -1,1)$. Then $A_i(F)=F,i=1,2,3$ implies $F\in ( x_1) +( x_3,x_5) ^2$, a contradiction.

Case iii) $A_1=\Diag(\xi_3,\xi_3^2,\xi_3,\xi_3^2,1)$ and $A_2=\Diag(1,1,-1,1,1)$.

  Then $x_2^4x_1,x_4^4x_1$ and $x_5^5\in F$. Then we may assume $A_3=\Diag(1,\pm 1,1,\pm 1,1)$. Since either $x_1^4x_2$ or $x_1^4x_4\in F$, we may assume $x_1^4x_2\in F$ and $A_3=\Diag(1, 1,1,- 1,1)$. Then  $A_1=\Diag(\xi_3,\xi_3^2,\xi_3,\xi_3^2,1)$, $A_2=\Diag(1,1,-1,1,1)$ and $A_3=\Diag(1, 1,1,- 1,1)$. By $A_1(F)=A_2(F)=A_3(F)=F$, as in case ii),  we get $F\in ( x_1) +( x_3,x_5) ^2$, a contradiction.
\end{proof}

\begin{lemma}\label{C_3XC_4}
Suppose $C_3\times C_4 \cong G < \Aut(X)$. Let $\widetilde{G}$ be an $F$-lifting of $G$. Then, up to change of coordinates, we may assume $\widetilde{G}$ is generated by either one of (i)-(v) below:

i) $\Diag(\xi_3,\xi_3^2,1,1,1),\Diag(1,1,\xi_4,1,1)$;

ii) $\Diag(\xi_3,\xi_3^2,1,1,1),\Diag(1,1,\xi_4,-1,1)$;

iii) $\Diag(\xi_3,\xi_3^2,\xi_3,\xi_3^2,1),\Diag(1,1,\xi_4,1,1)$;

iv)  $\Diag(\xi_3,\xi_3^2,\xi_3,\xi_3^2,1),\Diag(1,1,\xi_4,\xi_4^3,1)$;

v)  $\Diag(\xi_3,\xi_3^2,\xi_3,\xi_3^2,1),\Diag(1,1,\xi_4,\xi_4,1)$.

\end{lemma}

\begin{proof}
Similar to the proof of Lemma \ref{lem:3times2}.
\end{proof}

\begin{lemma}
The group $[24,11]=C_3\times Q_8$ is not a subgroup of $\Aut(X)$.
\end{lemma}

\begin{proof}
We mainly  use tricks a), b) d), e) in Remark \ref{rmk:tricks}.

\medskip

Assume to the contrary that $G< \Aut(X)$ and $G\cong C_3\times Q_8$. Let $\widetilde{G}$ be an $F$-lifting of $G$. Then $\widetilde{G}$ can be naturally viewed as a five dimensional faithful representation of $C_3\times Q_8$. 

Every  matrix in  $\widetilde{G}$ with order $4$ must have both $\xi_4$ and $\xi_4^3$ as eigenvalues by representation theory of $Q_8$.

Let $A_1$ and $A_2$ be  $\Diag(\xi_3,\xi_3^2,\xi_3,\xi_3^2,1)$ and $\Diag(1,1,\xi_4,\xi_4^3,1)$ respectively. By Lemma \ref{C_3XC_4}, we may assume $\langle A_1,A_2\rangle $ is contained in $\widetilde{G}$. Therefore, the five dimensional faithful representation of $Q_8$ induced by $\widetilde{G}$ is of type $2\oplus1\oplus 1\oplus 1$. 

Then all other matrices in $\widetilde{G}$  with order $4$ must of form $\Diag(I_2,B,1)$, where $I_2$ is the $2\times2$ identity matrix and $B\in \GL(2,\C)$. 

On the other hand, $A_1$ commutes with  $\Diag(I_2,B,1)$, hence $B$ is a diagonal matrix. Then we have a contradiction since $Q_8$ is not an abelian group.
\end{proof}

\begin{lemma}\label{lem:2times5}
Let $G< \Aut(X)$. Suppose $G\cong C_2\times C_5$. Let $\widetilde{G}$ be an $F$-lifting of $G$. Then, up to change of coordinates,  $\wt{G}$ is generated by either

 (i) $\xi_5^i\cdot\Diag(1,1,\xi_5,\xi_5^a,\xi_5^b)$ and $\Diag(-1,1,1,1,1)$ for some $i,a,b$; or
 
 (ii) $\xi_5^i\cdot\Diag(1,1,\xi_5^a,\xi_5^a,\xi_5^b)$ and $\Diag(-1,1,-1,1,1)$ for some $i,a,b$.
 \end{lemma}
 
 \begin{proof}
Similar to the proof of Lemma \ref{lem:3times2}.
\end{proof}

 \begin{lemma}\label{lem:[40,3]}
$[40,3]\cong C_5\rt C_8$ is not a subgroup of $\Aut(X)$.
\end{lemma}
 
 \begin{proof}
We mainly  use tricks a) and b) in Remark \ref{rmk:tricks}.

\medskip
 
 Assume to the contrary. Then there exists a five dimensional faithful representation, say $\rho$, of $[40,3]$ such that $\rho$ leaves $F$ invariant.
 
 By character table of $[40,3]$ (see Figure \ref{table:[40,3]}), $\rho$ must be of type $4\oplus 1$. The four dimensional irreducible component of $\rho$ must be $X.10$. Then by Lemma \ref{lem:2times5}, $[40,3]$ is  excluded.
 
 \begin{figure}[htbp]
\begin{center}

\includegraphics[width=20cm, height=22cm]{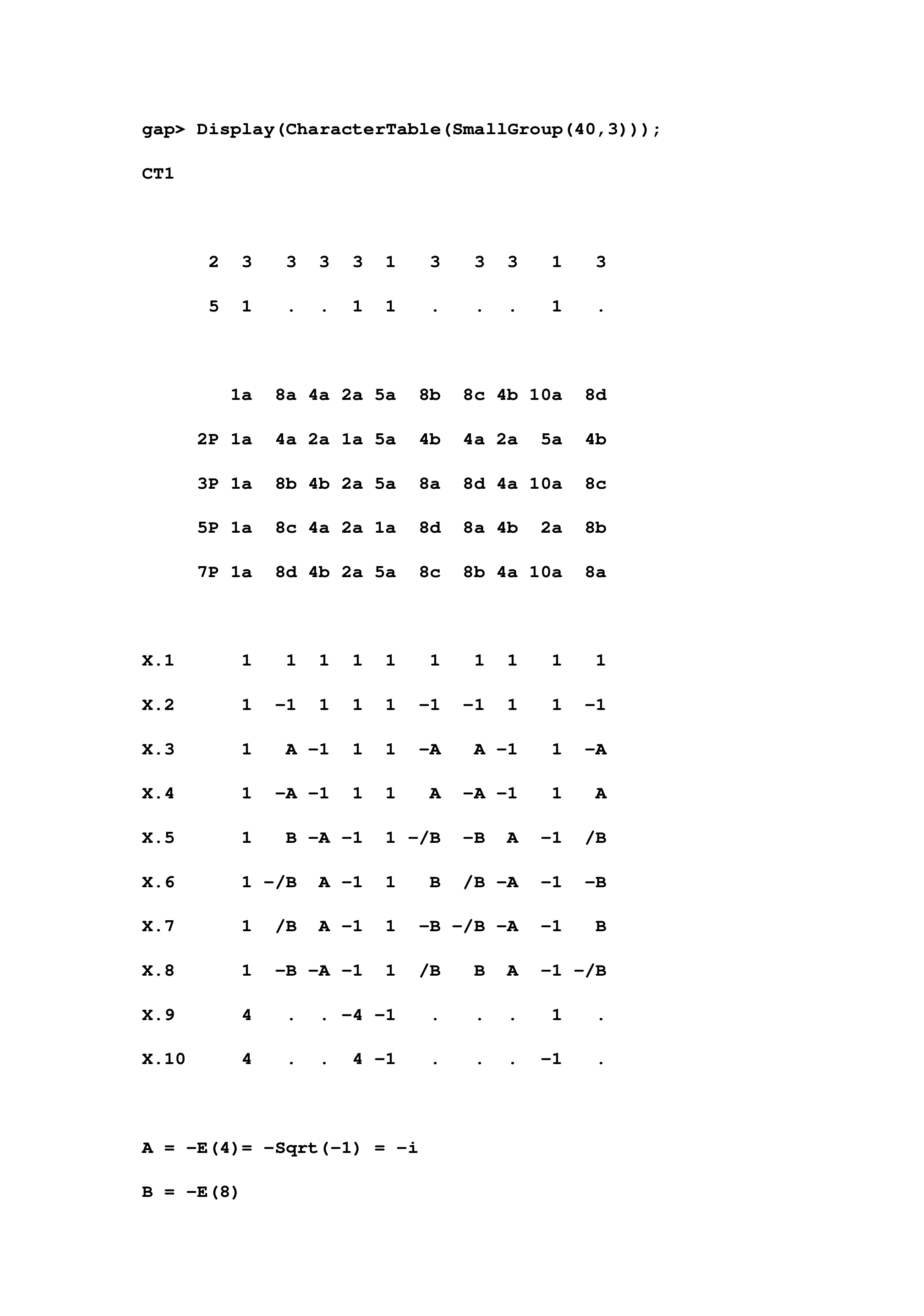}

\caption{Character table of $[40,3]$}
\label{table:[40,3]}
\end{center}
\end{figure}

  \end{proof}

\begin{lemma}\label{lem:2times2times5}
Let $G< \Aut(X)$. Suppose $G\cong C_2\times C_2\times C_5$. Let $\widetilde{G}$ be an $F$-lifting of $G$. Then, up to change of coordinates,  $\wt{G}$ is generated by either

 (i) $\xi_5^i\cdot\Diag(1,1,\xi_5,\xi_5,\xi_5^a)$, $\Diag(-1,1,1,1,1)$ and $\Diag(1,1,-1,1,1)$ for some $i,a$;
 
 (ii) $\xi_5^i\cdot\Diag(1,1,\xi_5,1,1)$, $\Diag(-1,1,1,1,1)$ and $\Diag(1,1,1,-1,1)$ for some $i$; or
 
 (iii) $\xi_5^i\cdot\Diag(1,1,1,1,\xi_5)$, $\Diag(-1,1,-1,1,1)$ and $\Diag(1,1,-1,-1,1)$ for some $i$.
\end{lemma}
 
  \begin{proof}
Similar to the proof of Lemma \ref{lem:3times2}.
\end{proof}

 \begin{lemma}\label{lem:225cent}
 Let $H< G< \Aut(X)$ and $K\lhd G$. Suppose $H\cong C_2\times C_2$ and $K\cong C_5$  (so that $H\cap K=1$, and $H\times K$ can be viewed as a subgroup of $G$). Then $C_G(H)=C_G(H\times K)$.
 \end{lemma}
 
 \begin{proof}
 Let $\wt{H}$ and $\wt{K}$ be $F$-liftings of $H$ and $K$ respectively. Then,  we have three cases, according to the $3$ cases (i)-(iii) in Lemma \ref{lem:2times2times5}: 
 
 Case i)  $\wt{H}=\langle A_1:=\Diag(-1,1,1,1,1)$, $A_2:=\Diag(1,1,-1,1,1)\rangle$  and $\wt{K}=\langle A_3:=\xi_5^i\cdot\Diag(1,1,\xi_5,\xi_5,\xi_5^a) \rangle$.
 
 Let $[B]\in C_G(H)$. It follows that $BA_1=A_1B$, and $BA_2=A_2B$. Then the matrix $B$ must be of \lq\lq{}special\rq\rq{} form. Since $K$ is normal in $G$, we have $BA_3B^{-1}=\xi_5^j A_3^k$. Then by a direct computation, we must have $j=0,k=1$. So $[B]\in C_G(H\times K)$.
 
 The remaining two cases (ii) and (iii) are similar.
 \end{proof}

\begin{lemma}\label{lem:[40,7]}
$[40,7]\cong C_2\times (C_5\rt C_4)$ is not a subgroup of $\Aut(X)$.
\end{lemma}

\begin{proof}

We mainly  use tricks a), b), and e) in Remark \ref{rmk:tricks}.

\medskip

Assume to the contrary, $[40,7]\cong G< \Aut(X)$. $\wt{G}$ is an $F$-lifting of $G$ and $\rho$ is the corresponding 5 dimensional representation of $[40,7]$.

Notice that $[40,7]$ has a subgroup $H\cong C_2\times C_2\times C_5$ such that $H$ contains elements belongs to conjugacy classes $2a$, $2b$, $2c$, $5a$ in the  character table of $[40,7]$ (see Figure \ref{table:[40,7]}).

 \begin{figure}[htbp]
\begin{center}

\includegraphics[width=20cm, height=22cm]{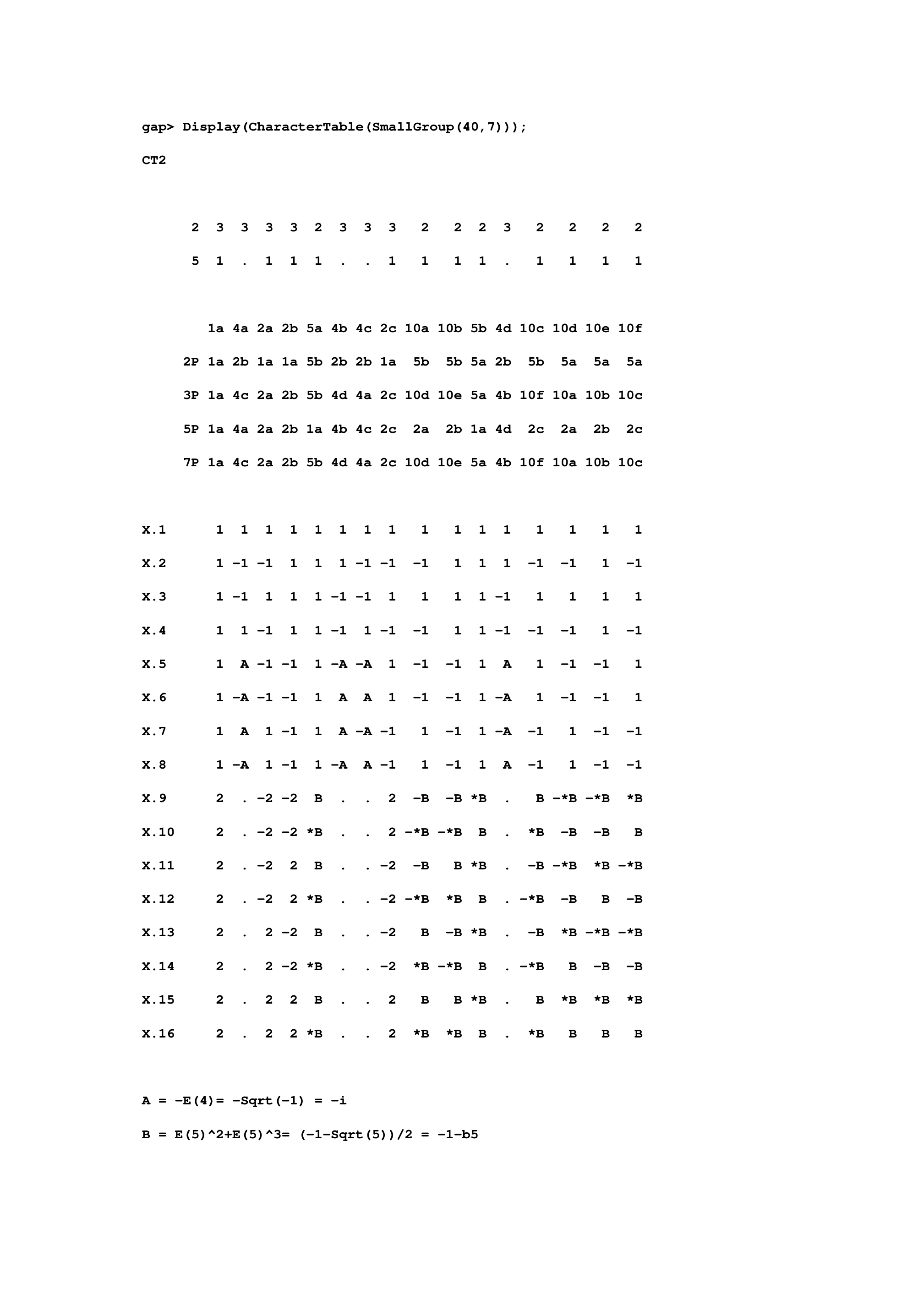}

\caption{Character table of $[40,7]$}
\label{table:[40,7]}
\end{center}
\end{figure}

By Lemma \ref{lem:2times2times5}, we may assume that $H$ is generated by $\Diag(-1,1,1,1,1)$, $\Diag(1,1,-1,1,1)$ and $\Diag(\xi_5^2,\xi_5^2,\xi_5^3,\xi_5^3,1)$.

Then $\rho$ must be of type $2\oplus 2\oplus 1$. Notice that by values of $4a$ the one dimensional representation can not be $X.3$. Therefore $\rho$ is of the form $2\oplus 2\oplus X.1$. Then by the character table of $[40,7]$, $2a,2b,2c$ all have $-1$ as eigenvalue of multiplicity two. 

 But $H$ has exactly one matrix which is of order 2 and has $-1$ as eigenvalue of multiplicity two, a contradiction. 
\end{proof}

\begin{lemma}\label{lem:3times4}

Suppose $C_3\times C_4 \cong G < \Aut(X)$. Let $\widetilde{G}$ be an $F$-lifting of $G$. Then, up to change of coordinates, we may assume $\widetilde{G}$ is generated by either one of:

i) $\Diag(\xi_3,\xi_3^2,1,1,1),\Diag(1,1,\xi_4,1,1)$.

ii) $\Diag(\xi_3,\xi_3^2,1,1,1),\Diag(1,1,\xi_4,-1,1)$.

iii) $\Diag(\xi_3,\xi_3^2,\xi_3,\xi_3^2,1),\Diag(1,1,\xi_4,1,1)$.

iv) $\Diag(\xi_3,\xi_3^2,\xi_3,\xi_3^2,1),\Diag(1,1,\xi_4,\xi_4^3,1)$.

v) $\Diag(\xi_3,\xi_3^2,\xi_3,\xi_3^2,1),\Diag(1,1,\xi_4,\xi_4,1)$.
\end{lemma}

  \begin{proof}
Similar to the proof of Lemma \ref{lem:3times2}.
\end{proof}

\begin{lemma}\label{lem:3times8}

Suppose $C_3\times C_8 \cong G < \Aut(X)$. Let $\widetilde{G}$ be an $F$-lifting of $G$. Then, up to change of coordinates, we may assume $\widetilde{G}$ is generated by $\Diag(\xi_3,\xi_3^2,1,1,1),\Diag(1,1,\xi_8,\xi_8^{-4},1)$.
\end{lemma}

  \begin{proof}
Use Lemma \ref{lem:3times4}.
\end{proof}

\begin{lemma}\label{lem:[48,5]}
$[48,5]\cong C_{24}\rtimes C_2$ is not a subgroup of $\Aut(X)$.
\end{lemma}

\begin{proof}
We mainly  use tricks a) and e) in Remark \ref{rmk:tricks}.

\medskip

Assume to the contrary, $[48,5]\cong G< \Aut(X)$. Let $\wt{G}$ be an $F$-lifting of $G.$ By the structure of $[48,5]$, we see that $\wt{G}\cong (C_3\times C_8)\rtimes C_2$  and Sylow $2$-subgroup of $\wt{G}$ is not abelian. 

Let $H\cong C_3\times C_8$ be a normal subgroup of $\wt{G}$. Then by Lemma \ref{lem:3times8}, we may assume $H$ is generated by $\Diag(\xi_3,\xi_3^2,1,1,1),A:=\Diag(1,1,\xi_8,\xi_8^{-4},1)$. 

Let $B\in \GL(5,\C)$. Then, $a=1$ if $BAB^{-1}=A^{a}$ by eigenvalue considerations. However, then $\wt{G}$ must have an abelian Sylow $2$-subgroup, a contradiction.
\end{proof}

\begin{lemma}\label{lem:[72,39]}
$[72,39]\cong C_3^2\rt C_8$ is not a subgroup of $\Aut(X)$.
\end{lemma}

\begin{proof}
We mainly  use tricks b) and f) in Remark \ref{rmk:tricks}.

\medskip

If otherwise, then $[72,39]$ must have a 5 dimensional faithful representation, say $\rho$. However, by the character table of  $[72,39]$ (see Figure \ref{table:[72,39]}), $\rho$ can not exist, a contradiction.

 \begin{figure}[htbp]
\begin{center}

\includegraphics[width=20cm, height=22cm]{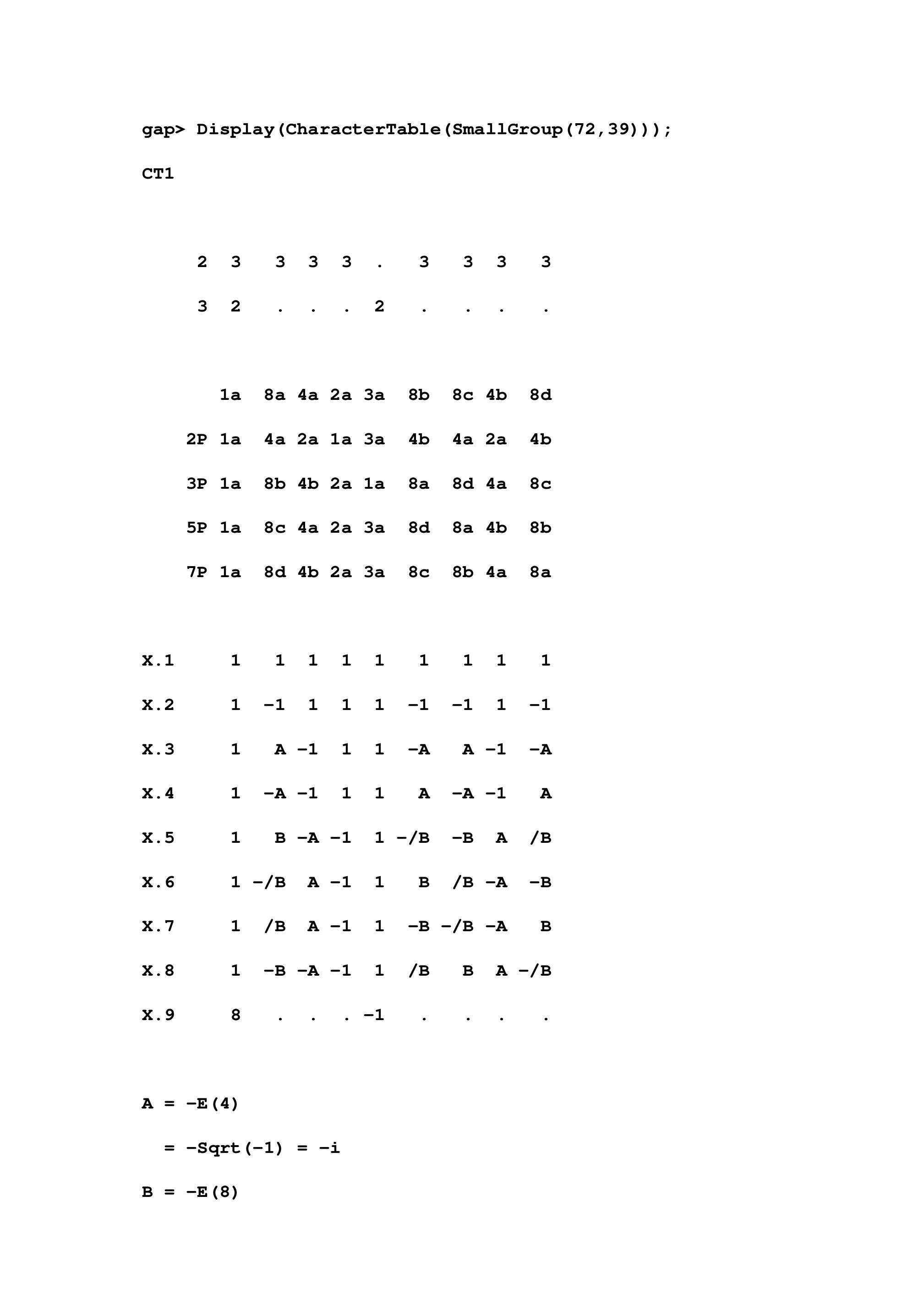}

\caption{Character table of $[72,39]$}
\label{table:[72,39]}
\end{center}
\end{figure}

\end{proof}

 \begin{lemma}\label{lem:355}
Let $G< \Aut(X)$. Suppose $G\cong C_3\times C_5\times C_5$. Let $\widetilde{G}$ be an $F$-lifting of $G$. (Existence of $\wt{G}$ can be proved similarly as before.) Then, up to change of coordinates,  $\wt{G}$ is generated by 

 (i) $\xi_5^i\cdot\Diag(1,1,\xi_5,1,\xi_5^a)$, $\xi_5^j\cdot\Diag(1,1,1,\xfi,\xi_5^b)$ and $\Diag(\xi_3,\xi_3^2,1,1,1)$ for some $i,j,a,b$; or
 
(ii) $\xi_5^i\cdot\Diag(1,1,\xfi,\xfi,1)$, $\xfi^j\cdot\Diag(1,1,1,1,\xfi)$ and $\Diag(\xi_3,\xi_3^2,\xi_3,\xi_3^2,1)$ for some $i,j$.

  In the case (ii), $X$ is isomorphic to Example (14) in Example \ref{mainex}. 

\end{lemma}

  \begin{proof}
Similar to the proof of Lemma \ref{lem:3times2}. For the last statement, we  use Mathematica to compute invariant monomials.
\end{proof}

\begin{lemma}\label{lem:[150,9]}
$[150,9]\cong C_3\ti ((C_5\ti C_5)\rt C_2)$ is not a subgroup of $\Aut(X)$.
\end{lemma}

\begin{proof}
We mainly  use tricks a), d), and e) in Remark \ref{rmk:tricks}.

\medskip

By GAP, [150,9] isomorphic to $C_3\ti ((C_5\ti C_5)\rt C_2)$ and $C_2\ti C_5$ is not a subgroup of [150,9].

It suffices to consider two cases: (i), (ii) in the Lemma \ref{lem:355}.

Case (i) $C_3\ti C_5 \ti C_5$ generated by  $\aod\xi_5^i\cdot\Diag(1,1,\xi_5,1,\xi_5^a)$, $\atd\xi_5^j\cdot\Diag(1,1,1,\xfi,\xi_5^b)$ and $\ahd\Diag(\xi_3,\xi_3^2,1,1,1)$ for some $i,j,a,b$.

  By the structure of $[150,9]$ there exists $[B]\in [150,9]$ such that $A_3B=BA_3, BA_1B^{-1}=A_1^{-1}$ and  $BA_2B^{-1}=A_2^{-1}$. It is easy to check such $B$ can not exist, a contradiction.

Case (ii)  $[150,9]$ is not a subgroup of $\Aut(X)$ if $X$ is isomorphic to Example (14) in Example \ref{mainex}, a contradiction.
 \end{proof}

\begin{lemma}\label{lem:558}
Suppose $C_5\ti C_5\ti C_8 \cong G< \Aut(X)$. Then $\Aut(X)$ is isomorphic to a subgroup of the group in Example 4 in Example \ref{mainex}.
\end{lemma}

\begin{proof}
By a similar argument as before, we may assume $\wt{G}$ is generated by $\xi_5^i\cdot\Diag(1,1,1,\xfi,1)$, $\xfi^j\cdot\Diag(1,1,1,1,\xfi)$ and $\Diag(\xi_8,\xi_8^{-4},1,1,1)$ for some $i,j$.

Then using Mathematica we can compute the invariant monomials of $\wt{G}$ and obtain $F=x_1^4x_2+x_2^4x_3+x_3^5+x_4^5+x_5^5+\lambda x_2^2x_3^3$. Now, we may apply the differential method in Section \ref{ss:differentialmethod} to get the result.
 \end{proof}
 
 \begin{lemma}\label{lem:[400,50]}
$[400,50]\cong C_5^2\rt C_{16}$ is not a subgroup of $\Aut(X)$.
\end{lemma}

\begin{proof}
We mainly  use tricks c), and e) in Remark \ref{rmk:tricks}.

\medskip

By GAP, $[400,50]$ contains a subgroup isomorphic to $C_5\ti C_5\ti C_8$. Then the result follows from Lemma \ref{lem:558}.

\end{proof}

\begin{lemma}\label{lem:o8}
Let $A\in \GL(5,\C)$. Suppose $[A]\in \Aut(X)$ and both $A$ and $[A]$ have order $8$ (as elements in $\GL(5,\C)$ and $\PGL(5,\C)$ respectively). Then, up to change of coordinates and up to odd power of $A$, $A$ is one of the followings: (i) $\Diag(\xi_8,-1,1,1,1)$, (ii) $\Diag(\xi_8,-1,1,1,\xi_8^2)$, (iii) $\Diag(\xi_8,-1,1,1,\xi_8^3)$, (iv) $\Diag(\xi_8,-1,1,1,-1)$, (v) $\Diag(\xi_8,-1$, $1,1,\xi_8^5)$, (vi) $\Diag(\xi_8,-1,1,1,\xi_8^6)$, (vii) $\Diag(\xi_8,-1,1,1,\xi_8^7)$, (viii) $\Diag(\xi_8,-1,1,\xi_8^2,\xi_8^3)$, (ix) $\Diag(\xi_8,-1,1,\xi_8^2,\xi_8^5)$, (x) $\Diag(\xi_8,-1,1,\xi_8^2,\xi_8^6)$, (xi) $\Diag(\xi_8,-1,1,\xi_8^2,\xi_8^7)$, (xii) $\Diag(\xi_8,-1,1$, $\xi_8^5,\xi_8^6)$.
\end{lemma}

\begin{proof}
Since $A(F)=F$ and $A$ has order 8, we may assume $A=\Diag(\xi_8,\xi_8^4,1,\xi_8^a,\xi_8^b)$, $0\leq a\leq b \leq 7$. Then one of $a$ and $b$ must be even, otherwise trace of $A^4$ is $-1$, a contradiction to Lemma \ref{lem:to2}.

Then for all possible pairs $(a,b)$, we compute the monomials invariant by $A$ and use the smoothness of $X$ and Proposition \ref{pp:nonsmoothquintic}. In this way, the lemma can be proved.

\end{proof}

\begin{lemma}\label{lem:[96,67]}
Suppose $[96,67]\cong G <\Aut(X)$. Then, up to change of coordinates, an $F$-lifting of $G$ is generated by the following four matrices: $$A_1=\begin{pmatrix} 
    0&1&0&0&0 \\ 
    -1&0&0&0&0 \\ 
    0&0&1&0&0 \\ 
    0&0&0&1&0 \\ 
    0&0&0&0&1 \\ 
    \end{pmatrix}, A_2=\begin{pmatrix} 
    \xi_4^3&0&0&0&0 \\ 
    0&\xi_4&0&0&0 \\ 
    0&0&1&0&0 \\ 
    0&0&0&1&0 \\ 
    0&0&0&0&1 \\ 
    \end{pmatrix}, A_3=\begin{pmatrix} 
    -\frac{1}{\sqrt{2}}\xi_8&\frac{1}{\sqrt{2}}\xi_8&0&0&0 \\ 
    \frac{1}{\sqrt{2}}\xi_8^3&\frac{1}{\sqrt{2}}\xi_8^3&0&0&0 \\ 
    0&0&\xi_3&0&0 \\ 
    0&0&0&\xi_3^2&0 \\ 
    0&0&0&0&1 \\ 
    \end{pmatrix},$$ $$A_4=\begin{pmatrix} 
    \xi_4&0&0&0&0 \\ 
    0&1&0&0&0 \\ 
    0&0&0&1&0 \\ 
    0&0&1&0&0 \\ 
    0&0&0&0&1 \\ 
    \end{pmatrix}\, .$$
\end{lemma}

\begin{proof}
Note that $[96, 67]\cong \SL(2,3)\rt C_4$.  Let $\wt{G}$ be an $F$-lifting of $G$.

Then $\wt{G}$ induces a five dimensional faithful  representation of $[96,67]$, say  $\rho: [96,67] \longrightarrow \GL(5,\C)$ such that Image of $\rho=\wt{G}$. We look at the character table of $[96,67]$ (see Figure \ref{table:[96,67]})

 \begin{figure}[htbp]
\begin{center}

\includegraphics[width=20cm, height=22cm]{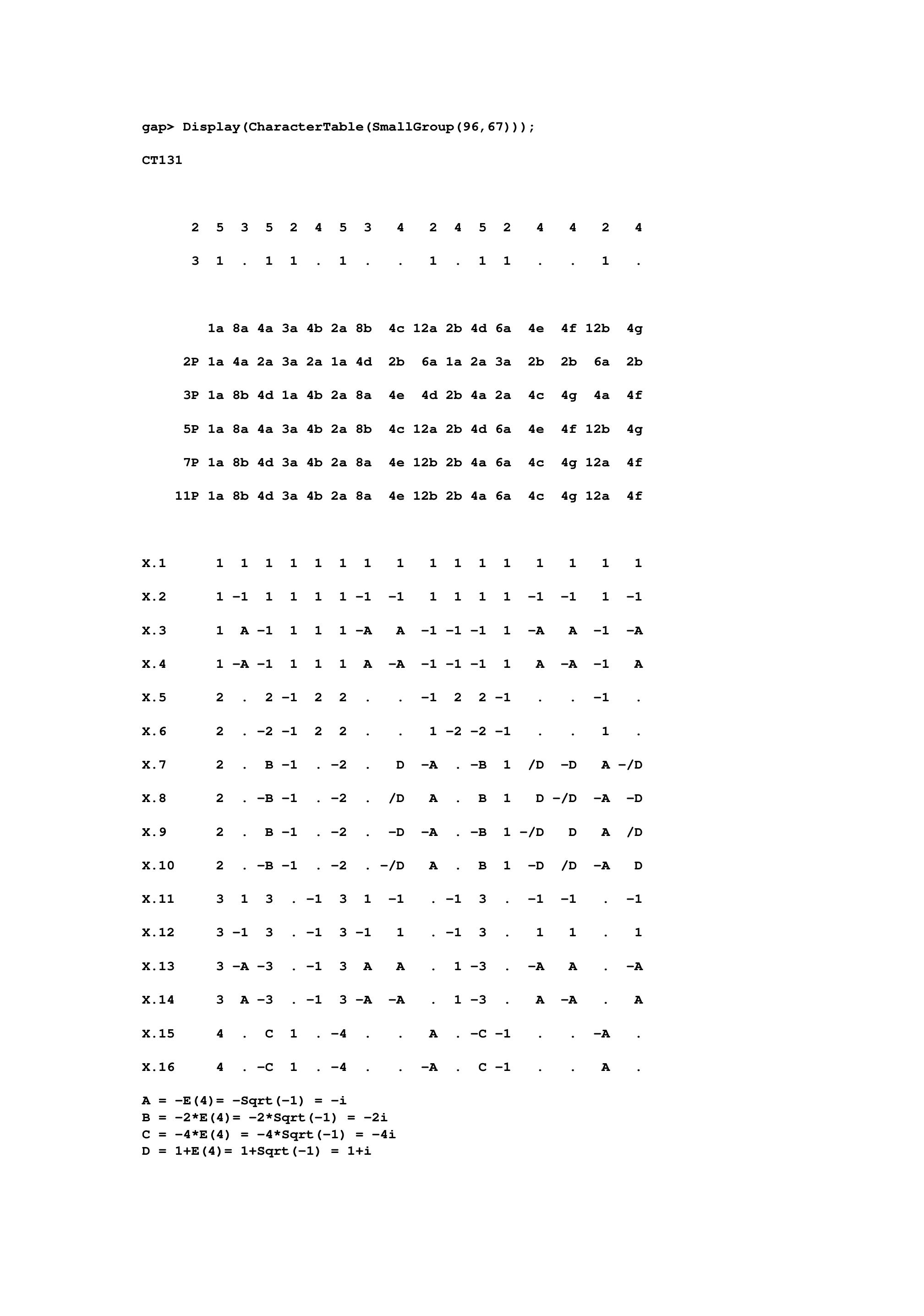}

\caption{Character table of $[96,67]$}
\label{table:[96,67]}
\end{center}
\end{figure}

Since $\wt{G}$ is not abelian, $\rho$  is one of the following types: (i) $4\oplus 1$, (ii) $3\oplus 1\oplus 1$, (iii) $3\oplus 2$, (iv) $2\oplus 1\oplus 1\oplus 1$, (v) $2\oplus 2\oplus 1$.

Case (i) $\rho$ of type $4\oplus 1$.  This case is impossible by consideration of conjugacy class \lq\lq{}2a\rq\rq{} and by Lemma \ref{lem:to2}.

Case (ii)  $\rho$ of type $3\oplus 1\oplus 1$. This case is impossible by consideration of conjugacy class \lq\lq{}2a\rq\rq{} and faithfulness of $\rho$.

Case (iii)  $\rho$ of type $3\oplus 2$. Faithfulness of $\rho$ and consideration of conjugacy class \lq\lq{}2a\rq\rq{} imply trace of conjugacy class \lq\lq{}4b\rq\rq{} is $-1$, a contradiction to Lemma \ref{lem:to4}. So this case is also impossible.

Case (iv)  $\rho$ of type $2\oplus 1\oplus 1\oplus 1$. By faithfulness of $\rho$ and consideration of conjugacy class \lq\lq{}2a\rq\rq{}, the $2$ dimensional component of $\rho$ must be one of $X.7$-$X.10$. Then trace of conjugacy class \lq\lq{}6a\rq\rq{} is 4, a contradiction to Lemma \ref{lem:3times2}. So, this case is impossible.

In sum, $\rho$ must be of type $2\oplus 2\oplus 1$. Furthermore, by consideration of conjugacy classes \lq\lq{}2a\rq\rq{} and \lq\lq{}2b\rq\rq{}, we have $\rho=\sigma \oplus X.5\oplus \tau$, $\sigma$ is one of $X.7$-$X.10$ and $\tau$ is one of $X.1$-$X.4$. By Lemma \ref{lem:o8}, $\tau$ can not be $X.2$. Then by Lemma \ref{lem:3times4} and by consideration of conjugacy classes \lq\lq{}4a\rq\rq{}, \lq\lq{}3a\rq\rq{}, $\tau$ can not be $X.3$ or $X.4$, either. Therefore, $\rho= \sigma \oplus X.5 \oplus X.1$, and  $\sigma$ is one of $X.7$-$X.10$.

Note that as characters of $[96,67]$, the complex conjugate of $X.7$ is $X.8$,  $X.7\otimes X.2$=$X.9$,  and the complex conjugate of $X.7\otimes X.2$ is $X.10$. Then, up to change of coordinates, we may assume $\wt{G}$ is generated by the four matrices $A_i$ in Lemma \ref{lem:[96,67]}. (Notice that $\langle A_1, A_2, A_3\rangle\cong \SL(2,3)$, and $\wt{G}=\langle A_1, A_2, A_3\rangle\rt \langle A_4 \rangle$.)
\end{proof}

\begin{lemma}\label{lem:[480,257]}
$[480,257]\cong (\SL(2,3)\rt C_4)\times C_5$ can not be a subgroup of $\Aut(X).$
\end{lemma}

\begin{proof}
Assume to the contrary, $[480,257]\cong G<\Aut(X)$.

By Theorem \ref{solvableliftable}, $G$ has an $F$-lifting, say $\wt{G}$. By Lemma \ref{lem:[96,67]}, we may assume $\wt{G}=\langle A_1,A_2, A_3, A_4\rangle \times \langle A_5\rangle$, where $A_1$, $A_2$, $A_3$, $A_4$ are as in Lemma \ref{lem:[96,67]}, and $A_5$ is of order $5$.

Notice that a degree five monomial $M=x_1^{a_1}...x_5^{a_5}$ is in $F$ only if $M$ satisfies both of the following two conditions:

(i)  $a_1+3a_2\equiv 0(\Mod 4)$, $a_1\equiv 0(\Mod 2)$ (as $A_2(F)=F$ and $(A_4^2)(F)=F$);

(ii) If $a_1=a_2=0$, then $a_3+2a_4\equiv 0(\Mod 3)$ (as $A_3(F)=F$).

There are exactly 16 different monomials satisfying both (i) and (ii): $x_1^4x_3$, $x_1^4x_4$, $x_1^4x_5$, $x_1^2x_2^2x_3$, $x_1^2x_2^2x_4$, $x_1^2x_2^2x_5$, $x_2^4x_3$, $x_2^4x_4$, $x_2^4x_5$, $x_3^4x_4$, $x_3^3x_5^2$, $x_3^2x_4^2x_5$, $x_3x_4^4$, $x_3x_4x_5^3$, $x_4^3x_5^2$, $x_5^5$.

Then we may write $F$ as: $$F =\lmd_1x_1^4x_3+\lmd_2x_2^4x_3+\lmd_3 x_1^2x_2^2x_3+\lmd_4 x_1^4x_4+\lmd_5x_2^4x_4+\lmd_6 x_1^2x_2^2x_4+ \lmd_7x_1^4x_5+\lmd_8x_2^4x_5+$$ $$\lmd_9 x_1^2x_2^2x_5+  \lmd_{10}x_3^4x_4+\lmd_{11}x_4^4x_3+\lmd_{12} x_3^3x_5^2+  \lmd_{13}x_4^3x_5^2+\lmd_{14}x_3x_4x_5^3+\lmd_{15} x_3^2x_4^2x_5+\lmd_{16}x_5^5.$$ 

By $A_1(F)=F$, we have $\lmd_1=\lmd_2$, $\lmd_4=\lmd_5$, $\lmd_7=\lmd_8$.

By $A_4(F)=F$, we have $\lmd_9=0$, $\lmd_{10}=\lmd_{11}$, $\lmd_{12}=\lmd_{13}$.

Notice that $A_3(x_1^4+x_2^4)$=$-\frac{x_1^4+6x_1^2x_2^2+x_2^4}{2}$. Then by $A_3(F)=F$, we have $\lmd_7=0$. Again, by $A_3(F)=F$, we have $A_3((\lmd_1(x_1^4+x_2^4)+\lmd_3x_1^2x_2^2)x_3)$=$(\lmd_1(x_1^4+x_2^4)+\lmd_3x_1^2x_2^2)x_3$, which implies $$-\lmd_1\frac{x_1^4+6x_1^2x_2^2+x_2^4}{2}+\lmd_3\frac{x_1^4-2x_1^2x_2^2+x_2^4}{4}=\xi_3^2(\lmd_1(x_1^4+x_2^4)+\lmd_3x_1^2x_2^2)$$

Then $\lmd_3=(2+4\xi_3^2)\lmd_1$. Similarly, $\lmd_6=-(2+4\xi_3^2)\lmd_4$. Then $A_4(F)=F$ implies $\lmd_1=\lmd_4$.

In sum, we may rewrite $F$ as: $$F=\lmd_1((x_1^4+x_2^4+(2+4\xi_3^2)x_1^2x_2^2)x_3+(x_1^4+x_2^4-(2+4\xi_3^2)x_1^2x_2^2)x_4)+$$ $$\lmd_{10}(x_3^4x_4+x_4^4x_3)+\lmd_{12}(x_3^3+x_4^3)x_5^2+\lmd_{14}x_3x_4x_5^3+\lmd_{15}x_3^2x_4^2x_5+\lmd_{16}x_5^5.$$

By the smoothness of $X$ and Proposition \ref{pp:nonsmoothquintic}, $\lmd_1\lmd_{10}\lmd_{16}\neq 0$.

Then  by adjusting variables $x_i$ by suitable nonzero constants, we may assume $\lmd_1=\lmd_{10}=\lmd_{16}=1$.

 Recall that, $A_5$ is of order $5$, and $A_5A_i=A_iA_5$, for $i=1,2,3,4$. Then by Lemma \ref{lem:matrixshape}, $A_5$ is diagonal, and we may assume $A_5$=$\Diag(\xi_5^a,\xi_5^b,\xi_5^c,\xi_5^d,\xi_5^e)$, for some $0\leq a, b,c,d,e\leq 4$.  
 
 Then $A_5(F)=F$ implies $A_5(x_1^4x_3)$=$x_1^4x_3$,  $A_5(x_2^4x_3)$=$x_2^4x_3$,  $A_5(x_3^4x_4)$=$x_3^4x_4$. So $a=b=c=d$. So we may assume $A_5=\Diag(1,1,1,1,\xi_5)$. Then $A_5(F)=F$ implies $\lmd_{12}=\lmd_{14}=\lmd_{15}=0$.
 
 To sum up above, we have proved that if $[480,257]$ is a subgroup of $\Aut(X)$, up to change of coordinates, we may assume:
 
 \begin{equation}\label{eq:[480,257]}
 F=(x_1^4+x_2^4+(2+4\xi_3^2)x_1^2x_2^2)x_3+(x_1^4+x_2^4-(2+4\xi_3^2)x_1^2x_2^2)x_4+x_3^4x_4+x_4^4x_3+x_5^5.
\end{equation}

However, then $X$ is singular by a direct computation. Therefore, $[480,257]$ is excluded.

\end{proof}

In this way, we exclude the $67$ groups remained after sub-test to obtain Theorem \ref{thm:solvableless2000}  (see the website \cite{Yu} for details of the rest $67-9=58$ groups).

\end{proof}

\begin{remark}
Recall that as we mentioned before, in the proof of Theorem \ref{thm:solvableless2000}, we need to exclude 67 groups. As we see from the proof of Theorem \ref{thm:solvableless2000} (especially, the proof of Lemma \ref{lem:[480,257]}), $[480,257]\cong \SL(2,3)\rt C_4$ is  very hard (probably the hardest one!) to exclude. Notice that the polynomial in Example (17) in Example \ref{mainex} and  the polynomial in the equation (\ref{eq:[480,257]}) are quite similar. Mysteriously, the first one is smooth, but the second one is singular. There might be some deep reason behind this phenomenon.

We also point out that the proofs of Lemma \ref{lem:[96,67]} and Lemma \ref{lem:[480,257]} contain useful strategies to find explicit examples of smooth hypersurfaces with expected group actions.

\end{remark}

\subsection{Other cases}

\begin{theorem}\label{5X5X16unique}
Let $G< \Aut (X)$. If $G\cong C_{16}\times C_5^2$, then, up to change of coordinates, $X$ is the Example (4) in Example \ref{mainex}.
\end{theorem}

\begin{proof}
Suppose $G\cong C_{16}\times C_5^2$.

By Lemma \ref{10or50liftable} and Theorem \ref{liftable}, $G$ has an $F$-lifting, say $\widetilde{G}$. 

Let $A_1=\Diag (\xi_{16},\xi_{16}^{-4},1,1,1), A_2=\Diag(1,1,1,\xi_5,1)$ and $A_3=\Diag(1,1,1,1,\xi_5)$.

Using linear change of coordinates if necessary, we may assume $\widetilde{G}=\langle [A_1],[A_2],[A_3]\rangle $. Then by computing the invariant monomials of $\widetilde{G}$, we have $$F=a x_1^4x_2+bx_2^4x_3+cx_3^5+dx_4^5+ex_5^5\, ,$$

where $a,b,c,d,e$ are nonzero complex numbers. Clearly, adjusting the coordinates by nonzero multiples if necessary, we may assume $a=b=c=d=e=1$. Then $X$ is just the Example (4) in Example \ref{mainex}.
\end{proof}

\begin{theorem}\label{5X5X4X4unique}
Let $G< \Aut (X)$. If $G\cong C_{4}^2\times C_5^2$, then, up to change of coordinates, $X$ is the Example (3) in Example \ref{mainex}.
\end{theorem}

\begin{proof}
Similar to Theorem \ref{5X5X16unique}.
\end{proof}

To prove Theorem \ref{orderdivides2^63^25^2} below, we need the following purely group theoretical results:

\begin{theorem}\label{thm:abeliansylowandnormalp}
Let $p$ be a prime. Let $G$ be a finite group of order $p^{\alpha}n$, $(p,n)=1$. If $G_p$ is abelian and the order of the automorphism group of $G_p$ is coprime to $n$, then $G$ has a normal $p$-complement.
\end{theorem}

\begin{proof}
Since $G_p$ is abelian and the order of the automorphism group of $G_p$ is coprime to $n$, then $G_p$ is in the center of its normalizer $N_{G}(G_p)$. Then by Theorem \ref{thm:Burnsidenormalp}, $G$ must have a normal $p$-complement. 
\end{proof}

\begin{theorem}\label{thm:autofabelian}
(\cite[Theorem~4.1]{HR07}) Let $p$ be a prime. Let $H$ be the abelian $p$-group $C_{p^{e_1}}\times \cdots\times C_{p^{e_n}}$, $1\leq e_1\leq \cdots \leq e_n$. Define the following $2n$ numbers: $$d_k=\Max\{r | e_r=e_k\}, \; c_k=\Min\{r|e_r=e_k\}$$

then one has in particular $d_k\geq k$, $c_k\leq k$, and 

$$|\Aut(H)|=\prod_{k=1}^{n}(p^{d_k}-p^{k-1})\prod_{j=1}^{n}(p^{e_j})^{n-d_j}\prod_{i=1}^{n}(p^{e_{i-1}})^{n-c_i+1}.$$

\end{theorem}

\begin{proposition}\label{pp:normal2complement}
Let $G$ be a finite group. Suppose $$G_2\cong C_{2^{e_1}}\times \cdots \times C_{2^{e_n}},\; 1\leq e_1 <\cdots <e_n.$$ 

Then $G$ has a normal $2$-complement and $G$ is solvable.
\end{proposition}

\begin{proof}
By Theorem \ref{thm:autofabelian}, the order of automorphism group of $G_2$ is a power of $2$. Then by Theorem \ref{thm:abeliansylowandnormalp}, $G$ has a normal $2$-complement, say $N$. By  Feit-Thompson Theorem \cite{FT63}, $N$ is solvable. Then $G$ is solvable since both $N$ and $G/N$ are solvable.
\end{proof}

\begin{theorem}\label{orderdivides2^63^25^2}
Let $G$ be a subgroup of $\Aut(X)$ whose order divides $2^63^25^2$. Then $G$ is isomorphic to a subgroup of one of the groups appearing in the Example (1)-(22) in Example \ref{mainex}.
\end{theorem}

\begin{proof}
First, we assume $G$ is solvable.

If $|G|\leq 2000$, then we are done by Theorem \ref{thm:solvableless2000}.

If $|G|> 2000$, then  $|G|=2400,2880,3600,4800,7200$, or $14400$. 

Suppose $|G|=2400$. Since $G$ is solvable, $G$ has a subgroup, say $H$, of order $800$ by Theorem \ref{thm:solvablesylow}. By Theorem \ref{thm:solvableless2000}, $H$ contains either $C_{16}\times C_5^2$ or $C_{4}^2\times C_5^2$. However, then by Theorem \ref{5X5X16unique}, Theorem \ref{5X5X4X4unique} and Theorem \ref{thm:Aut(X)1-16}, such $G$ does not exist.

Suppose $|G|=2880$. Since $G$ is solvable, $G$ has a subgroup of order $2^63^2=576$, a contradiction to Theorem \ref{thm:solvableless2000}.

Suppose $|G|=3600=2^43^25^2$. Since $G$ is solvable, $G$ has a subgroup of order $144=2^43^2$, a contradiction to Theorem \ref{thm:solvableless2000}.

Suppose $|G|=4800=2^63^15^2$. Since $G$ is solvable, $G$ has a subgroup of order $2^65^2=1600$, a contradiction to Theorem \ref{thm:solvableless2000}.

Suppose  $|G|=7200=2^53^25^2$ or $14400=2^63^25^2$. We get a contradiction similar to previous cases.

Next, we assume $G$ is non-solvable. 

By Proposition \ref{pp:normal2complement}, any finite group whose Sylow $2$-subgroups are isomorphic to $C_{2^m}\times C_{2^n}$, where $m\neq n$, must be solvable. By Burnside\rq{}s $p^aq^b$ theorem, $a_3> 0$ and $a_5 >0$. Therefore, $|G|=2^{a_2}3^{a_3}5^{a_5}$, where $2\leq a_2\leq 5$, $1\leq a_3\leq 2$, and $1\leq a_5\leq 2$. So $|G|$ has 16 possibilities. We will do case by case checking according to the order $|G|$:

1)  $|G|=2^2\cdot3\cdot5$: By classification (using GAP), there is only one non-solvable group of order 60: alternating group $A_5$, which is clearly a subgroup of $\Aut(X)$.   

2) $|G|=2^2\cdot3^2\cdot5$: By classification, there is only one non-solvable group of order 180: $A_5\times C_3$, which is not a subgroup of $\Aut(X)$ since it contains a subgroup isomorphic to $C_2^2\times C_3$.

3) $|G|=2^2\cdot3\cdot5^2$: By classification, there is only one non-solvable group of order 300:  $A_5\times C_5$, which could be a subgroup of $\Aut(X)$ (cf. the Example (21) in Example \ref{mainex}).

4)  $|G|=2^2\cdot3^2\cdot5^2$: By classification, there is only one non-solvable group of order 900: $A_5\times C_{15}$, which is not a subgroup of $\Aut(X)$ since it contains a subgroup isomorphic to $C_2^2\times C_3$.

5) $|G|=2^3\cdot3\cdot5=120$: By classification, there are three non-solvable groups of order 120, $A_6\times C_2$, $S_5$, and $\SL(2,5)$. By sub-test $A_6\times C_2$ is excluded (cf. Lemma \ref{lem:[8,5]}). We know the symmetric group $S_5$ is a subgroup of $\Aut(X)$ in  examples in Example \ref{mainex}. So we are reduced to exclude the group $\SmallGroup(120,5)\cong \SL(2,5).$

\begin{lemma}\label{traceof4}
 Let $A\in \GL(5,\mathbb{C})$ of order $4$. Suppose $[A]\in \Aut(X), \Ord([A])=4$, and $A(F)=F$. Then ${\rm tr}(A)\neq -1$.
   \end{lemma} 
   
   \begin{proof}
   Similar to Lemma \ref{lem:3times2}.
   \end{proof}

\begin{lemma}\label{lem:exclude SL(2,5)}
The group $\SL(2,5)$ is not a subgroup of $\Aut(X)$.
\end{lemma}

\begin{proof}
Assume to the contrary that $G< \Aut(X)$ and $G\cong \SL(2,5)$.

Since $|G|=2^3\cdot 3\cdot 5$ and $G$ has a subgroup of order $10$ (could be checked by GAP), by Lemma \ref{10or50liftable}, $G$ has an $F-$lifting, say $\widetilde{G}$. The group $\widetilde{G}$ corresponds to five dimension faithful linear representation of $\SL(2,5)$, say $\rho$. By the character table (see Figure \ref{[120,5]}). $\SL(2,5)$ has nine different characters $\rm X.1-X.9$. The representation $\rho$ can not be a 5-dimensional irreducible representation (i.e., $\rho=\rm X.8$ ) as $\rm X.8$ is not a faithful representation (note that $\rm X.8 (2a)=5$, i.e.,  the character  $\rm X.8$ takes value 5 at the conjugacy class $2a$).

$\rho$ can not be of type $4\oplus 1$. Indeed:

1) $X.7\oplus X.1$ is impossible by the value of $2a$; and 

2) $X.6\oplus X.1$ is impossible by the value of $2a$.

$\rho$ cannot be of type $3\oplus 1\oplus 1$ as this is not faithful, again because of values of $2a$.

$\rho$ cannot be of type $3\oplus 2$  by the value of $4a$ and Lemma \ref{traceof4}.

$\rho$ cannot be of type $2\oplus 2\oplus 1$ by the value of $2a$ (trace of order 2 matrices can not be negative).

$\rho$ can not be type $2\oplus 1\oplus 1\oplus 1$ by value of $6a$ and Lemma \ref{lem:3times2}.

$\rho$ can not be type $ 1\oplus 1\oplus 1 \oplus 1\oplus 1$ since $\SL(2,5)$ is not abelian. 

Therefore,  $\rho$ with required properties does not exist. So we are done.
\begin{figure}[htbp]
\begin{center}

\includegraphics[width=20cm, height=20cm]{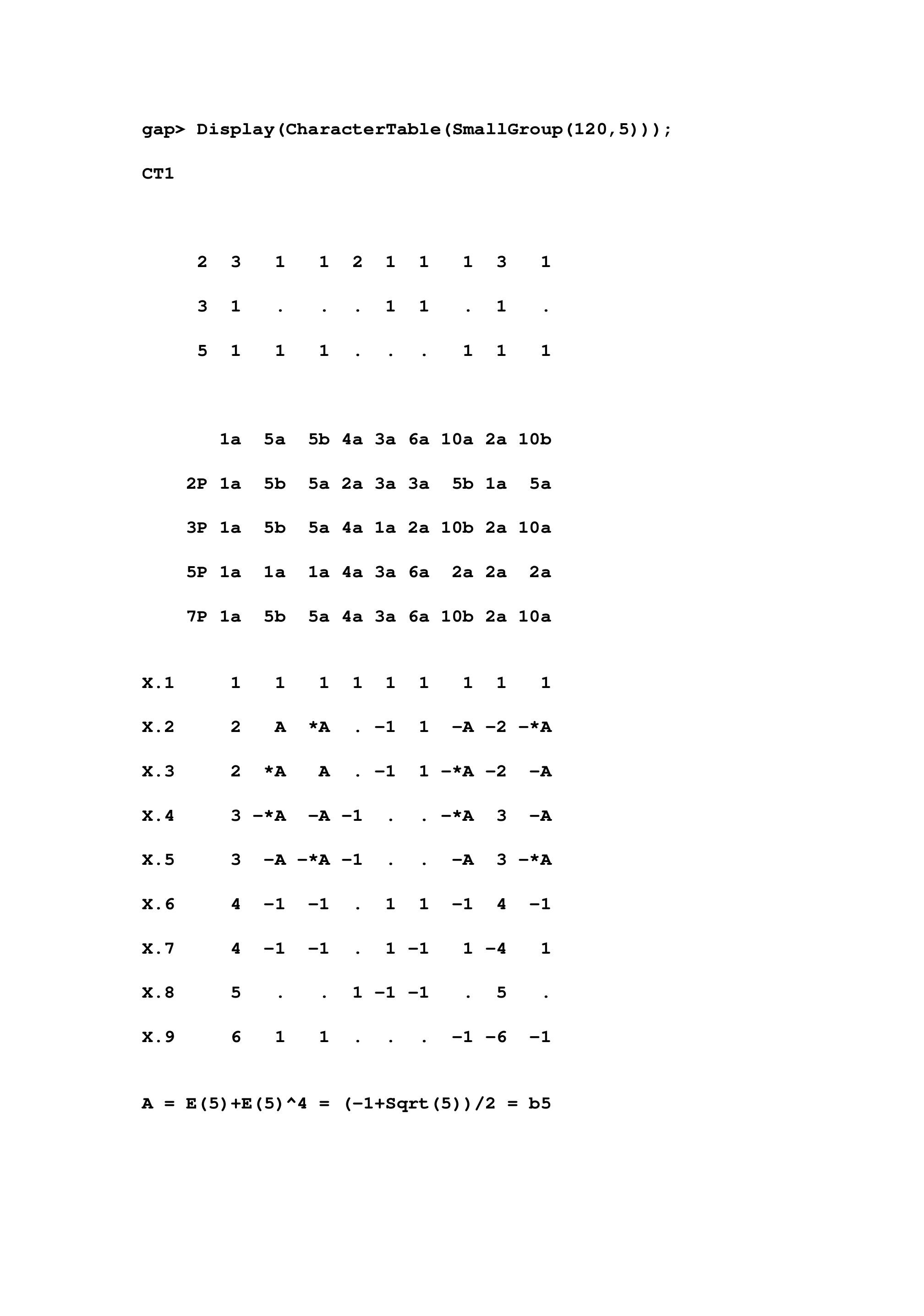}

\caption{Character table of $\SL(2,5)$}
\label{[120,5]}
\end{center}
\end{figure}
\end{proof}

6) $|G|=2^3\cdot3^2\cdot5=360$: By classification, there are six non-solvable groups of order 360. By sub-test, five of them are excluded and only $A_6$ survives. So we are reduced to exclude $A_6$.

\begin{lemma}
The alternating group $A_6$ is not a subgroup of $\Aut(X)$.
\end{lemma}

\begin{proof}

Assume to the contrary that $G< \Aut(X)$ and $G\cong A_6$. Notice that $A_6$ has a subgroup of order $10$, and then $G_5$ is $F-$liftable by Lemma \ref{10or50liftable}.

Then by Theorem \ref{liftable}, $G$ has an $F-$lifting, say  $\widetilde{G}$. So $\widetilde{G}$ corresponds to a five dimensional faithful linear representation of $A_6$. By linear representation theory of $A_6$,  the group $\widetilde{G}$ contains a matrix of order 4 whose trace is $-1$, a contradiction to Lemma \ref{traceof4}.
\end{proof}

7) $|G|=600,1800,240,720,1200,480,1440$: These orders are all less than 2000, so the methods to exclude groups are essentially the same as the methods for cases (1)-(6), and we omit the details.

8) $|G|=2^5\cdot3\cdot 5^2=2400$: Because in GAP library, groups of order 2400 are not approached by $\SmallGroup(-,-)$ function, we need to use slightly different methods.

\begin{lemma}
No non-solvable group of order $2400$ is a subgroups of $\Aut(X)$.
\end{lemma}

\begin{proof}
Assume to the contrary that $G< \Aut(X)$, and $G$ is non-solvable of order $2400$.

Let $N$ be a maximal proper normal subgroup of $G$. Then the quotient group $G/N$ must be a nontrivial simple group, and by classification of finite simple groups, $G/N\cong C_2,C_3,C_5$ or $A_5$.

If $G/N \cong C_2,C_3$, or  $C_5$,  then $N$ must be a non-solvable subgroup of $\Aut(X)$ of order $1200,800$, or $480$, which is impossible by previous results.

If $G/N\cong A_5$, then $N$ has order $40$, then $G$ has a subgroup of order $400$. On the other hand, by  the classification (which we are done before) of subgroups of $\Aut(X)$ of order $400$, we can explicitly compute (using the differential method in Section \ref{ss:differentialmethod}) $\Aut(X)$ if $\Aut(X)$ has a subgroup of order $400$. In particular, $\Aut(X)$ can not contain a non-solvable subgroup of order $2400$ when $400$ divides $|\Aut(X)|$.  

Therefore, $G/N\cong A_5$ is also impossible.
\end{proof}

9) $|G|=3600$, or $7200$: Impossible by similar arguments above.

Therefore, the theorem is proved.

\end{proof}

\section{Proof of main Theorem}\label{ss:proofmainthm}

In this section, we prove our main Theorem (Theorem \ref{thm:Main}).

Let $G< \Aut(X)$. Then, by our classification of Sylow subgroups of subgroups of $\Aut(X)$ done in previous sections, it follows that $$|G|=2^{a_2}3^{a_3}5^{a_5}13^{a_{13}}17^{a_{17}}41^{a_{41}}\, ,$$ where $0\leq a_2\leq 8$, $ 0\leq a_3\leq 2$, $0\leq a_5\leq 5$, $0\leq a_{13}\leq 1$,  $0\leq a_{17}\leq 1$,  $0\leq a_{41}\leq 1$.

If $a_{13}, a_{17}$, or $a_{41}$ is not zero, by Theorems \ref{thm:containorder13}, \ref{thm:containorder17} or \ref{thm:containorder41}, $G$ is isomorphic to a subgroup of one of the 22 groups in Example \ref{mainex}.

In the rest of the proof we assume $a_{13}=a_{17}=a_{41}=0$.

If $a_{5}\ge 3$, then by Theorem \ref{thm:greaterthan125}, $G$ is isomorphic to a subgroup of one of the groups in Example \ref{mainex}.

Then  we may furthermore assume $a_5< 3$, i.e.,  $$|G|=2^{a_2}3^{a_3}5^{a_5},$$ where $0\leq a_2\leq 8$, $ 0\leq a_3\leq 2$, $0\leq a_5\leq 2$.

If $a_2=7$ or $8$, then by Theorem \ref{thm:C128}, $G$ is isomorphic to a subgroup of one of the groups in Example \ref{mainex}.

If $0\leq a_2\leq 6$, then we may just apply Theorem \ref{orderdivides2^63^25^2}.

Theorem \ref{thm:Main} is thus proved.

\section{Application-Gorenstein automorphism groups}\label{ss:goren}

{\it In this section, $X$ is a smooth quintic threefold defined by $F$.} In this section, we study the {\it Gorenstein} automorphism group of $X$.

\begin{definition}
Let $Y$ be a Calabi-Yau threefold. Let $\omega_Y$ be a nonzero holomorphic 3-form on $Y$. An automorphism of $Y$ or an action of a group on $Y$ is called {\it Gorenstein } if it fixes $\omega_Y$.
\end{definition}

\begin{lemma}\label{lem:mukailemma}
Let $A\in \GL(5,\C)$. Suppose $[A]\in \Aut(X)$. Then the automorphism $[A]$ of $X$ is Gorenstein if and only if $A(F)=\Det(A) F$.
\end{lemma}

\begin{proof}
See, for instance, \cite[Lemma~2.1]{Mu88}.
\end{proof}

\begin{lemma}\label{lem:Gorensteingroup}
Let $H<\Aut(X)$. Suppose $H$ has an $F$-lifting, say $\wt{H}$. Then $H$ is Gorenstein if and only if $\wt{H}\subset \SL(5,\C)$.
\end{lemma}

\begin{proof}
By definition of $F$-lifting, for all $A$ in $\wt{H}$, we have $A(F)=F$. Then just apply Lemma \ref{lem:mukailemma}.
\end{proof}

Let $G_i\subset\PGL(5,\C)$ and  $X_i$ ($i=1,2,...,22$) be  the finite group and the smooth quintic threefold defined in Example (i) in Example \ref{mainex}. Then by Lemma \ref{lem:mukailemma}, we can easily compute the Gorenstein subgroup (i.e., the subgroup consists of the Gorenstein automorphism of $X_i$), say $H_i$,  of $G_i\cong \Aut(X_i)$:

Example (1): $H_1$=$\la [A_4A_5^4], [A_4A_6^4], [A_4A_7^4], [A_2],[A_3]\ra\cong C_5^3\rt A_5$, and $|H_1|=2^2\cdot 3\cdot 5^4=7500$. (The matrices $A_i$ here are the same as those in Example (1) in Example \ref{mainex}. We use similar convention below.)

\medskip

Example (2): $H_2$=$\la [A_1^2A_2], [A_3], [A_4A_5^4], [A_4A_6^4]\ra\cong (C_5^2\rt C_3)\rt C_2$, and $|H_2|=2\cdot 3\cdot 5^2=150$.
\medskip

Example (3): $H_3$=$\la [A_1A_2^3], [A_3A_4^4], [A_5]\ra \cong D_{40}$, and $|H_3|=2^3\cdot 5=40$.
\medskip

Example (4): $H_4$=$\la [A_1^8A_4], [A_2A_3^4]\ra \cong D_{10}$, and $|H_4|=2\cdot 5=10$.

\medskip

Example (5): $H_5$=$\la [A_1A_7],[A_2], [A_3A_4^4], [A_3A_5^4],[A_6]\ra \cong (C_3\times (C_5^2\rt C_3))\rt C_2$, and $|H_5|=2\cdot 3^2\cdot 5^2=450$.

\medskip

Example (6): $H_6$=$\la [A_1^4A_2^3]\ra \cong C_{4}$, and $|H_6|=2^2=4$.
\medskip

Example (7): $H_7$=the trivial group.
\medskip

Example (8): $H_8$=$\la [A_1^2A_3], [A_2^2A_5], [A_4]\ra \cong C_5\times S_3$, and $|H_8|=2\cdot 3\cdot 5=30$.
\medskip

Example (9): $H_9$=$\la [A_1], [A_2], [A_4A_5^4]\ra \cong C_5\times(C_{13}\rt C_3)$, and $|H_9|=3\cdot 5\cdot 13=195$.
\medskip

Example (10): $H_{10}$=$\la [A_1^8A_2], [A_3]\ra \cong S_3$, and $|H_{10}|=2\cdot3=6$.

\medskip

Example (11): $H_{11}$=the trivial group.
\medskip

Example (12: $H_{12}$=$\la [A_3], [A_4]\ra \cong C_{13}\rt C_3$, and $|H_{12}|=3\cdot 13=39$.
\medskip

Example (13): $H_{13}$=$\la [A_1], [A_2], [A_3^2]\ra \cong C_3\times D_{34}$, and $|H_{13}|=2\cdot 3\cdot 17=102$.
\medskip

Example (14): $H_{14}$=$\la [A_1], [A_3], [A_2A_4^4],[A_5A_6],[A_5A_7]\ra \cong (C_5\times (C_3^2\rt C_2))\rt C_2$, and $|H_{14}|=2^2\cdot 3^2 \cdot 5=180$.

\medskip

Example (15): $H_{15}$=$\la [A_1], [A_3]\ra \cong C_{41}\rt C_5$, and $|H_{15}|=5\cdot 41=205$.

\medskip

Example (16): $H_{16}$=$\la [A_1], [A_2], [A_4]\ra \cong C_3\times (C_{13}\rt C_3)$, and $|H_{16}|=3^2\cdot 13=117$.

\medskip

Example (17): $H_{17}$=$\la [A_1], [A_2], [A_3]\ra \cong \GL(2,3)$, and $|H_{17}|=2^4\cdot 3=48$.

\medskip

Example (18): $H_{18}$=$\la [A_1], [A_2], [A_3]\ra \cong \SL(2,3)$, and $|H_{18}|=2^3\cdot 3=24$.

\medskip

Example (19): $H_{19}$=$\la [A_2], [A_4]\ra \cong C_{12}$, and $|H_{19}|=2^2\cdot 3=12$.
\medskip

Example (20):  $H_{20}$=$\la [A_2],[A_3], [A_4]\ra \cong D_{24}$, and $|H_{20}|=2^3\cdot 3=24$.
\medskip

Example (21):  $H_{21}$=$\la [A_2], [A_3]\ra \cong A_{5}$, and $|H_{21}|=2^2\cdot 3\cdot 5=60$.
\medskip

Example (22):  $H_{22}$=$\la [A_1^{16}A_2]\ra \cong C_2$, and $|H_{22}|=2$.

\medskip

It turns out that the above examples cover almost all maximal (with respect to inclusions) finite groups which can have an effective Gorenstein group action on a smooth quintic threefold:

\begin{theorem}\label{thm:goren}

Let $H$ be a finite group. If $H$ has an effective Gorenstein group action on a smooth quintic threefold, then $H$ is isomorphic to a subgroup of one of the following 11 groups: $ C_5^3\rt A_5$, $ D_{40}$, $(C_3\times (C_5^2\rt C_3))\rt C_2$, $C_5\times(C_{13}\rt C_3)$, $ C_3\times D_{34}$, $(C_5\times (C_3^2\rt C_2))\rt C_2$, $C_{41}\rt C_5$, $C_3\times (C_{13}\rt C_3)$, $\GL(2,3)$, $D_{24}$, which are isomorphic to $H_1$,$H_3$,$H_5$,$H_9$,$H_{13}$,$H_{14}$,$H_{15}$,$H_{16}$,$H_{17}$,$H_{20}$ defined above, and $C_4\ti C_2$.
\end{theorem}

\begin{proof}

First, if $X: x_1^4x_4+x_2^4x_5+x_3^4x_4+x_4^5+x_5^5+x_1x_2x_3^3=0$, then $A_1:=\Diag(\xi_4,\xi_4,-1,1,1)$ and $A_2:=\Diag(1,-1,-1,1,1)$ act on $X$ Gorensteinly. Therefore, $C_4\ti C_2$ has an effective Gorenstein group action on a smooth quintic threefold.

Since the main ideas and strategies of the rest of the proof already appear in the previous sections, we only sketch it here.

Suppose $H$ has an effective Gorenstein group action on $X$. We identify $H$ with the corresponding subgroup of $\PGL(5,\C)$.

Of course, $H$ must be isomorphic to a subgroup of the 22 groups in Example \ref{mainex} by Theorem \ref{thm:Main}.

If $|H|$ is divided by 128, 125,41,17 or 13, then,  by using the results in previous sections, we can easily determine $H$ (more precisely, the matrices generate $H$) and the defining equation $F$ of $X$. Then by using Lemma \ref{lem:mukailemma} or \ref{lem:Gorensteingroup}, the theorem can be proved in these cases.

Now it remains to treat the cases where  $$|H|=2^{a_2}3^{a_3}5^{a_5}, 0\leq a_2\leq 6, 0\leq a_3\leq 2, 0\leq a_5\leq 2\, .$$ In these cases, like in Section \ref{ss:Sylow2} and Section \ref{ss:2^63^25^2}, we use GAP and the method explained in Remark \ref{rmk:stepstoexcludegroups} (how to exclude groups). In fact, by sub-test, we are reduced to exclude the following 21 groups (GAP IDs and their structure description): $[12,1]\cong C_3\rt C_4$, $[16,1]\cong C_{16}$,  $[16,2]\cong C_{4}\times C_4$,  $[16,5]\cong C_{8}\times C_2$,  $[16,6]\cong C_{8}\rt C_2$,  $[16,7]\cong D_{16}$,  $[16,9]\cong Q_{16}$ (generalized quaternion group),  $[16,13]\cong (C_{4}\times C_2)\rt C_2$,  $[20,1]\cong C_{5}\rt C_4$,  $[20,3]\cong C_{5}\rt C_4$,  $[20,5]\cong C_{10}\times C_2$,  $[24,1]\cong C_{3}\rt C_8$,  $[24,2]\cong C_{24}$,  $[24,12]\cong S_{4}$,  $[30,4]\cong C_{30}$,  $[36,9]\cong (C_{3}\times C_3)\rt C_4$,  $[40,1]\cong C_{5}\rt C_8$,  $[40,2]\cong C_{40}$,  $[40,11]\cong C_{5}\times Q_8$,  $[50,5]\cong C_{10}\times C_5$,  $[225,6]\cong C_{3}^2\times C_5^2$.

Notice that by results before about $F$-liftability we can easily show that if $H$ is isomorphic to one of the above 21 groups then $H$ is $F$-liftable. Then the task of excluding these groups is essentially reduced to show  that they can not have a five dimension faithful linear representation into $\SL(5,\C)$ which leaves the smooth polynomial $F$ invariant.  To give an example, we show how to exclude $D_{16}$ here:

\begin{lemma}
The group $D_{16}$ does not admit a Gorenstein action on a smooth quintic threefold.
\end{lemma}

\begin{proof}
Assume to the contrary that $H\cong D_{16}$ has a Gorenstein action on a smooth quintic threefold. By Theorem \ref{liftable}, $H$ has an $F$-lifting, say $\wt{H}$. Then there exist matrices $A_1,A_2$ in $\GL(5,\C)$ such that $\wt{H}=\la A_1,A_2\ra$, $\la A_1\ra \cong C_8$, $\la A_2\ra \cong C_2$, and $A_2A_1A_2^{-1}=A_1^{-1}$. By Lemma \ref{lem:mukailemma}, $\Det(A_1)=1$ as $A_1(F)=F$.

Then by Lemma \ref{lem:o8}, we may assume $A_1$=$\Diag(\xi_8,-1,1,1,\xi_8^3)$ or $\Diag(\xi_8,-1,1,\xi_8^5,\xi_8^6)$. Then $A_1$ and $A_1^{-1}$ have different sets of eigenvalues, a contradiction to $A_2A_1A_2^{-1}=A_1^{-1}$. 

Therefore, $D_{16}$ is excluded.

\end{proof}

More details about how to exclude other groups can be found on the website \cite{Yu}.

\end{proof}

\medskip

     \noindent
  {\bf Acknowledgement.}  \\[.2cm]
The second author would like to thank Max Planck Institute for Mathematics in Bonn for support and hospitality during his visit. He also would like to thank Professor Takayuki Hibi for financial support during his stay at Osaka University. It is our great pleasure to dedicate this paper to Professor  Shigeru Mukai. Our paper is in fact much inspired by his seminal work \cite{Mu88}.

\medskip


\end{document}